\documentclass[10pt,reqno]{amsart}

\usepackage{xr}
\usepackage{booktabs,array,ragged2e}
\usepackage{pdflscape}
\usepackage{amssymb}
\usepackage{graphicx}
\usepackage{epstopdf}
\usepackage[headings]{fullpage}
\usepackage{url}
\usepackage{color}
\usepackage[usenames,dvipsnames,svgnames,table]{xcolor}
\usepackage{equipert}
\usepackage[square,numbers,sort&compress]{natbib}

\usepackage{subfigure}
\usepackage{multirow}

\usepackage{tikz}
\usetikzlibrary{calc}

\usepackage{pgfplots} 

\usetikzlibrary{positioning,calc}

\def\visible<#1>{}  

\usepackage[utf8]{inputenc}
\usepackage{amsfonts}
\usepackage{amsmath}
\usepackage{latexsym}
\usepackage{color}
\usepackage{subfigure}
\usepackage{enumerate}

\usepackage{ifpdf}

\usepackage{hyperref}  

\let\ParaSign=\S

\DeclareMathOperator    \aff                    {aff}

\DeclareMathOperator    \cl                     {cl}
\DeclareMathOperator    \conv           {conv}

\DeclareMathOperator    \intr                   {int}

\DeclareMathOperator    \relint         {rel\,int}

\DeclareMathOperator    \verts          {vert}

\newcommand{\old}[1]{{}}

\newcommand{\bb}{\mathbb}

\newcommand{\R}{\bb R}
\newcommand{\Q}{\bb Q}
\newcommand{\Z}{\bb Z}
\newcommand{\N}{\bb N}

\newcommand\st{\mid}
\newcommand\bigst{\mathrel{\big|}}

\newcommand\biggst{\mathrel{\bigg|}}

\def\ve#1{\mathchoice{\mbox{\boldmath$\displaystyle\bf#1$}}
{\mbox{\boldmath$\textstyle\bf#1$}}
{\mbox{\boldmath$\scriptstyle\bf#1$}}
{\mbox{\boldmath$\scriptscriptstyle\bf#1$}}}

\newcommand{\bpi}{\bar \pi}

\newcommand{\setcond}[2]{\left\{\, #1 \,\st\, #2 \,\right\}}
\newcommand{\I}{\mathcal{P}}  

\renewcommand{\P}{\mathcal{P}}

\renewcommand{\S}{\mathcal{S}}

\newcommand{\T}{\mathcal T}

\newcommand{\merge}{\mathbin{\lozenge}}
\newcommand{\mergeProj}{\mathbin{\lozenge_n^1}}

\newcommand\FundaTriangleLower{{}^{}_{\ve0}\EquiTriangleLower}  
\newcommand\FundaTriangleUpper{{}^{}_{\ve0}\EquiTriangleUpper}  

\newcommand{\bPizero}{\bar \Pi^E_{\mathrm{zero}(\T)}(\R^k, \Z^k)}
\newcommand{\bPizeroOne}{\bar \Pi^E_{\mathrm{zero}(\T)}(\R, \Z)}

\newcommand{\rx}{{\ve r}}
\newcommand{\x}{{\ve x}}
\newcommand{\y}{{\ve y}}

\renewcommand{\v}{{\ve v}}
\newcommand{\g}{{\ve g}}

\renewcommand{\u}{{\ve u}}
\renewcommand{\a}{{\ve a}}
\newcommand{\f}{{\ve f}}
\newcommand{\0}{{\ve 0}}

\renewcommand{\t}{{\ve t}}
\newcommand{\w}{{\ve w}}

\newcommand{\cve}{{\ve c}}

\newcommand{\rr}{{\ve r}}
\newcommand{\gp}{{\ve {\bar g}}}
\newcommand{\gt}{{\ve {\tilde g}}}
\newcommand{\gs}{{\ve  g}}

\newcommand{\Ball}{B}
\newcommand{\B}{B}

\def\st{\mid}

\newenvironment{psmallmatrixbig}{\bigl(\smallmatrix}{\endsmallmatrix\bigr)}
\newcommand\InlineFrac[2]{#1/#2}  
\newcommand\ColVec[3][\relax]
{
  \ifx#1\relax
  \bgroup\let\frac=\InlineFrac\begin{psmallmatrixbig}#2\vphantom{/}\\#3\vphantom{/}\end{psmallmatrixbig}\egroup
  \else
  \bgroup\let\frac=\InlineFrac\begin{psmallmatrixbig}\ifx#200\else#2/#1\fi\\\ifx#300\else#3/#1\fi\end{psmallmatrixbig}\egroup
  \fi
}

\newtheorem{theorem}{Theorem}[section]

\makeatletter
\newcommand\MkNewTheorem[2]{%
  \newtheorem{#1}{#2}
  \expandafter\def\csname c@#1\endcsname{\c@theorem}
  \expandafter\def\csname p@#1\endcsname{\p@theorem}
  \expandafter\def\csname the#1\endcsname{\thetheorem}
  \expandafter\def\csname #1name\endcsname{#2}
}

\MkNewTheorem{corollary}{Corollary}
\MkNewTheorem{lemma}{Lemma}
\MkNewTheorem{proposition}{Proposition}
\MkNewTheorem{prop}{Proposition}
\MkNewTheorem{claim}{Claim}
\MkNewTheorem{observation}{Observation}
\MkNewTheorem{obs}{Observation}
\MkNewTheorem{conjecture}{Conjecture}
\MkNewTheorem{openquestion}{Open question}

\theoremstyle{definition}
\MkNewTheorem{example}{Example}
\MkNewTheorem{exercise}{Exercise}
\MkNewTheorem{notation}{Notation}
\MkNewTheorem{assumption}{Assumption}
\MkNewTheorem{definition}{Definition}
\MkNewTheorem{remark}{Remark}
\MkNewTheorem{goal}{Goal}

\makeatletter
\let\savemathbb=\mathbb
\let\OurMathBbAux=\mathbb
\DeclareRobustCommand\OurMathBb{\OurMathBbAux}
\let\mathbb=\OurMathBb
\let\bfseries=\undefined
\DeclareRobustCommand\bfseries
{\not@math@alphabet\bfseries\mathbf
  \boldmath\fontseries\bfdefault\selectfont\let\OurMathBbAux=\mathbf}
\def\@thm#1#2#3{%
  \ifhmode\unskip\unskip\par\fi
  \normalfont
  \trivlist
  \let\thmheadnl\relax
  \let\thm@swap\@gobble
  \thm@notefont{\fontseries\mddefault\upshape\unboldmath\let\OurMathBbAux=\savemathbb}
  \thm@headpunct{.}
  \thm@headsep 5\p@ plus\p@ minus\p@\relax
  \thm@space@setup
  #1
  \@topsep \thm@preskip               
  \@topsepadd \thm@postskip           
  \def\@tempa{#2}\ifx\@empty\@tempa
    \def\@tempa{\@oparg{\@begintheorem{#3}{}}[]}%
  \else
    \refstepcounter{#2}%
    \def\@tempa{\@oparg{\@begintheorem{#3}{\csname the#2\endcsname}}[]}%
  \fi
  \@tempa
}
\makeatother

\makeatletter
\renewcommand{\pod}[1]
{\allowbreak\mathchoice{\mkern18mu}{\mkern8mu}{\mkern8mu}{\mkern8mu}(#1)}
\makeatother

\chardef\Myunderscore=`\_
\newcommand\underscore{\Myunderscore\allowbreak}

\newcommand\githubsearchurl{https://github.com/mkoeppe/infinite-group-relaxation-code/search}
\pgfkeyssetvalue{/sagefunc/california_ip}{\href{\githubsearchurl?q=\%22def+california_ip(\%22}{\sage{california\underscore{}ip}}}
\pgfkeyssetvalue{/sagefunc/automorphism}{\href{\githubsearchurl?q=\%22def+automorphism(\%22}{\sage{automorphism}}}
\pgfkeyssetvalue{/sagefunc/multiplicative_homomorphism}{\href{\githubsearchurl?q=\%22def+multiplicative_homomorphism(\%22}{\sage{multiplicative\underscore{}homomorphism}}}
\pgfkeyssetvalue{/sagefunc/projected_sequential_merge}{\href{\githubsearchurl?q=\%22def+projected_sequential_merge(\%22}{\sage{projected\underscore{}sequential\underscore{}merge}}}
\pgfkeyssetvalue{/sagefunc/restrict_to_finite_group}{\href{\githubsearchurl?q=\%22def+restrict_to_finite_group(\%22}{\sage{restrict\underscore{}to\underscore{}finite\underscore{}group}}}
\pgfkeyssetvalue{/sagefunc/restrict_to_finite_group_3}{\href{\githubsearchurl?q=\%22def+restrict_to_finite_group_3(\%22}{\sage{restrict\underscore{}to\underscore{}finite\underscore{}group\underscore{}3}}}
\pgfkeyssetvalue{/sagefunc/interpolate_to_infinite_group}{\href{\githubsearchurl?q=\%22def+interpolate_to_infinite_group(\%22}{\sage{interpolate\underscore{}to\underscore{}infinite\underscore{}group}}}
\pgfkeyssetvalue{/sagefunc/two_slope_fill_in}{\href{\githubsearchurl?q=\%22def+two_slope_fill_in(\%22}{\sage{two\underscore{}slope\underscore{}fill\underscore{}in}}}
\pgfkeyssetvalue{/sagefunc/ll_strong_fractional}{\href{\githubsearchurl?q=\%22def+ll_strong_fractional(\%22}{\sage{ll\underscore{}strong\underscore{}fractional}}}
\pgfkeyssetvalue{/sagefunc/hildebrand_2_sided_discont_1_slope_1}{\href{\githubsearchurl?q=\%22def+hildebrand_2_sided_discont_1_slope_1(\%22}{\sage{hildebrand\underscore{}2\underscore{}sided\underscore{}discont\underscore{}1\underscore{}slope\underscore{}1}}}
\pgfkeyssetvalue{/sagefunc/hildebrand_2_sided_discont_2_slope_1}{\href{\githubsearchurl?q=\%22def+hildebrand_2_sided_discont_2_slope_1(\%22}{\sage{hildebrand\underscore{}2\underscore{}sided\underscore{}discont\underscore{}2\underscore{}slope\underscore{}1}}}
\pgfkeyssetvalue{/sagefunc/hildebrand_discont_3_slope_1}{\href{\githubsearchurl?q=\%22def+hildebrand_discont_3_slope_1(\%22}{\sage{hildebrand\underscore{}discont\underscore{}3\underscore{}slope\underscore{}1}}}
\pgfkeyssetvalue{/sagefunc/dr_projected_sequential_merge_3_slope}{\href{\githubsearchurl?q=\%22def+dr_projected_sequential_merge_3_slope(\%22}{\sage{dr\underscore{}projected\underscore{}sequential\underscore{}merge\underscore{}3\underscore{}slope}}}
\pgfkeyssetvalue{/sagefunc/chen_4_slope}{\href{\githubsearchurl?q=\%22def+chen_4_slope(\%22}{\sage{chen\underscore{}4\underscore{}slope}}}
\pgfkeyssetvalue{/sagefunc/gmic}{\href{\githubsearchurl?q=\%22def+gmic(\%22}{\sage{gmic}}}
\pgfkeyssetvalue{/sagefunc/gj_2_slope}{\href{\githubsearchurl?q=\%22def+gj_2_slope(\%22}{\sage{gj\underscore{}2\underscore{}slope}}}
\pgfkeyssetvalue{/sagefunc/gj_2_slope_repeat}{\href{\githubsearchurl?q=\%22def+gj_2_slope_repeat(\%22}{\sage{gj\underscore{}2\underscore{}slope\underscore{}repeat}}}
\pgfkeyssetvalue{/sagefunc/dg_2_step_mir}{\href{\githubsearchurl?q=\%22def+dg_2_step_mir(\%22}{\sage{dg\underscore{}2\underscore{}step\underscore{}mir}}}
\pgfkeyssetvalue{/sagefunc/kf_n_step_mir}{\href{\githubsearchurl?q=\%22def+kf_n_step_mir(\%22}{\sage{kf\underscore{}n\underscore{}step\underscore{}mir}}}
\pgfkeyssetvalue{/sagefunc/gj_forward_3_slope}{\href{\githubsearchurl?q=\%22def+gj_forward_3_slope(\%22}{\sage{gj\underscore{}forward\underscore{}3\underscore{}slope}}}
\pgfkeyssetvalue{/sagefunc/drlm_backward_3_slope}{\href{\githubsearchurl?q=\%22def+drlm_backward_3_slope(\%22}{\sage{drlm\underscore{}backward\underscore{}3\underscore{}slope}}}
\pgfkeyssetvalue{/sagefunc/drlm_2_slope_limit}{\href{\githubsearchurl?q=\%22def+drlm_2_slope_limit(\%22}{\sage{drlm\underscore{}2\underscore{}slope\underscore{}limit}}}
\pgfkeyssetvalue{/sagefunc/drlm_2_slope_limit_1_1}{\href{\githubsearchurl?q=\%22def+drlm_2_slope_limit_1_1(\%22}{\sage{drlm\underscore{}2\underscore{}slope\underscore{}limit\underscore{}1\underscore{}1}}}
\pgfkeyssetvalue{/sagefunc/bhk_irrational}{\href{\githubsearchurl?q=\%22def+bhk_irrational(\%22}{\sage{bhk\underscore{}irrational}}}
\pgfkeyssetvalue{/sagefunc/bccz_counterexample}{\href{\githubsearchurl?q=\%22def+bccz_counterexample(\%22}{\sage{bccz\underscore{}counterexample}}}
\pgfkeyssetvalue{/sagefunc/drlm_3_slope_limit}{\href{\githubsearchurl?q=\%22def+drlm_3_slope_limit(\%22}{\sage{drlm\underscore{}3\underscore{}slope\underscore{}limit}}}
\pgfkeyssetvalue{/sagefunc/dg_2_step_mir_limit}{\href{\githubsearchurl?q=\%22def+dg_2_step_mir_limit(\%22}{\sage{dg\underscore{}2\underscore{}step\underscore{}mir\underscore{}limit}}}
\pgfkeyssetvalue{/sagefunc/rlm_dpl1_extreme_3a}{\href{\githubsearchurl?q=\%22def+rlm_dpl1_extreme_3a(\%22}{\sage{rlm\underscore{}dpl1\underscore{}extreme\underscore{}3a}}}
\pgfkeyssetvalue{/sagefunc/hildebrand_5_slope_22_1}{\href{\githubsearchurl?q=\%22def+hildebrand_5_slope_22_1(\%22}{\sage{hildebrand\underscore{}5\underscore{}slope\underscore{}22\underscore{}1}}}
\pgfkeyssetvalue{/sagefunc/hildebrand_5_slope_24_1}{\href{\githubsearchurl?q=\%22def+hildebrand_5_slope_24_1(\%22}{\sage{hildebrand\underscore{}5\underscore{}slope\underscore{}24\underscore{}1}}}
\pgfkeyssetvalue{/sagefunc/hildebrand_5_slope_28_1}{\href{\githubsearchurl?q=\%22def+hildebrand_5_slope_28_1(\%22}{\sage{hildebrand\underscore{}5\underscore{}slope\underscore{}28\underscore{}1}}}
\pgfkeyssetvalue{/sagefunc/kzh_7_slope_1}{\href{\githubsearchurl?q=\%22def+kzh_7_slope_1(\%22}{\sage{kzh\underscore{}7\underscore{}slope\underscore{}1}}}
\pgfkeyssetvalue{/sagefunc/kzh_28_slope_1}{\href{\githubsearchurl?q=\%22def+kzh_28_slope_1(\%22}{\sage{kzh\underscore{}28\underscore{}slope\underscore{}1}}}
\pgfkeyssetvalue{/sagefunc/extremality_test}{\href{\githubsearchurl?q=\%22def+extremality_test(\%22}{\sage{extremality\underscore{}test}}}
\pgfkeyssetvalue{/sagefunc/plot_2d_diagram}{\href{\githubsearchurl?q=\%22def+plot_2d_diagram(\%22}{\sage{plot\underscore{}2d\underscore{}diagram}}}
\pgfkeyssetvalue{/sagefunc/generate_example_e_for_psi_n}{\href{\githubsearchurl?q=\%22def+generate_example_e_for_psi_n(\%22}{\sage{generate\underscore{}example\underscore{}e\underscore{}for\underscore{}psi\underscore{}n}}}
\pgfkeyssetvalue{/sagefunc/chen_3_slope_not_extreme}{\href{\githubsearchurl?q=\%22def+chen_3_slope_not_extreme(\%22}{\sage{chen\underscore{}3\underscore{}slope\underscore{}not\underscore{}extreme}}}
\pgfkeyssetvalue{/sagefunc/psi_n_in_bccz_counterexample_construction}{\href{\githubsearchurl?q=\%22def+psi_n_in_bccz_counterexample_construction(\%22}{\sage{psi\underscore{}n\underscore{}in\underscore{}bccz\underscore{}counterexample\underscore{}construction}}}
\pgfkeyssetvalue{/sagefunc/gomory_fractional}{\href{\githubsearchurl?q=\%22def+gomory_fractional(\%22}{\sage{gomory\underscore{}fractional}}}
\pgfkeyssetvalue{/sagefunc/not_extreme_1}{\href{\githubsearchurl?q=\%22def+not_extreme_1(\%22}{\sage{not\underscore{}extreme\underscore{}1}}}
\pgfkeyssetvalue{/sagefunc/bhk_irrational_extreme_limit_to_rational_nonextreme}{\href{\githubsearchurl?q=\%22def+bhk_irrational_extreme_limit_to_rational_nonextreme(\%22}{\sage{bhk\underscore{}irrational\underscore{}extreme\underscore{}limit\underscore{}to\underscore{}rational\underscore{}nonextreme}}}
\pgfkeyssetvalue{/sagefunc/zhou_two_sided_discontinuous_cannot_assume_any_continuity}{\href{\githubsearchurl?q=\%22def+zhou_two_sided_discontinuous_cannot_assume_any_continuity(\%22}{\sage{zhou\underscore{}two\underscore{}sided\underscore{}discontinuous\underscore{}cannot\underscore{}assume\underscore{}any\underscore{}continuity}}}
\pgfkeyssetvalue{/sagefunc/drlm_gj_2_slope_extreme_limit_to_nonextreme}{\href{\githubsearchurl?q=\%22def+drlm_gj_2_slope_extreme_limit_to_nonextreme(\%22}{\sage{drlm\underscore{}gj\underscore{}2\underscore{}slope\underscore{}extreme\underscore{}limit\underscore{}to\underscore{}nonextreme}}}
\pgfkeyssetvalue{/sagefunc/not_minimal_2}{\href{\githubsearchurl?q=\%22def+not_minimal_2(\%22}{\sage{not\underscore{}minimal\underscore{}2}}}
\pgfkeyssetvalue{/sagefunc/drlm_not_extreme_1}{\href{\githubsearchurl?q=\%22def+drlm_not_extreme_1(\%22}{\sage{drlm\underscore{}not\underscore{}extreme\underscore{}1}}}
\pgfkeyssetvalue{/sagefunc/kzh_2q_example_1}{\href{\githubsearchurl?q=\%22def+kzh_2q_example_1(\%22}{\sage{kzh\underscore{}2q\underscore{}example\underscore{}1}}}

\DeclareRobustCommand\sage[1]{\texttt{#1}}
\DeclareRobustCommand\sagefunc[1]{\pgfkeys{/sagefunc/#1}}

\newcommand\NEWRESULT{$\clubsuit$}
\newcommand\TheoremNEWRESULT{New result \NEWRESULT}

\title[Infinite Group Problem]{Light on the Infinite Group
  Relaxation}

\author{Amitabh Basu} 
\address{Amitabh Basu: Dept.~of Applied Mathematics and Statistics, The Johns Hopkins University}
\email{basu.amitabh@jhu.edu}

\author{Robert Hildebrand}
\address{Robert Hildebrand: Institute for Operations Research, Dept. of Mathematics, ETH Z\"urich, Switzerland}
\email{robert.hildebrand@ifor.math.ethz.ch}

\author{Matthias K\"oppe}
\address{Matthias K\"oppe: Dept.\ of Mathematics, University of California, Davis}
\email{mkoeppe@math.ucdavis.edu}

\date{$\relax$Revision: 1763 $ - \ $Date: 2015-03-24 13:37:42 -0700 (Tue, 24 Mar 2015) $ $\!\!\!}

\usepackage[normalem]{ulem}

\begin{document}
 \newcommand{\tgreen}[1]{\textsf{\textcolor {ForestGreen} {#1}}}
 \newcommand{\tred}[1]{\texttt{\textcolor {red} {#1}}}
 \newcommand{\tblue}[1]{\textcolor {blue} {#1}}

\maketitle
\begin{abstract}
  This is a survey on the infinite group problem, an infinite-dimensional
  relaxation of integer linear optimization problems introduced by Ralph Gomory
  and Ellis Johnson in their groundbreaking papers titled \emph{Some continuous
    functions related to corner polyhedra I, II} [Math.\@ Programming
  3~(1972), 23--85, 359–389].
  The survey presents the infinite group problem in the modern context of cut
  generating functions.  It focuses on the recent developments, such as algorithms for testing extremality and breakthroughs for the $k$-row problem for general $k\geq 1$ that extend previous work on the single-row and two-row problems. The survey also includes some previously unpublished results; among other
  things, it unveils piecewise linear extreme functions with more than four different slopes.
  An interactive companion program, implemented in
  the open-source computer algebra package Sage, 
  provides an updated compendium of known extreme functions.
\end{abstract}

\thispagestyle{empty}\enlargethispage{7ex}

{\scriptsize
\tableofcontents
}

\clearpage

\section{Introduction}

A recent line of activity in integer programming research is the development
of cutting plane theory for general purpose mixed-integer linear
programs. Although this theory was initiated by Gomory's seminal
work~\cite{MR0102437,Gom60,gomory1963algorithm,gomory1965relation,gom,gomory2007atoms}
in integer programming, in the 1980s this general theory was overshadowed by
the success of cutting planes for specially structured combinatorial
optimization problems such as the TSP, the stable set problem and the knapsack
problem. A re-evaluation of Gomory's cutting planes in the
1990s~\cite{balas96gomory} led to a renewed interest in general purpose
cutting plane theory. A key turning point in the 2000s was the emphasis on the
so-called {\em multi-row cuts}, which hold the promise of making significant
breakthroughs in algorithms for solving large-scale mixed-integer
programs. The last decade has witnessed considerable progress for multi-row
cuts. This recent research is collectively referred to under the label of
{\em cut-generating functions}, a term coined by Cornu\'ejols \cite{conforti2013cut}. 

A central problem and a driving force behind this line of work has been the
so-called {\em infinite group problem} (or \emph{infinite relaxation}), 
introduced by Gomory and Johnson in two seminal papers in 1972~\cite{infinite,infinite2}. In this sense, the infinite group problem was a visionary contribution that anticipated this modern trend in integer programming decades earlier.  To make further progress in the elaborate research program of cut-generating functions, it is imperative to understand the infinite group problem even better. The bulk of Gomory and Johnson's contributions were in the {\em single-row infinite group problem}, and until recently the theory behind the {\em multi-row infinite group problem} was mostly in the dark. With the modern focus on multi-row cuts within cut-generating functions, it is very important to understand the multi-row infinite group problem. The last decade has seen some excellent progress in this question, and this survey attempts to present this story.

\subsection{Cut-generating function pairs} We begin with a quick overview of the cut-generating function approach to unifying cutting plane theory. Let $d \in \N$ and $I$ be a fixed subset of $\{1, \ldots, d\}$. A mixed-integer optimization problem of the form
\begin{equation}
  \label{eq:orig-ip}
  \max \{ \, \ve c \cdot\ve x \st A \ve x = \ve b,\ \ve x \in \R^d_+, \ x_i \in\Z \quad \forall i\in I \,\}
\end{equation}
is first solved by ignoring the integrality constraints and using the simplex algorithm. This leads to a simplex tableau reformulation:
\begin{equation}\label{eq:simplex-tab}
  A_B^{-1} A_N\ve x_N = A_B^{-1} \ve b\, -\, \ve x_B, \quad \ve x_{B\cap I}
  \in \Z_+^{B\cap I},\ \ve x_{B \setminus I} \in \R_+^{B \setminus I},\ \ve x_{N\cap I} \in \Z_+^{N \cap I},\ \ve x_{N \setminus I} \in \R_+^{N \setminus I}
\end{equation}
where the subscripts $B$ and $N$ denote the basic and non-basic parts of the
solution~$\ve x$ and matrix~$A$, respectively.  
The following change of notation will be convenient: let $k = |B|$, $m = |N \setminus I|$, $\ell = |N \cap I|$, let $R$ denote the submatrix of $A^{-1}_BA_N$ indexed by $N \setminus I$, $P$ denote the submatrix of $A^{-1}_BA_N$ indexed by $N \cap I$, and set $\bar{S} = A^{-1}_B\ve b - (\Z_+^{B\cap I}\times \R_+^{B \setminus I})$.  Then in the new notation, we describe system~\eqref{eq:simplex-tab} as
\begin{equation}
	\label{eq:def mixed-int set}
	X_{\bar{S}}(R,P) := \setcond{(\ve s,\ve y) \in \R_+^m \times \Z_+^\ell }{
          R\ve s + P\ve y\in \bar{S} }.
\end{equation}
In the following, we will consider general systems of the form~\eqref{eq:def
  mixed-int set}, where $m, \ell \in \Z_+$ and $k \in \N$, $R \in \R^{k
  \times m}$ and $P \in \R^{k \times \ell}$ are matrices, and $\bar{S}$ is a
closed subset of $\R^k$ such that $\0 \not\in \bar{S}$. Instead of using
the full simplex tableau~\eqref{eq:simplex-tab}, one can as well consider
relaxations of~\eqref{eq:simplex-tab}, for example by taking a subset of the rows only. In the simplest case, one focuses on only one row, i.e., $k=1$. 

We denote the columns of matrices $R$ and $P$ by $\ve r^1,\ldots,\ve r^m$ and $\ve p^1,\ldots, \ve p^{\ell},$ respectively.
Given $k \in \N$ and $\bar{S} \subseteq \R^k$, a \emph{cut-generating function pair (or simply, cut-generating pair) $(\psi, \pi)$} for $\bar{S}$ is a pair of functions $\psi, \pi\colon\R^n \to \R$ such that 
\begin{equation}
	\label{psi pi ineq}
	\sum_{i=1}^m\psi(\ve r^i) s_i + \sum_{j=1}^\ell\pi(\ve p^j) y_j \ge 1
\end{equation}
is a valid inequality (also called a \emph{cutting plane} or {\em cut}) for the set $X_{\bar{S}}(R,P)$ for every choice of $m, \ell \in \Z_+$ and for all matrices $R \in \R^ {k \times m}$ and $P \in \R^ {k \times \ell}$. We emphasize that cut-generating pairs depend on $k$ and $\bar{S}$ and do \emph{not} depend on $m,\ell$, $R$ and $P$. {\em A priori} it is not clear that such cut-generating function pairs can exist. However, it has been observed that for many special cases of model~\eqref{eq:def mixed-int set} the convex hull of points in $X_{\bar{S}}(R,P)$ can be completely described using cut-generating functions, i.e., not only do they exist, but they are sufficient for the purposes of optimization from a theoretical perspective.

Gomory and Johnson's joint work in the 1970s~\cite{infinite,infinite2}, together with Johnson's independent results~\cite{johnson} in the same decade, shows that cut-generating pairs can be understood by studying infinite-dimensional
convex sets parameterized by $k \in \N$ and $\bar{S} \subseteq \R^k$. For any index
set $I$ (not necessarily finite), $\R^I$ will denote the vector space of all
real-valued functions with domain $I$, and $\R^{(I)}$ will denote the subspace
of real-valued functions with domain $I$ that have finite support, i.e.,
functions that take value zero except on a finite set.\footnote{This notation
  for functions of finite support is used, for example, in~\cite{hitchhiker}.} For example, $\R^{(\R^k)}$ is the set of all functions $s \colon \R^k \to \R$ with finite support. The object of interest is
\begin{equation}
  \label{eq:inf-dim}
  X_{\bar{S}} := \biggl\{\, (s,y) \in \R^{(\R^k)} \times \R^{(\R^k)} \biggst \sum_{\ve r \in \R^k} \ve r s(\ve r) + \sum_{\ve p \in \R^k}\ve p y(\ve p) \in  \bar{S},\,\, s(\ve r) \in \R_+ \;\forall \ve r \in \R^k,\ y(\ve p) \in \Z_+ \;\forall \ve p\in \R^k \,\biggr\}.
\end{equation}

The convex hull of points in $X_{\bar{S}}$ is an infinite-dimensional convex
set in $\R^{(\R^k)} \times \R^{(\R^k)}$ that contains the convex hull of every $X_{\bar{S}}(R,P)$
(for every choice of $R$ and $P$) as a finite-dimensional face. Cut generating
function pairs can then be interpreted as halfspaces in the vector space
$\R^{(\R^k)} \times \R^{(\R^k)}$ that contain $X_{\bar{S}}$. 

\subsection{Approaches to understanding cut-generating function pairs}\label{s:diff-approach}

The setting of $\conv(X_{\bar{S}})$ where $\bar{S}$ is a translate of $\Z^k$
has received the most attention in the literature.\footnote{This model is
  called the \emph{mixed-integer infinite relaxation}, for example in the
  survey \cite{corner_survey}, or sometimes the \emph{mixed-integer group
    problem}, but we shall not use either of these terms in the remainder of
  our survey.}  
Fix a point $\f \in \R^k \setminus \Z^k$ and let $\bar{S} = \f + \Z^k$.
Two distinct approaches have emerged within the study of the
facial structure of $\conv(X_{\bar{S}})$, which we will compare below.

\begin{enumerate}
\item {\bf The infinite group problem.} Gomory and Johnson, in their work
  in~\cite{infinite,infinite2}, study the \emph{infinite group problem}, 
  which appears as the face of $\conv(X_{\bar{S}})$
  given by $\conv(X_{\bar{S}}) \cap \big\{\,(s,y) \in \R^{(\R^k)} \times
  \R^{(\R^k)} \st s = 0 \,\big\}$. This produces cut-generating functions $\pi
  \colon \R^k \to \R$ that are useful for the study of {\em pure} integer
  optimization problems.  The structure of these functions $\pi$ can be very
  complicated; it is the main topic of our survey. 

  By Johnson's fundamental work~\cite{johnson}, we know that 
  these functions~$\pi$ can then be easily lifted to
  strong cut-generating pairs $(\psi, \pi)$ for mixed-integer optimization
  problems using closed form formulas.
    
\item {\bf Intersection cuts.}  Another approach to cut-generating pairs has
  its roots in Balas' work on {\em intersection cuts}~\cite{bal} and Balas and
  Jeroslow's work on {\em monoidal strengthening}~\cite{baljer}.  More recent
  work by Andersen, Louveaux, Weismantel, and Wolsey~\cite{alww} renewed the
  interest in this approach.  Borozan and Cornu\'ejols~\cite{BorCor} put it in
  the framework of cut generating functions, and Dey and Wolsey~\cite{dw2008}
  interpreted monoidal strengthening in this setting.
  This line of research was developed further in many papers,
  including~\cite{bccz,bccz2,basu2013unique,basu-cornuejols-koeppe:unique-minimal-liftings-simplicial,ccz,dey2010constrained}.  

  Consider again the case $\bar{S} = \f + \Z^k$.
  Then the face of $\conv(X_{\bar{S}})$ given by
  $\conv(X_{\bar{S}}) \cap \big\{\,(s,y) \in \R^{(\R^k)} \times \R^{(\R^k)}
  \st y = 0 \,\big\}$ is studied first,\footnote{This model is called the
    \emph{continuous infinite relaxation}, for example in the survey \cite{corner_survey}, or sometimes the \emph{continuous
      group problem}.} giving cut-generating functions $\psi \colon
  \R^k \to \R$. 
  They are obtained as the gauge functions of maximal lattice-free convex
  bodies. 
  The functions~$\psi$ are then lifted to cut-generating pairs $(\psi, \pi)$ for $X_{\bar S}$. 

The advantage of the intersection cut approach, compared with Gomory--Johnson's infinite group problem, is that
the gauge functions $\psi$ can be evaluated using simpler formulas. Further,
generalizations have been studied in which the
set $\bar{S}$ is allowed to be more general than just a translated lattice --
the most frequently studied $\bar{S}$ is of the form $C \cap (\f + \Z^k)$
where $C$ is a convex subset of $\R^k$ and $\f \in \R^k \setminus \Z^k$ (for
example, $C = \R^k_+$ would correspond to model~\eqref{eq:simplex-tab}).  In
this case, the cut-generating functions are obtained from so-called maximal
$\bar{S}$-free convex sets.

The drawback of the intersection cut approach is that lifting a gauge function
$\psi$ to a strong cut-generating pair $(\psi,\pi)$ can be rather difficult.
This difficulty has been recently studied
by~\cite{basu2013unique,basu-cornuejols-koeppe:unique-minimal-liftings-simplicial,ccz,averkov2014unique}. Moreover,
this approach produces a much smaller subset of cut-generating pairs as
compared with the infinite group approach when $\bar{S}$ is a translated
lattice. In this case, there exist undominated cut-generating pairs $(\psi,
\pi)$ where $\psi$ is not the gauge of a {\em maximal} lattice-free set --
these can still be obtained in the context of the infinite group problem. In
contrast the approach outlined above starts with a function~$\psi$ that is the
gauge function of a maximal lattice-free set, and so the approach cannot
derive such cut-generating functions.
\end{enumerate}

\begin{remark}  The study of cut-generating functions for $k=1$ is referred to
  as the {\em single-row} problem, and the general $k\geq 2$ case is referred
  to as the {\em multi-row} problem in the literature. Algorithms used in practice for solving mixed-integer problems have so far used only insights from the single-row problem. It is believed that the general multi-row analysis can lead to stronger cutting planes that can significantly boost the performance of state-of-the-art algorithms.
\end{remark}

\subsection{Outline of the survey} We will survey the recent progress made on
the infinite group problem approach described in~\autoref{s:diff-approach}. We
view this as a follow-up to two excellent surveys, the first by Conforti,
Cornu\'ejols, and Zambelli~\cite{corner_survey}, which discusses the basic structure of the corner
polyhedron and its relation with cut-generating functions, and the
second by Richard and Dey~\cite{Richard-Dey-2010:50-year-survey}, which
focuses on the group-theoretic approach. Our survey focuses on the milestones
that have been reached
since~\cite{corner_survey,Richard-Dey-2010:50-year-survey} were
written. Although we do not
intend~\cite{corner_survey,Richard-Dey-2010:50-year-survey} to be
prerequisites to this article, the reader who is familiar with the material
from~\cite{corner_survey,Richard-Dey-2010:50-year-survey} will certainly have
a better context for the current article.  The reader may use
  \autoref{tab:notation} in \autoref{s:appendix-notation}
as a reference to notation in these surveys and other literature.

Section~\ref{sec:basic-notions} formally introduces the problem, the main objects of study such as {\em valid functions, minimal valid functions, extreme functions, and facets}, and their basic properties. We conclude the section with a discussion of families of valid functions and some open questions (\autoref{s:taxonomy}). The discussion references a compendium that summarizes known families from the literature (\autoref{s:appendix-compendium}), and contains some previously unknown families such as extreme functions with 5 slopes and some discontinuous extreme functions with left and right discontinuity at the origin. Section~\ref{sec:piecewise} introduces the notation and concepts from discrete geometry required for analyzing the problem, and collects foundational techniques for the general $k$-row problem. Section~\ref{sec:foundations} surveys higher-dimensional variants of the celebrated {\em Interval Lemma}. Section~\ref{s:sufficient-cond} introduces one of the most general sufficient conditions for the fundamental notion of {\em extremality}, illustrating how all the techniques introduced in the previous sections come together to analyze extremality. Section~\ref{s:limits} investigates some analytic properties of the problem and demonstrates the use of analytical ideas to construct extreme functions. Sections~\ref{sec:one-two-dim} and~\ref{sec:alg-restriction-finite-groups} discuss important algorithmic and structural results known for the one-row and two-row problems. These results are based on recent breakthroughs in~\cite{basu-hildebrand-koeppe:equivariant,basu-hildebrand-koeppe:equivariant-2,bhk-IPCOext}.

We highlight results that are new in this survey with the annotation ``\TheoremNEWRESULT''. To the best of our knowledge, these do not appear elsewhere in the literature.

Due to constraints of space, we must limit the topics covered in this survey. We briefly mention some of the important highlights in the literature that are not discussed in this survey. A wealth of results on the {\em finite group problem} are closely related to the infinite group problem. We invite the reader to explore the survey by Richard and Dey~\cite{Richard-Dey-2010:50-year-survey} for more details about this direction. Furthermore, we focus on the structural results of the infinite group problem, as opposed to the implementation of these results in algorithms to solve integer programming problems. This includes the so-called {\em shooting experiments} discussed in~\cite{gomory2003corner} to empirically judge quality of the cutting planes, and the discussion of relative strength in~\cite[section 6]{tspace}.

%
%
%
%
%
%
%

\section{The Infinite Group Problem}\label{sec:basic-notions}

As stated in~\autoref{s:diff-approach}, Gomory and Johnson introduced the so-called {\em infinite group problem}. 
It has its roots in Gomory's \emph{group problem}~\cite{gom}, which was
introduced by him as an algebraic relaxation of pure integer linear
optimization problems. 
We introduce this next as it will be useful for formulating many of our results in a unified language. One considers an
abelian group $G$, written additively, and studies
the set of functions $y \colon G \to \R$ satisfying the following constraints:  
\begin{equation}
  \label{GP} 
  \begin{aligned}
    &\sum_{\rx \in G} \rx\, y(\rx) \in \ve f + S \\
    &y(\rx) \in \mathbb{Z}_+ \ \ \textrm{for all $\rx \in G$}  \\
    &y \textrm{ has finite support}, 
  \end{aligned}
\end{equation}
where $S$ is a subgroup of $G$ and $\ve f$ is a given element in $G\setminus S$; so $\ve
f + S$ is the coset containing the element~$\ve f$. 
We are interested in studying the convex hull $R_{\ve f}(G,S)$ of
the set of all functions $y\colon G \to \R$ satisfying the constraints in~\eqref{GP}. $R_{\ve f}(G,S)$ is a convex subset of the vector space $\R^{(G)}$, which is infinite-dimensional when $G$ is an infinite group, i.e., of infinite order. The nomenclature {\em $k$-row infinite group problem} is reserved for the situation when $G =\R^k$ is taken to be the group
of real $k$-dimensional vectors under addition, and $S= \Z^k$ is the subgroup of the
integer vectors
. When $k=1$, we refer to it as the {\em single-row infinite group problem}. Recall that the connection with the cut-generating function model~\eqref{eq:def mixed-int set} is made by setting $\bar S = \f + S$, whence we get $R_\f(G,S)$ as the projection of $\conv(X_{\bar{S}}) \cap \big\{(s,y) \in \R^{(\R^k)} \times \R^{(\R^k)} \st s = 0 \big\}$ onto the $y$ space.

\begin{remark} 
  \label{rem:finite-group}
  Note that there is a correspondence between the sets $R_\f(G,S)$ and $R_{\bar
    \f}(G/S,0)$ where $G/S$ is the quotient group with respect to the (normal)
  subgroup $S$ and $\bar \f$ is the element corresponding to the coset $\f +
  S$, by standard aggregation of variables.\footnote{Indeed, $y \in R_\f(G,S)$
    gives an element $\bar y \in \R_{\bar \f}(G/S,0)$ 
    by setting $\bar y(C) = \sum_{\rx \in C} y(\rx)$ for every coset $C \in
    G/S$. In the other direction, given $\bar y \in \R_{\bar \f}(G/S,0)$ we get
    a solution $y \in R_\f(G,S)$ by simply picking a canonical representative
    $\rx_C$ for each coset $C \in G/S$ and setting $y(\rx_C) = \bar y(C)$.
    From aggregation of variables it follows that the strongest valid
    inequalities for the convex hull of 
    $R_\f(G,S)$ will have identical coefficients on any coset; see
    \autoref{thm:minimal}.}
  
  In the earlier literature on the infinite group problem, the
  aggregated formulation $R_{\bar \f}(\R^k/\Z^k, 0)$ was used.  The quotient
  $\R^k/\Z^k$ is the $k$-dimensional torus; it can be identified with the
  half-open unit cube $[0,1)^k$, using coordinatewise arithmetic modulo 1.  
  In this survey, however, we follow the trend in the recent literature
  \cite{corner_survey,basu-hildebrand-koeppe:equivariant,basu-hildebrand-koeppe:equivariant-2,bhk-IPCOext}
  to work with $R_\f(\R^k,\Z^k)$ instead.  This removes the need for
  complicated notation for mapping between elements of~$\R^k$ and elements of
  $\R^k/\Z^k$ (see \autoref{tab:notation} for an overview), 
  and for complicated geometric notions, such as ``wrap-around'' line segments in
  Johnson's \emph{cylindrical space} \cite{tspace}, in favor of the standard mathematical
  language of periodic, locally finite polyhedral complexes on~$\R^k$
  (\autoref{sec:def-complex-piecewise}).   We pay a small price for the simplicity
  and precision of this approach:  We will often work with infinite objects where finite objects
  would suffice.  However, it is very easy to go back to finite objects in the
  moments when we want to state algorithms.

  The aggregated formulation is still of interest for the case where $G/S$ is
  a finite group, as then $R_{\bar \f}(G/S,0)$ is finite-dimensional and thus
  amenable to polyhedral techniques.  This case is referred to as a {\em
    finite group problem}; it will appear in \autoref{s:finite-group}.
  Due to the correspondence between the sets
  $R_\f(G,S)$ and $R_{\bar \f}(G/S,0)$, we shall also refer to $R_\f(G,S)$ as
  a finite group problem whenever $S$ has finite index in $G$, i.e., $G/S$ is a
  finite group.  
\end{remark}

 \subsection{Valid inequalities and valid functions} 
 
Following Gomory and Johnson, we are interested in the description of $R_{\ve f}(G,S)$ as
the intersection of halfspaces in $\R^{(G)}$.  We first describe the general
form that these halfspaces take and then a standard normalization that leads
to the idea of cut-generating functions.  

\subsubsection{Valid inequalities}
Any halfspace in $\R^{(G)}$ is given by a pair $(\pi, \alpha)$, where $\pi \in \R^G$ and $\alpha \in
\R$, and the halfspace is the set of all $y \in \R^{(G)}$ that satisfy $\sum_{\rx \in G} \pi(\rx)y(\rx) \geq
\alpha$. The left-hand side of the inequality is a finite sum because $y$ has finite
support. Such an inequality is called a {\em valid inequality} for $R_{\ve f}(G,S)$
if $\sum_{\rx \in G} \pi(\rx)y(\rx) \geq \alpha$ for all $y \in R_{\ve
  f}(G,S)$, i.e., $R_{\ve f}(G,S)$ is contained in the halfspace defined by
$(\pi, \alpha)$. Note that the set of all valid inequalities $(\pi,\alpha)$ is a cone in
the space $\R^G \times \R$.  

\subsubsection{Sign of the coefficients of valid inequalities} If $S$ has
finite index in $G$, then it can be shown that if $(\pi, \alpha)$ gives a
valid inequality, then $\pi \geq 0$. An even stronger statement is easily seen to be true: if $\rx \in G$ is such that there exists $n \in \N$ satisfying $n \rx \in S$, then $\pi(\rx) \geq 0$ \cite[section 5]{corner_survey}.  However,  when this is not the case, there may exist valid inequalities $(\pi,\alpha)$ where $\pi$ takes negative values. We give an explicit example below for the one-row infinite group problem ($G = \R$ and $S=\Z$). 

It is well-known that there exist functions $h\colon \R \to \R$ such that they
satisfy $h(a+b) = h(a) + h(b)$ for all $a, b \in \R$ and whose graph is dense
in $\R^2$. These are the non-regular solutions to the so-called {\em Cauchy
  functional equation}~\cite[chapter 2, Theorem 3]{aczel1989functional}. This functional equation is discussed further in \autoref{sec:foundations}.

\begin{prop}[\TheoremNEWRESULT]\label{prop:neg-pi}
Let $f$ be any rational number. Let $h\colon\R \to \R$ be any function such
that $h(a+b) = h(a) + h(b)$ for all $a,b \in \R$ and the graph of $h$ is dense
in $\R^2$. Define $\pi^*\colon\R \to \R$ as $\pi^*(a) = h(a) - h(1)a$
for all $a \in \R$.
Then $(\pi^*, 0)$ defines an \emph{implicit equality} of~$\R_{\ve f}(G,S)$, 
i.e., the equation 
$$ \sum_{\rx \in G} \pi^*(\rx)y(\rx) = 0\quad\text{holds for $y \in \R_{\ve f}(G,S)$}.$$
Thus both $(\pi^*, 0)$ and $(-\pi^*, 0)$ define valid inequalities for
$\R_{f}(G,S)$. 
Moreover $\pi^*$ has a dense graph in $\R^2$. \end{prop}

\begin{proof}
Using additivity, $h(a) = h(1)a$ for all rational $a$ and therefore we have $\pi^*(f + w) = 0$ for any $w \in \Z$. Moreover, since $h(a+b) = h(a) + h(b)$ for all $a,b \in \R$, we also have $\pi^*(a+b) = \pi^*(a) + \pi^*(b)$ for all $a,b \in \R$. Consider any $y \in \R^{(\R)}$ such that $\sum_{r \in \R} r\, y(r) = f + w$ for some $w \in \Z$, and $y(r) \in \Z_+$ for all $r \in \R$. Then $0 = \pi^*(f + w) = \pi^*(\sum_{r \in \R} r\, y(r)) = \sum_{r \in \R} \pi^*(r)\, y(r)$. This establishes that $(\pi^*, 0)$ defines an implicit equality for $R_f(G,S)$. The graph of $\pi^*$ is dense in $\R^2$ because the graph of $h$ is dense in $\R^2$.\end{proof}

In fact, for the infinite group problem $R_{\ve f}(\R^k,\Z^k)$ with
rational~$\ve f$ we show that the
set of implicit equalities (equivalently, the lineality space of the cone of
valid inequalities) 
consists of the $(\pi, \alpha)$ such that $\pi$ is additive and $\alpha =
0$. The discussion above says that for any valid inequality given by the pair
$(\pi, \alpha)$ we have $\pi(\rx) \geq 0$ for every $\rx \in \Q^k$.  

\begin{prop}[\TheoremNEWRESULT]\label{prop:additive}
Let $\ve f$ be a rational vector. A pair $(\pi, \alpha) \in \R^{\R^k} \times \R$ satisfies $\sum_{\rx \in \R^k} \pi(\rx)y(\rx) = \alpha$ for all $y \in R_{\ve f}(\R^k,\Z^k)$ if and only if $\pi$ is additive, i.e., $\pi(\rx^1) + \pi(\rx^2) = \pi(\rx^1 + \rx^2)$ for all $\rx^1, \rx^2 \in \R^k$, and $\alpha = 0$.
\end{prop}


\begin{proof}
The ``if'' direction can be proved using the same calculations as in the proof of Proposition~\ref{prop:neg-pi}. We prove the ``only if'' direction.

For $\rx \in \R^k$, let $e_\rx$ denote the finite support function which takes value 1 at $\rx$ and 0 everywhere else. Then $y^1 = e_{\rx^1 + \rx^2} + e_{\ve f - \rx^1 - \rx^2} \in R_{\ve f}(\R^k,\Z^k)$. Therefore, $\alpha = \sum_{\rx \in \R^k} \pi(\rx)y^1(\rx) = \pi(\rx^1 + \rx^2) + \pi(\ve f  - \rx^1 - \rx^2)$. Similarly, $e_{\rx^1} + e_{\rx^2} + e_{\ve f - \rx^1 - \rx^2} \in R_{\ve f}(\R^k,\Z^k)$ and therefore $\alpha = \pi(\rx^1) + \pi(\rx^2) + \pi(\ve f - \rx^1 - \rx^2)$. Therefore, $\pi(\rx^1) + \pi(\rx^2) = \pi(\rx^1 + \rx^2)$. 

Additive functions take value $0$ at the origin: $\pi(\ve 0) + \pi(\ve 0) = \pi(\ve 0)$ which implies $\pi(\ve 0) = 0$. Since, $\pi(\rx) \geq 0$ for every $\rx \in \Q^k$ and for any rational $\rx$, $\pi(\rx) + \pi(-\rx) = \pi(\ve 0) = 0$ we must have $\pi(\rx) = 0$ for every rational $\rx$. Thus, using the fact that $e_{\ve f} \in R_{\ve f}(\R^k,\Z^k)$, we have $\alpha = \sum_{\rx \in \R^k} \pi(\rx)e_{\ve f}(\rx) = \pi(\ve f) = 0$ since $\ve f$ is rational. \end{proof}

We next show that the intersection of all halfspaces of the form $\sum_{\rx \in G} \pi(\rx)y(\rx) \geq \alpha$ with $\pi \geq 0$ is a much larger superset of $R_{\ve f}(G,S)$. Our example is for $R_f(\R,\Z)$.

\begin{prop}[\TheoremNEWRESULT]\label{prop:nonneg-insufficient}
Let $f$ be any rational number. Let $h\colon\R \to \R$ be any function such
that $h(a+b) = h(a) + h(b)$ for all $a,b \in \R$ and the graph of $h$ is dense
in $\R^2$. Define $\pi^*\colon\R \to \R$ as $\pi^*(a) = h(a) - h(1)a$ for all
$a \in \R$. Let $r^1, \ldots, r^k$ be a finite set of real numbers such that
$\pi^*(r^i) < 0$ for all $i =1, \ldots, k$. Define $y^* \in \R^{(\R)}$ as
$y^*(r) = 1$ if $r \in \{r^1, \ldots, r^k\}\cup \{f\}$, and $y^*(r) = 0$
otherwise. Then \begin{enumerate} \item $y^*$ violates the implicit equality $\sum_{r \in G}
\pi^*(r)\, y(r) = 0$ and thus, does not lie in $R_{\ve f}(\R,\Z)$, \item $y^*$
satisfies all valid inequalities $\sum_{r \in \R} \pi(r)y(r) \geq \alpha$ where
$\pi \geq 0$. \end{enumerate}
\end{prop}

\begin{proof}
Observe that $\sum_{r \in G} \pi^*(r)\, y^*(r) = \sum_{i=1}^k \pi^*(r^i) +
\pi^*(f) = \sum_{i=1}^k \pi^*(r^i) < 0$. By Proposition~\ref{prop:neg-pi},
$(\pi^*,0)$ is an implicit equality for $R_{f}(\R,\Z)$ and therefore, $y^* \not\in R_{f}(\R,\Z)$. 

On the other hand, for any valid inequality given by $(\pi, \alpha)$ such that $\pi \geq 0$, we have $\pi(f) \geq \alpha$ (since $e_f \in R_{f}(\R,\Z)$). So, for any such valid function $\pi$, we have $\sum_{r \in \R} \pi(r)y^*(r) = \sum_{i=1}^k \pi(r^i) + \pi(f) \geq \pi(f) \geq \alpha$ (since $\pi \geq 0$).\end{proof}

The above example takes points that do not satisfy the implicit
equalities, i.e., we consider points outside the affine hull of the
feasible region. If we restrict ourselves to satisfy the implicit equalities,
are the nonnegative valid inequalities sufficient? This is an open question. 

\begin{openquestion}
  Is every valid inequality $(\pi,\alpha)$ the sum of a nonnegative valid
  inequality $(\pi^+,\alpha)$ and an implicit equality $(\pi^=, 0)$? 
\end{openquestion}


\subsubsection{Valid functions} Since data in finite-dimensional integer programs is usually rational, and this is our main motivation for studying the infinite group problem, it is
customary to concentrate on valid inequalities with $\pi \geq 0$; then we can choose, after a scaling, $\alpha = 1$ (otherwise, the inequality is implied by the nonnegativity of $y$). Thus, we only focus on
valid inequalities of the form $\sum_{\rx \in G} \pi(\rx)y(\rx) \geq 1$ with $\pi \geq 0$. Such functions $\pi \in \R^G$ are called {\em valid functions} for $R_{\ve f}(G,S)$. We remind the reader that this choice comes at a price because of \autoref{prop:nonneg-insufficient}; however, it can be shown that for rational {\em corner polyhedra}, which form an important family of relaxations for integer programs, all valid inequalities are restrictions of nonnegative valid functions for the infinite group problem. See~\cite{corner_survey} for a discussion.


%
 \subsection{Minimal functions, extreme functions and facets}\label{s:minimal-def}
 Gomory and Johnson~\cite{infinite,infinite2} defined a hierarchy on the set of valid
 functions, capturing the strength of the corresponding valid inequalities, which we summarize now.
 
\begin{figure}[tp]\centering
  \input{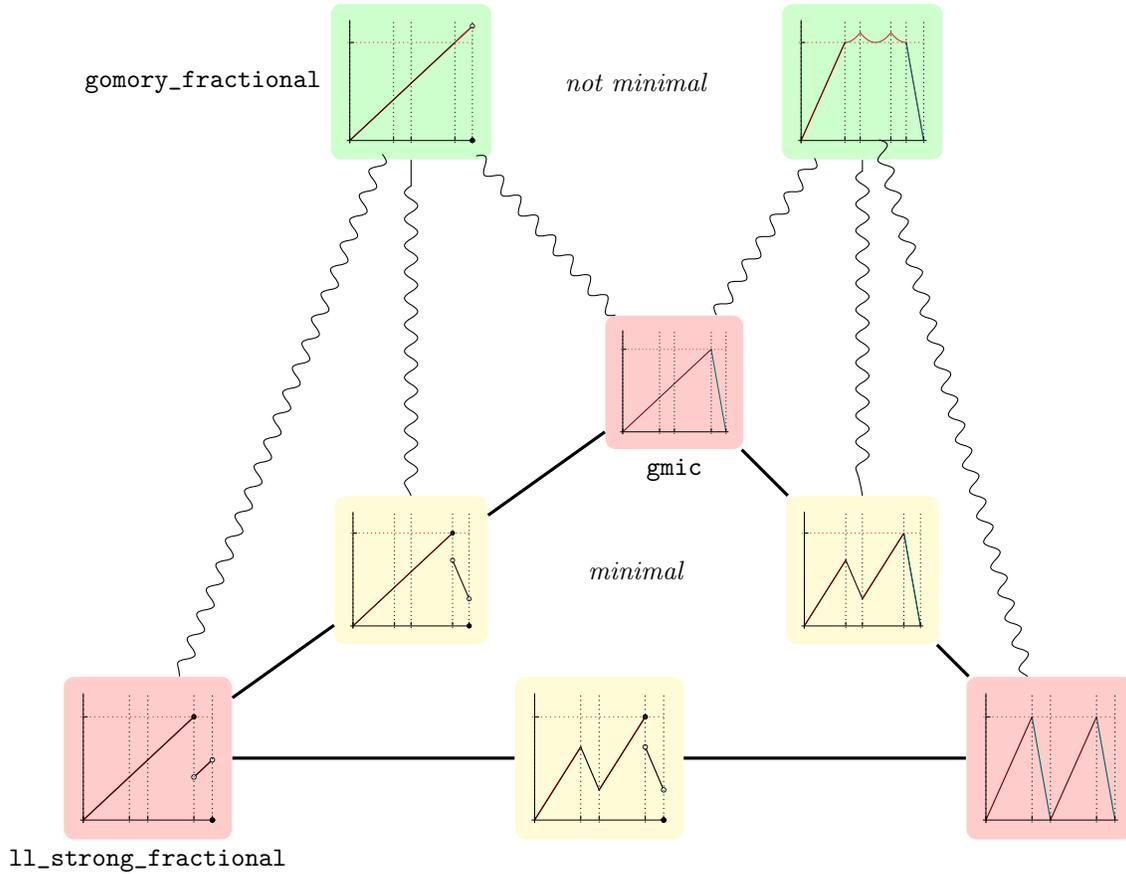}
  \caption{The hierarchy of valid, minimal, and extreme functions by example
    for the case $R_f(\R,\Z)$.
    Pairwise convex combinations (\emph{solid lines forming the bottom triangle}) of three extreme functions
    (\emph{graphs on red background at the corners}) give non-extreme, minimal functions
    (\emph{graphs on yellow background on the edges}).  
    These functions dominate (\emph{wavy lines}) various non-minimal, valid functions
    (\emph{graphs on green background, top}).  Even without checking the
    dominance, it is easy to see that the functions shown on the top cannot be
    minimal: they have some function values larger 
    than~$1$ (\emph{international orange}), but minimal valid functions are upper bounded by~1 by \autoref{thm:minimal}.
    Since minimal valid functions for $R_f(\R,\Z)$ are periodic with respect to~$\Z$, we only show
    the interval $[0,1]$.
  }
  \label{fig:hierarchy}
\end{figure}

\begin{figure}[tp]
  \subfigure[]{\label{subfig:infinite-hierarchy}
  \begin{tikzpicture}
  [font=\small,
  valid/.style={rounded corners,fill=green!20},
  minimal/.style={rounded corners,fill=yellow!20},
  extreme/.style={rounded corners,fill=red!20},
  convexcomb/.style={very thick},
  domination/.style={},
  facet/.style={rounded corners,fill=blue!20},
  weakfacet/.style={rounded corners,fill=orange!60}]
  \draw[valid] (0,0)--(0,6)--(7,6)--(7,0)--cycle;
  \draw[minimal] (.5,1)--(.5,5)--(6.5,5)--(6.5,1)--cycle;
  \draw[extreme] (2.5,3) ellipse (1.75 and 1.25);
  \draw[weakfacet, opacity = .5] (4.5,3) ellipse (1.75 and 1.25);
   \draw (2.5,3) ellipse (1.75 and 1.25);
    \draw (4.5,3) ellipse (1.75 and 1.25);
  \draw[facet] (3.5,3) ellipse (.75 and .75);
  
  \node at (3.5, 5.5) { Valid Functions};
  \node at (3.5, 4.6) { Minimal Functions};
  \node at (1.8, 3.25) { Extreme};
  \node at (1.8, 2.75) { Functions};
  \node at (5.1, 3.25) {Weak};
  \node at (5.1, 2.75) {Facets};
  \node at (3.5, 3) { Facets};

\end{tikzpicture}
}
\subfigure[]{\label{subfig:finite-hierarchy}
  \begin{tikzpicture}
  [font=\small,
  valid/.style={rounded corners,fill=green!20},
  minimal/.style={rounded corners,fill=yellow!20},
  extreme/.style={rounded corners,fill=red!20},
  convexcomb/.style={very thick},
  domination/.style={},
  facet/.style={rounded corners,fill=blue!20},
  weakfacet/.style={rounded corners,fill=orange!20}]
  \draw[valid] (0,0)--(0,6)--(7,6)--(7,0)--cycle;
  \draw[minimal] (1,1)--(1,5)--(6,5)--(6,1)--cycle;
  \draw[extreme] (3.5,3) ellipse (2 and 1.25);
  
  \node at (3.5, 5.5) { Valid Functions};
  \node at (3.5, 4.6) { Minimal Functions};
  \node at (3.5, 3.5) { Extreme Functions};
  \node at (3.5, 2.5) { Weak Facets};
  \node at (3.5, 3) { Facets};
\end{tikzpicture}  
}
\caption{The hierarchy of valid, minimal, and extreme functions and facets
    and weak facets. (a) General case. (b) Situation in the finite-dimensional
    case and in the case of continuous piecewise linear functions with
    rational breakpoints.   
  }
  \label{fig:hierarchy-with-facets}
\end{figure}
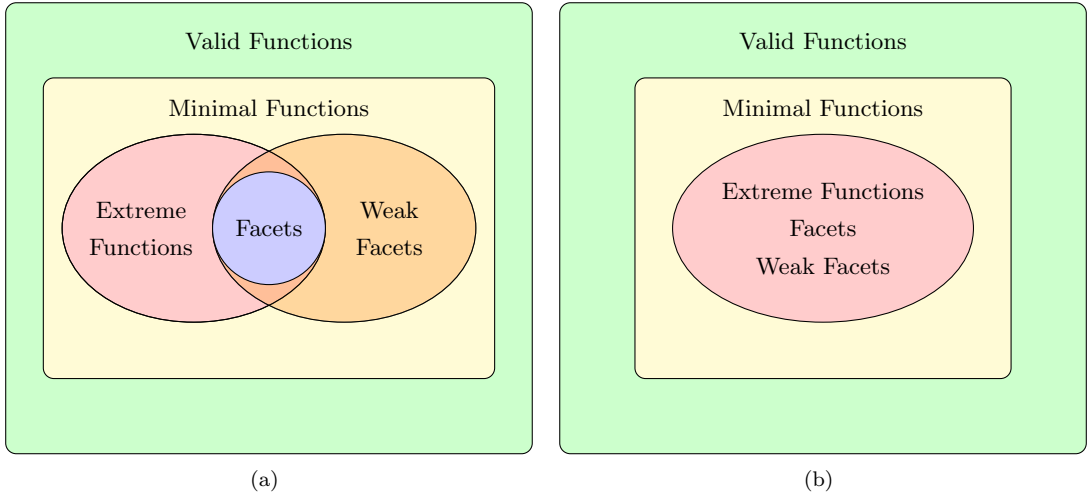

\subsubsection{Minimal functions}
A valid function $\pi$ for $R_{\ve f}(G,S)$ is said to be
\emph{minimal} for $R_{\ve f}(G,S)$ if there is no valid function $\pi' \neq \pi$
such that $\pi'(\rx) \le \pi(\rx)$ for all $\rx \in G$.  For every valid
function $\pi$ for $R_{\ve f}(G,S)$, there exists a minimal valid function $\pi'$
such that $\pi' \leq \pi$ \cite[Theorem~1.1]{bhkm}, and thus non-minimal valid
functions are redundant in the description of $R_{\ve f}(G,S)$.  
Note that $\pi'$ is not uniquely determined (\autoref{fig:hierarchy}).

Minimal
functions for $R_{\ve f}(G,S)$ were characterized by Gomory for the case where
$S$ has finite index in~$G$ in~\cite{gom}, and later for $R_{\ve f}(\R,\Z)$ by Gomory and
Johnson~\cite{infinite}. We state these results in a unified notation in the
following theorem. 

A function $\pi\colon G \rightarrow \mathbb{R}$ is \emph{subadditive} if
$\pi(\x + \y) \le \pi(\x) + \pi(\y)$ for all $\x,\y \in G$. We say that  $\pi$ is
\emph{symmetric} (or \emph{satisfies the symmetry condition}) if $\pi(\x) +
\pi(\f - \x) = 1$ for all $\x \in G$.

\begin{theorem}[Gomory and Johnson \cite{infinite}] \label{thm:minimal} Let $G$ be an abelian group, $S$ be a subgroup of $G$ and $\f \in G\setminus S$. Let
  $\pi \colon G \rightarrow \mathbb{R}$ be a nonnegative function. Then $\pi$
  is a minimal valid function for $R_{\ve f}(G,S)$ if and only if $\pi(\ve z) = 0$ for
  all $\ve z\in S$, $\pi$ is subadditive, and $\pi$ satisfies the symmetry
  condition. (The first two conditions imply that $\pi$ is periodic modulo
  $S$, that is, $\pi(\x) = \pi(\x + \ve z)$ for all $\ve z \in S$, and the symmetry condition implies that the values of minimal functions are bounded between $0$ and $1$.) 
\end{theorem}

See \cite[Theorem~5.4]{corner_survey} for a proof. 

\begin{figure}

\begin{center}
\includegraphics[width=.31\linewidth]{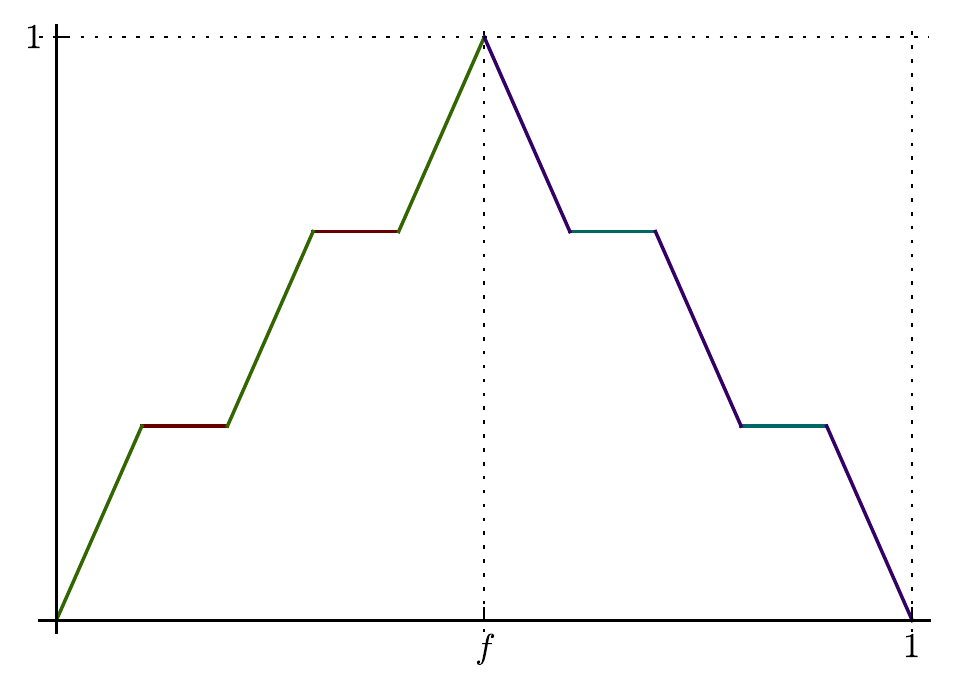}\quad
\includegraphics[width=.31\linewidth]{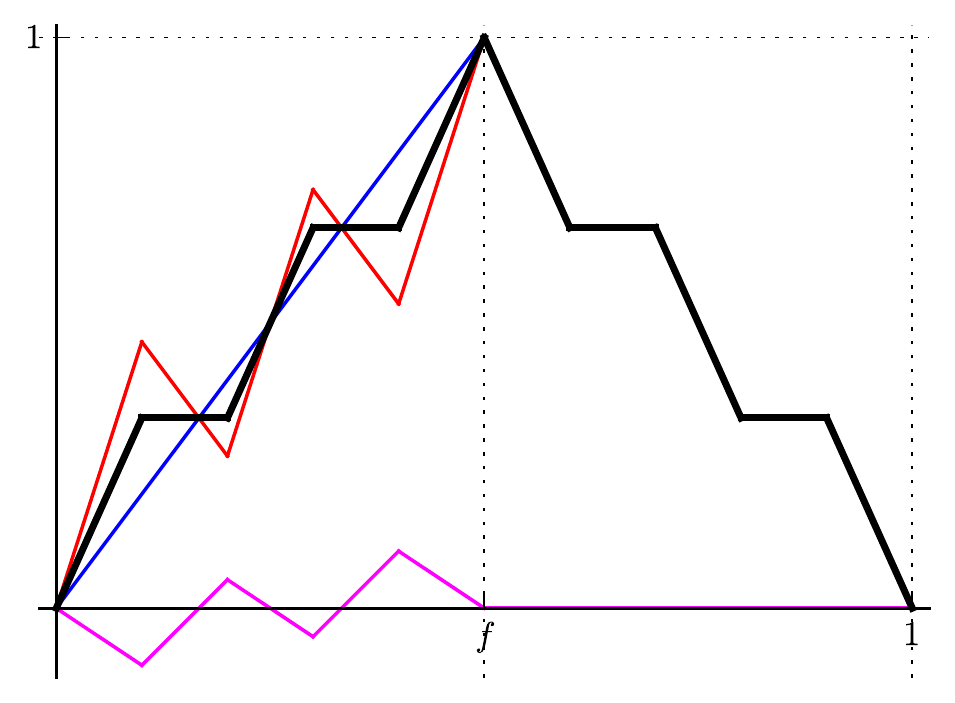}\quad
\includegraphics[width=.31\linewidth]{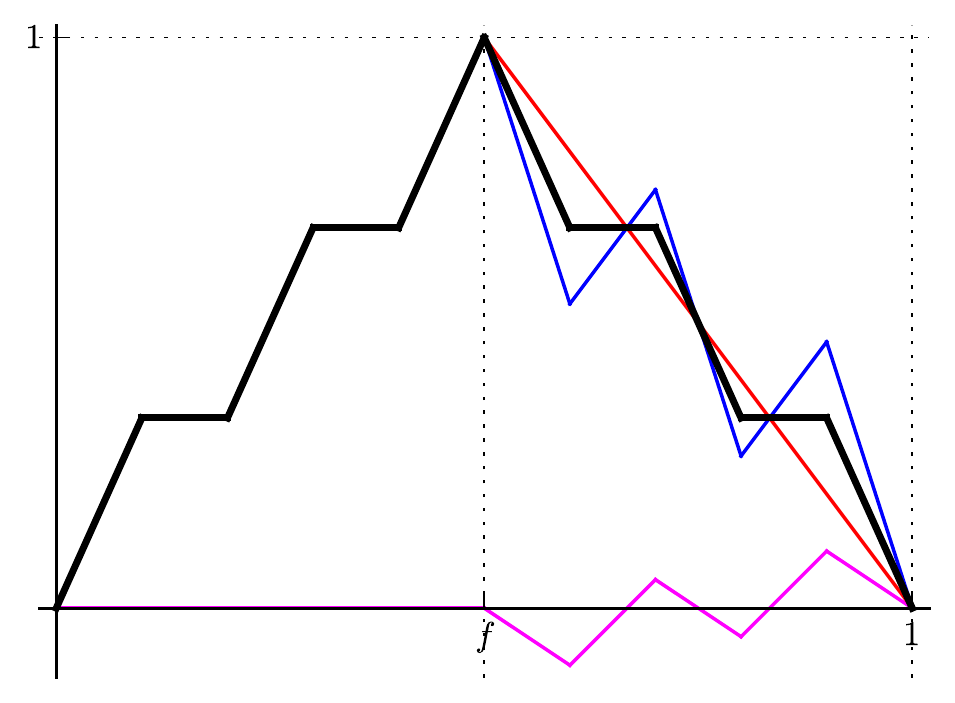}
\end{center}
\caption{This function (\sage{h = \sagefunc{not_extreme_1}()}) is
  minimal, but not extreme (and  hence also not a facet), as proved by
  \sage{\sagefunc{extremality_test}(h, show\underscore{}plots=True)}.
  The procedure first shows that
  for any distinct minimal $\pi^1 = \pi + \bar\pi$ (\emph{blue}), $\pi^2 = \pi
  - \bar\pi$ (\emph{red}) such that $\pi = \tfrac{1}{2}\pi^1
  + \tfrac{1}{2} \pi^2$, the functions $\pi^1$ and $\pi^2$ are continuous
  piecewise linear with the same breakpoints as $\pi$ (in the terminology of
  \cite{basu-hildebrand-koeppe:equivariant}, $\pi$ is \emph{affine imposing}
  on all intervals between breakpoints).  A finite-dimensional extremality
  test then finds two linearly independent perturbations $\bar\pi$ (\emph{magenta}), as shown.
}
\label{fig:minimalNotExtreme}
\end{figure}


\subsubsection{Extreme functions}

In polyhedral combinatorics, one is interested in classifying the
facet-defining inequalities of a polytope, which are the strongest
inequalities and provide a finite minimal description. In the infinite group
problem literature, three notions analogous to that of a facet-defining
inequality have been proposed, which are not known
to be equivalent.  We start with the notion of an {\em extreme function}. 

A~valid function~$\pi$ is \emph{extreme}
for $R_{\ve f}(G,S)$ if it cannot be written as a convex combination of two other
valid functions for $R_{\ve f}(G,S)$, i.e., $\pi = \tfrac12(\pi^1 + \pi^2)$
implies $\pi = \pi^1 = \pi^2$ (see \autoref{fig:hierarchy} and \autoref{fig:minimalNotExtreme}).  
Extreme functions are easily seen to be minimal. In fact we may view this definition from a convex geometry perspective. By Theorem~\ref{thm:minimal}, the set of minimal valid functions is a convex subset of the infinite-dimensional space $\R^G$ of real-valued functions on $G$; this follows from the observation that all the properties in Theorem~\ref{thm:minimal} are preserved under taking convex combinations of functions. 

\begin{proposition}[\TheoremNEWRESULT]\label{obs:compact}
The set of minimal valid functions is a compact convex set under the product topology on the space $\R^G$ of real-valued functions on $G$.
\end{proposition}

The proof appears in~\autoref{s:minimal-limits}.
 In the light of \autoref{obs:compact}, it is natural to study the
 extreme points of this compact convex set of minimal valid functions. These
 are precisely the extreme functions. By an application of the Krein--Milman
 theorem, all minimal valid functions are either convex combinations of extreme functions or pointwise limits of such convex combinations (i.e., limits in the product topology).

\subsubsection{Facets and weak facets}
A related notion is that of a \emph{facet}.  
Let $P(\pi)$ denote the set of all feasible solutions $y \in \R^{(G)}$ satisfying~\eqref{GP} such that $\sum_{r \in G} \pi(r) y(r) = 1$.  A valid function $\pi$ is called a \emph{facet} if for every valid function $\pi'$ such that $P(\pi) \subseteq P(\pi')$ we have that $\pi' = \pi$, as defined in~\cite{tspace}. Equivalently, a valid function $\pi$ is a facet if this condition holds for all such \emph{minimal} valid functions $\pi'$  (cf.~\cite{bhkm}).

%

A similar facet definition, which we call a \emph{weak facet}, is given in~\cite{dey3} and in fact was used in an erroneous proof of the so-called {\em Facet Theorem} in~\cite[Theorem 3]{tspace} (see \autoref{thm:facet})\footnote{In a proof by contradiction, they say that if $\pi$ is not a facet, then there exists a valid function $\pi^*$ and a $y^* \in R_\f(G,S)$ such that $y^* \in P(\pi^*)\setminus P(\pi)$. This works when $\pi$ is not a weak facet, but does not work if we assume that $\pi$ is not a facet.}. In particular, a valid function $\pi$ is called a weak facet if for every valid function $\pi'$ such that $P(\pi) \subseteq P(\pi')$ we have that $P(\pi) = P(\pi')$.  
\subsubsection{Relation between the three notions}
\label{subsec:relation-between-three-notions}
Facets are extreme functions (cf.~\cite[Lemma 1.3]{bhkm}), but it is unknown if all
extreme functions are facets.  A facet is also a weak facet, but it is unknown
if all weak facets are facets. Thus, facets are a subset of the intersection
of extreme functions and weak facets, but nothing further is known in general;
see \autoref{fig:hierarchy-with-facets}\,(a). When $G$ is a finite abelian group, the set of minimal functions is a finite-dimensional polyhedron (given by constraints coming from Theorem~\ref{thm:minimal}); see~\autoref{s:finite-group}.  
In this setting, it is well known that the three notions of weak facets,
facets and extreme inequalities are equivalent, and form the extreme points of
this polyhedron; see \autoref{fig:hierarchy-with-facets}\,(b).  In the one-row infinite group problem, we can also establish some equivalence as stated below, which is a consequence of~\autoref{thm:extreme-facet-and-restriction}. The result is new and has not been published before.

\begin{proposition}[\TheoremNEWRESULT]
Suppose $\pi\colon \R \to \R$ is a continuous piecewise linear
function\footnote{See \autoref{sec:def-complex-piecewise} for the definition that we use.} with rational breakpoints in $\frac{1}{q} \Z$ for some $q\in \N$.  Then $\pi$ is extreme if and only if $\pi$ is a facet.
\end{proposition}

\begin{openquestion}
  Are the definitions of facets, weak facets, and extreme functions equivalent?
\end{openquestion}

\subsection{A roadmap for proving extremality and facetness}\label{s:roadmap}

An understanding of the set of points for which the subadditivity relations of a minimal function hold at equality is crucial to the study of both extreme functions and facets. This motivates the following definition.\begin{definition}
  Define the \emph{subadditivity slack} of~$\pi$ as
  \begin{equation}
    \label{eq:delta-pi-definition}
    \Delta \pi(\x, \y) := \pi(\x) + \pi(\y) -\pi(\x + \y)
  \end{equation}
  and the \emph{additivity domain} of~$\pi$ as
  \begin{equation}
    \label{eq:Epi}
    E(\pi) := \setcond{(\x, \y)} { \Delta \pi(\x, \y) = 0}.
  \end{equation}
\end{definition}

Additivity domains are used by Gomory and Johnson to define the notion of {\em merit index} in~\cite{tspace}. The merit index is the volume of $E(\pi)$ (modulo $\Z^n$) and can be taken as a quantitative measure of strength of minimal valid functions. Work on the merit index also appears in~\cite{dey1}. We will not discuss the merit index in this survey; however, the set $E(\pi)$ will be crucial in what follows.


The main technique used to show a function $\pi$ is extreme is to assume that $\pi = \tfrac{1}{2}(\pi^1 + \pi^2)$ where $\pi^1, \pi^2$ are valid functions, and then show that $\pi = \pi^1 = \pi^2$. One then employs the following lemma to infer important properties of $\pi^1, \pi^2$. These following facts can be found in the literature for the one-row problem; the extension to general $k$ is straightforward. 
\begin{lemma}
\label{lem:tightness}\label{lem:minimality-of-pi1-pi2}\label{lemma:tight-implies-tight}\label{Theorem:functionContinuous} \label{lem:lipschitz}
  Let $\pi : \R^k \to \R_+$ be minimal, $\pi = \frac12(\pi^1+\pi^2)$, and $\pi^1,\pi^2$ 
  valid functions.  Then the following hold:
  \begin{enumerate}[(i)]
  \item \label{lem:tightness-minimal} $\pi^1,\pi^2$ are minimal~\cite[Lemma 1.4]{infinite}.
  \item \label{lem:tightness-subadditive} All subadditivity relations $\pi(\x + \y) \le \pi(\x) + \pi(\y)$  
  that are tight for~$\pi$ are also tight for
  $\pi^1,\pi^2$, i.e., $E(\pi) \subseteq E(\pi^1) \cap E(\pi^2)$~\cite[proof of Theorem 3.3]{infinite}.\footnote{When $\pi$ is a discontinuous piecewise linear function, subadditivity gives certain relations on the limit values of the function. We omit this more subtle discussion in this survey; see~\cite{basu-hildebrand-koeppe:equivariant} for more details.}
  \item Suppose there exists a real number $M$ such that $\limsup_{h \to 0} \left| \frac{\pi(h \rx)}{h}\right| \leq M$ for all $\rx \in \R^k$ such that $\lVert \rx \rVert = 1$. 
  Then $\pi$ is Lipschitz continuous.  Furthermore, this
      condition holds for $\pi^1$ and $\pi^2$ and $\pi^1, \pi^2$ are Lipschitz
      continuous~{\cite[Theorem 2.9]{basu-hildebrand-koeppe:equivariant}}.
  \item If $\pi$ is continuous piecewise linear\footnote{See \autoref{sec:def-complex-piecewise} for the definition that we
  use.}, then $\pi, \pi^1, \pi^2$ are
    all Lipschitz continuous {\cite[\autoref{equi3:lemma:tight-implies-tight}]{bhk-IPCOext}}. 
      \item Suppose $k=1$, i.e., $\pi \colon \R\to \R_+$ and $\pi$ is piecewise linear\footnote{See \autoref{sec:def-complex-piecewise} for the definition that we
  use, which includes certain discontinuous functions.} and continuous from
the right at $0$
or continuous from the left at~$0$.\footnote{This condition is also not always
  true for piecewise linear functions. See Table~\ref{tab:compendium} for
  examples of extreme functions that are discontinuous on both sides of the
  origin.  The condition of one-sided continuity at the origin cannot be
  removed from the hypothesis of \autoref{lem:lipschitz}\,(v) (\TheoremNEWRESULT). This is
  illustrated by example
  \sagefunc{zhou_two_sided_discontinuous_cannot_assume_any_continuity},
  constructed by Zhou (2014, unpublished).}
Then $\pi^1$ and $\pi^2$ are continuous at all points at which $\pi$ is continuous~{\cite[Theorem 2]{dey1}}.   \end{enumerate}
\end{lemma}
To prove that a valid inequality is a facet, the main tool is the so-called \emph{Facet Theorem}, originally proved by Gomory and Johnson~\cite{tspace}  for the one-row case; it extends verbatim to
the $k$-row case.\footnote{Gomory and Johnson's original proof actually holds only for weak facets, and not for facets as claimed in~\cite{tspace}.} We present a stronger version of the theorem, which
first appeared in \cite{bhkm}.\footnote{In contrast to Gomory--Johnson's Facet Theorem, 
  the condition that $E(\pi) \subseteq E(\pi')$ implies $\pi' = \pi$ only
  needs to be tested on minimal valid functions, not all valid functions.} 
\begin{theorem}[Facet Theorem \cite{tspace}, {\cite[Theorem 3.1]{bhkm}}]
\label{thm:facet}
Let $\pi$ be a minimal valid function.  Suppose for every minimal valid function $\pi'$, $E(\pi) \subseteq E(\pi')$ implies $\pi' = \pi$.  Then $\pi$ is a facet.
\end{theorem}

In the light of \autoref{lem:tightness} and \autoref{thm:facet}, if one can establish that for a minimal valid function $\pi$, $E(\pi) \subseteq E(\pi')$ implies $\pi' = \pi$ for every minimal valid function $\pi'$, then $\pi$ is extreme, as well as a facet. Indeed, if $\pi = \tfrac{1}{2}(\pi^1 + \pi^2)$ where $\pi^1, \pi^2$ are valid functions, by Lemma~\ref{lem:tightness} (i), $\pi^1$ and $\pi^2$ are minimal and by \autoref{lem:tightness} (ii), $E(\pi) \subseteq E(\pi^1) \cap E(\pi^2)$, and so $\pi = \pi^1 = \pi^2$. The facetness follows directly from~\autoref{thm:facet}, and gives an alternate proof of extremality since all facets are extreme.

The condition that $E(\pi) \subseteq E(\pi')$ implies $\pi' = \pi$ for every minimal valid function $\pi'$ is established along the following lines. First, structural properties of $\pi$ can be used to obtain a structured description of $E(\pi)$.  For example, the fact that $\pi$ is piecewise linear often shows that $E(\pi)$ is the union of many full-dimensional convex sets. $E(\pi')$ shares this structure with $E(\pi)$ because of the assumption that $E(\pi) \subseteq E(\pi')$. Then, results such as the {\em Interval Lemma}, discussed in \autoref{sec:foundations}, are used to show that $\pi'$ must be affine on the set of points contributing to $E(\pi')$. Finally, the conditions that all minimal valid functions are $0$ at the origin and $1$ at $\f + \Z^k$ puts further restrictions on the values that $\pi'$ can take, and ultimately force $\pi' = \pi$.\footnote{Sometimes certain continuity arguments need to be made, where results like Lemma~\ref{lem:tightness} (iii), (iv) and (v) are helpful. In such situations, the proof of extremality is usually slightly simpler than a proof for facetness, owing to Lemma~\ref{lem:tightness} (iii); see \autoref{rem:facet-extreme} and \autoref{rem:facet-non-piece}.}

%
%
%
%

\subsection{Classification and taxonomy of facets and extreme functions}\label{s:taxonomy}

The main goal in the study of the infinite group problem is to obtain a classification of facets and extreme valid functions. We do not believe that a simple classification exists like Theorem~\ref{thm:minimal} for minimal valid functions. In spite of this, several beautiful theorems have been obtained regarding the structure of facets and extreme valid functions, and there is a lot more to be discovered. This survey attempts to highlight the most important known results in this research area and outline some of the challenging open problems. 

Inspired by the survey by Richard and Dey
\cite[p.~786]{Richard-Dey-2010:50-year-survey}, we provide an updated
compendium, or ``taxonomy,'' of known extreme functions at the end of this
survey (\autoref{s:appendix-compendium}).
The focus lies on the case of the one-row ($k=1$) infinite group problem, $R_f(\R, \Z)$, for which many types of
extreme functions have been discovered and analyzed
(\autoref{tab:compendium-1}, \ref{tab:compendium-1b}, \ref{tab:compendium-2}, \ref{tab:compendium-3}). Also a number of ``procedures'' (operations) have been studied in the literature
that preserve extremality under some conditions; we present these in \autoref{tab:compendium-procedures}.

We do not provide explicit constructions or descriptions of these functions
here.  Instead, we invite the interested reader to investigate the functions
in an interactive companion program \cite{infinite-group-relaxation-code},
including the electronic compendium of extreme functions
\cite{electronic-compendium}. 
The program and the electronic compendium are implemented in the free
(open-source) computer algebra package Sage 
\cite{sage}.\footnote{%
  The program \cite{infinite-group-relaxation-code} can be run on a local installation of Sage, or online via
  SageMathCloud.  The help system provides a discussion of parameters of the
  extreme functions, bibliographic information, etc.  It is accessed 
  by typing the function name as shown in the table,
  followed by a question mark. Example:
  \sage{gmic?}}

Most facets and extreme functions described in the literature are piecewise linear
functions.\footnote{See \autoref{sec:def-complex-piecewise} for the definition that we
  use, which includes certain discontinuous functions.}  The number of slopes (i.e., the different values of the
derivative) of a function is a statistic that has received much attention in the
literature.  
In fact, one of the classic results in the study of extreme functions for the single-row problem is the following:\footnote{See \autoref{thm:k+1slope} for a general $k$-row result.}
\begin{theorem}[Gomory--Johnson 2-Slope Theorem \cite{infinite}]
  \label{th:2-slope}
  If a continuous piecewise linear minimal function of $R_f(\R,\Z)$ has
  only 2 values for the derivative wherever it exists (2 slopes), then the
  function is extreme.
\end{theorem}

Among the types of extreme functions that are piecewise linear functions,
there are discontinuous and continuous ones.  In the single-row case ($k=1$),
continuous piecewise linear extreme functions with 2, 3 and 4 different slopes
were previously known, and discontinuous piecewise linear extreme functions
with 1 and 2 slopes were previously known. Moreover, all previously known
examples of extreme discontinuous functions were continuous on one side of the
origin.
Hildebrand (2013, unpublished) found continuous piecewise linear
extreme functions with 5~slopes using computer-based search, as well as various
discontinuous piecewise linear extreme functions.  K\"oppe and Zhou
\cite{koeppe-zhou:extreme-search} later found continuous piecewise linear extreme
functions with up to 28~slopes. 

\begin{prop}[\TheoremNEWRESULT]\label{prop:5-slope} There exist continuous
  piecewise linear extreme functions with 5, 6, 7, and 28 slopes. There exist discontinuous piecewise linear extreme functions with 3 slopes and discontinuous piecewise linear extreme functions that are discontinuous on both sides at the origin. See~\autoref{tab:compendium-3}.
\end{prop}

This prompts the following question.

\begin{openquestion}
  For the single-row problem $R_f(\R, \Z)$, do there exist 
  continuous and discontinuous extreme functions with $s$ slopes for every $s \geq 2$?
\end{openquestion}

The additivity domain $E(\pi)$ for any minimal function $\pi$
(see~\eqref{eq:Epi}) can be decomposed as the union of its maximal convex
subsets. The first 5-slope functions found by Hildebrand (2013) have an
additivity domain which contains lower-dimensional maximal convex
components.\footnote{The functions are available in the electronic compendium \cite{electronic-compendium}
  as \sage{hildebrand\underscore{}5\underscore{}slope\dots}} 
This begs the question: 

\begin{openquestion}
  For the single-row problem $R_f(\R, \Z)$, do there exist continuous
  piecewise linear extreme functions of $R_f(\R, \Z)$ with $s$ slopes such
  that $E(\pi)$ is the union of full-dimensional convex sets for every $s \geq 2$?
\end{openquestion}
Not all facets and extreme functions are piecewise linear though.  Basu, Conforti, Cornu\'ejols, and Zambelli~\cite{bccz08222222}
constructed a family of facets that are not piecewise linear, yet
the derivatives (where they exist) only take 2 values; see \autoref{sec:bccz_counterexample}.  A function~$\tilde\pi$ from this family is absolutely continuous and therefore it is differentiable
almost everywhere (a.e.).  The derivative~$\tilde\pi'$ happens to take only
two different values a.e., so $\tilde\pi$ is a ``generalized 2-slope
function.'' This suggests the following refined version of Gomory and Johnson's original piecewise linear conjecture for extreme functions.

\begin{conjecture}\label{conj:non-piece}
For every absolutely continuous extreme function $\pi\colon \R\to \R$, the derivative $\pi'$ is a simple function. Thus, there exists a finite partition of $\R$ into measureable subsets $M_0, \ldots, M_t$ such that $M_0$ is of measure zero and $\pi'$ is constant over each of $M_1, \ldots, M_t$.
\end{conjecture}

The fact that the derivative of the counterexample from \cite{bccz08222222} happens to take only two different values a.e.\@ also gives rise to the following \emph{generalized 2-slope conjecture}. This conjecture would generalize Theorem~\ref{th:2-slope}.

\begin{conjecture}
  Let $\pi \colon \R \rightarrow \mathbb{R}$ be a minimal function that is absolutely continuous and whose derivative $\pi'$ only takes two values outside of a set of measure zero. Then $\pi$ is extreme.
\end{conjecture}

%


The key difficulty in answering the above questions is that the tools of functional
equations (such as the {\em Interval Lemma} as discussed in \autoref{sec:foundations}) no longer directly apply and new tools will most likely need to be employed for the resolution. Thus, there are still substantial questions left to be explored, even for the single-row ($k=1$) problem. 

Much less is known about the $k$-row problem $R_{\ve f}(\R^k, \Z^k)$ for
general $k$.  Dey and Richard~\cite{dey2} pioneered the construction of
extreme functions for the $k$-row problem.  Their {\em sequential-merge}
procedure constructs extreme functions and facets for $k\geq 2$
dimensions by combining extreme functions and facets for
smaller~$k$; see~\autoref{s:sequential-merge}. 
As mentioned earlier, a breakthrough was made when
\autoref{thm:k+1slope} was proved in~\cite[Theorem 1.7]{bhkm}, generalizing
Gomory and Johnson's single-row result (\autoref{th:2-slope}) to the general
$k$-row problem, giving a very general sufficient condition for extremality
and facetness. 

\section{The $k$-dimensional theory of piecewise linear minimal valid functions}\label{sec:piecewise}

\subsection{Polyhedral complexes and piecewise linear functions}\label{sec:def-complex-piecewise}

We introduce the notion of polyhedral complexes, which serves two purposes.
First, it provides a framework to define piecewise linear functions,
generalizing the familiar situation of functions of a single real variable. 
Second it is a tool for studying subadditivity and additivity relations of these functions.
This exposition follows \cite{bhk-IPCOext}.

\begin{definition}
\label{def:polyhedralComplex}
A (locally finite) {\em polyhedral complex} is a collection $\P$ of polyhedra in $\R^k$ such that:
\begin{enumerate}[\rm(i)]
\item $\emptyset \in \P$,
\item if $I \in \P$, then all faces of $I$ are in $\P$,
\item the intersection $I \cap J$ of two polyhedra $I,J \in \P$ is a face of both $I$ and $J$,
\item any compact subset of $\R^k$ intersects only finitely many faces in $\P$.
\end{enumerate}
A polyhedron $I$ from $\P$ is called a {\em face} of the complex.
A polyhedral complex~$\P$ is said to be {\em pure} if all its maximal faces (with
respect to set inclusion) have the same
dimension. In this case, we call the maximal faces of $\P$ the {\em cells}
of~$\P$. The zero-dimensional faces of $\P$ are called {\em vertices} and the set of vertices of $\P$ will be denoted by $\verts(\P)$. A polyhedral complex~$\P$ is said to be {\em complete} if the union of all faces of
the complex is~$\R^k$. A pure and complete polyhedral complex $\P$ is called a {\em triangulation} of $\R^k$ if every maximal cell is a simplex.  
\end{definition}

\begin{example}[Breakpoint intervals in $\R^1$ \cite{basu-hildebrand-koeppe:equivariant}]\label{ex:1d-breakpoint-complex}
  Let $0=x_0 < x_1 < \dots < x_{n-1} < x_n=1$ be a list of 
  ``breakpoints'' in $[0,1]$.  
  We extend it periodically as
  \begin{math}
    \B = \{\, x_0 + t, x_1 + t, \dots, x_{n-1}+t\st
    t\in\Z\,\}
  \end{math}.
  Define the set of 0-dimensional faces to be the collection of singletons,
  \begin{math}
    \I_{\B,\EquiOneDimPoint} = \bigl\{\, \{ x \} \st x\in B\,\bigr\},
  \end{math}
  and the set of one-dimensional faces to be the collection of closed intervals,
  \begin{math}
    \I_{\B,\EquiOneDimEdge} = \bigl\{\, [x_i+t, x_{i+1}+t] \st i=0, \dots, {n-1} \text{
      and } t\in\Z \,\bigr\}. 
  \end{math}
  Then $\I_{\B} = \{\emptyset\} \cup \I_{\B,\EquiOneDimPoint} \cup \I_{\B,\EquiOneDimEdge}$ is a
  locally finite 
  polyhedral
  complex
  .
\end{example}

\begin{example}[Standard triangulations of~$\R^2$ \cite{bhk-IPCOext}]\label{ex:2d-standard-triangulation}
   Let $q$ be a positive integer.  Consider the
arrangement~$\mathcal H_q$ of all hyperplanes (lines) of~$\R^2$ of the form
$\ColVec{0}{1}\cdot \x = b$, $\ColVec{1}{0}\cdot \x = b$, and $\ColVec{1}{1}\cdot\x  = b$,
where $b \in \tfrac{1}{q}\Z$.  The complement of the arrangement~$\mathcal
H_q$ consists of two-dimensional cells, whose closures are the triangles
$$\FundaTriangleLower = \tfrac1q \conv(\{ \ColVec{0}{0}
, \ColVec{1}{0}
, \ColVec{0}{1}
\})\qquad\text{and}\qquad \FundaTriangleUpper = \tfrac1q \conv(\{\ColVec{1}{0}
, \ColVec{0}{1}
,
\ColVec{1}{1}
\})$$ and their translates by elements of the lattice $\smash[t]{\frac1q\Z^2}$. 
We denote by $\P_q$ the collection of these triangles and the vertices and
edges that arise as intersections of the triangles, and the empty set.  Thus $\P_q$ is a locally finite
polyhedral complex
.  Since all
nonempty faces of~$\P_q$ are simplices, it is a triangulation of the
space~$\R^2$.
\end{example}

We give a precise definition of affine linear functions over a domain, suitable for
the general $k$-dimensional case.
\begin{definition}
Let $U \subseteq \R^k$. We say $\pi\colon U \to \R$ is {\em affine (or affine linear)} over $U$
if there exists a \emph{gradient} $\cve \in \R^k$ such that 
for any $\u_1, \u_2 \in U$ we have $$\pi(\u_2) - \pi(\u_1) =  \cve\cdot (\u_2 - \u_1).$$ 
\end{definition}

Given a pure and complete polyhedral complex $\P$, we call a function $\pi\colon \R^k\to \R$ {\em piecewise linear} over $\P$ if it is affine linear over the {\em relative interior} of each face of the complex. Under this definition, piecewise linear functions can be discontinuous. We say the function $\pi$ is \emph{continuous piecewise linear over $\P$} if it is affine over each of the cells of~$\P$ (thus automatically imposing continuity). Most of the results presented in this survey will be about continuous piecewise linear functions. %
%


Motivated by Gomory--Johnson's characterization of minimal valid functions
(\autoref{thm:minimal}), we are interested in functions~$\pi\colon \R^k\to \R$
that are periodic modulo~$\Z^k$, i.e., for all $\ve x\in \R^k$ and all vectors 
$\ve t\in\Z^k$, we have $\pi(\ve x+\ve t) = \pi(\ve x)$.  If $\pi$ is periodic modulo~$\Z^k$
and continuous piecewise linear over a pure and complete complex~$\P$, then we
can assume without loss of generality that $\P$ is
also \emph{periodic modulo~$\Z^k$}, i.e., for all $I\in \P$ and all vectors
$\ve t\in\Z^k$, the translated polyhedron~$I + \ve t$ also is a face of~$\P$.
This is the case in Examples \ref{ex:1d-breakpoint-complex} and~\ref{ex:2d-standard-triangulation}.

\begin{remark}\label{rem:p-finite} If all the cells of the polyhedral complex
  are bounded, the value of a continuous piecewise linear function at any
  point $\x$ can be obtained by interpolating the values of the function at
  the vertices of the minimal face containing $\x$. This is utilized in
  \autoref{s:interpolation}. The assumption of boundedness of the cells can be
  made without loss of generality; see
  \autoref{sec:gen-k-functions}. Moreover, for a periodic continuous piecewise
  linear function over a periodic complex, we can give a finite description
  for $\pi$ by further restricting to the values in $\verts(\P) \cap  D$ where
  $ D = [0,1]^k$ or any set such that $ D + \Z^k = \R^k$. The finiteness of
  the set $\verts(\P) \cap  D$ is guaranteed by the assumption of local
  finiteness in \autoref{def:polyhedralComplex} (iv).
\end{remark}

\subsection{The extended complex $\Delta \P$} \label{section:delta-p-definition}

 For any $I,J,K\subseteq \R^k$, we define the set
\begin{equation}\label{eq:F-def}
F(I,J,K) = \setcond{(\x,\y) \in \R^k \times \R^k}{\x \in I,\, \y \in J,\, \x + \y \in K}.
\end{equation}   
When $I,J,K$ are polyhedra, $F(I,J,K)$ is also a polyhedron. Let $\P$ be a pure, complete polyhedral complex of $\R^k$ and let $\pi$ be a continuous piecewise
linear function over $\P$. In order to study the additivity domain $E(\pi)$, 
we define the family of polyhedra in $\R^k \times \R^k$,
$$\Delta\P =
\setcond{F(I,J,K)}{I, J, K \in \P},$$
which is also polyhedral complex
{\cite[\autoref{equi3:lemma:delta-p-is-complex}]{bhk-IPCOext}}; see \autoref{fig:delta-p}.
\begin{figure}[t]
  \centering
  \includegraphics[width=.44\linewidth]{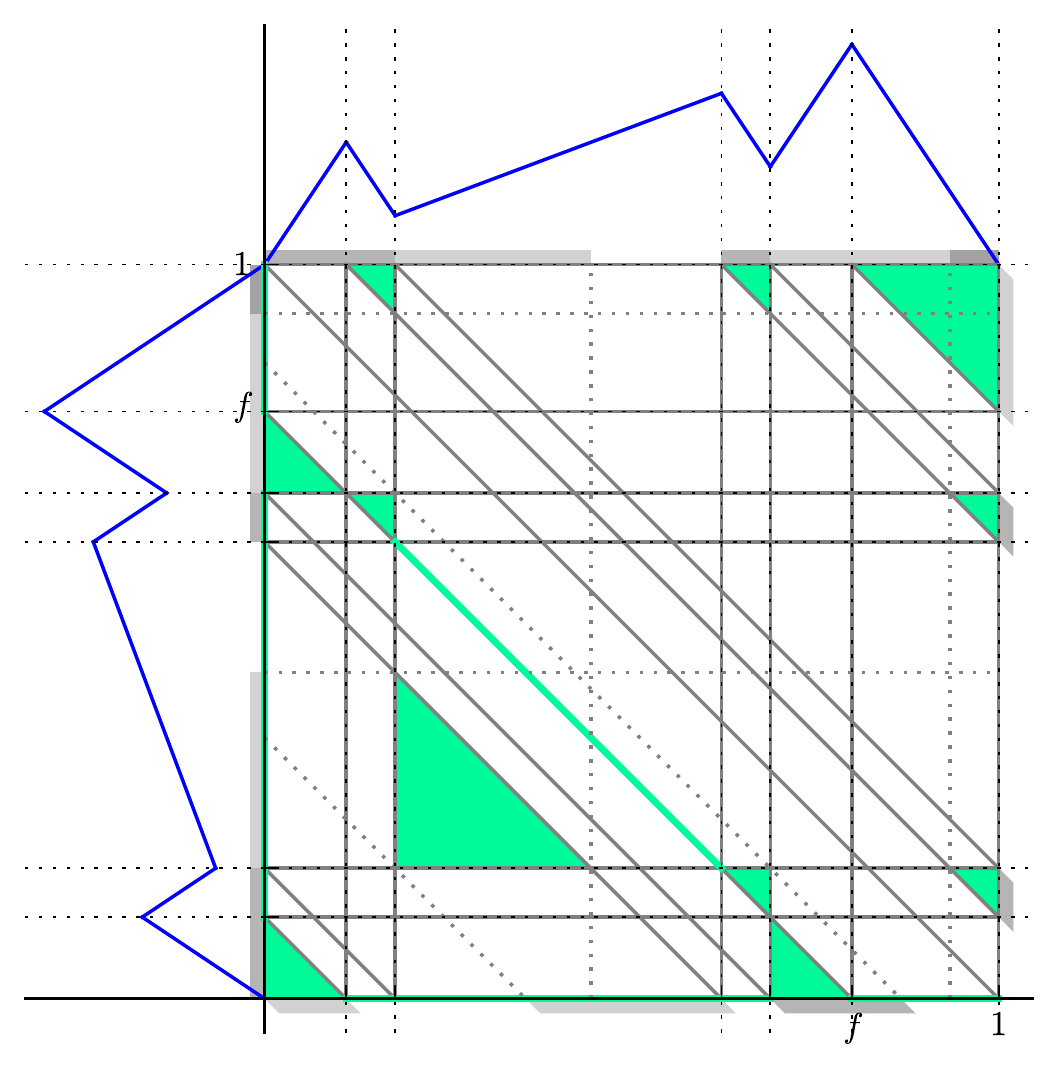}\quad
  \includegraphics[width=.44\linewidth]{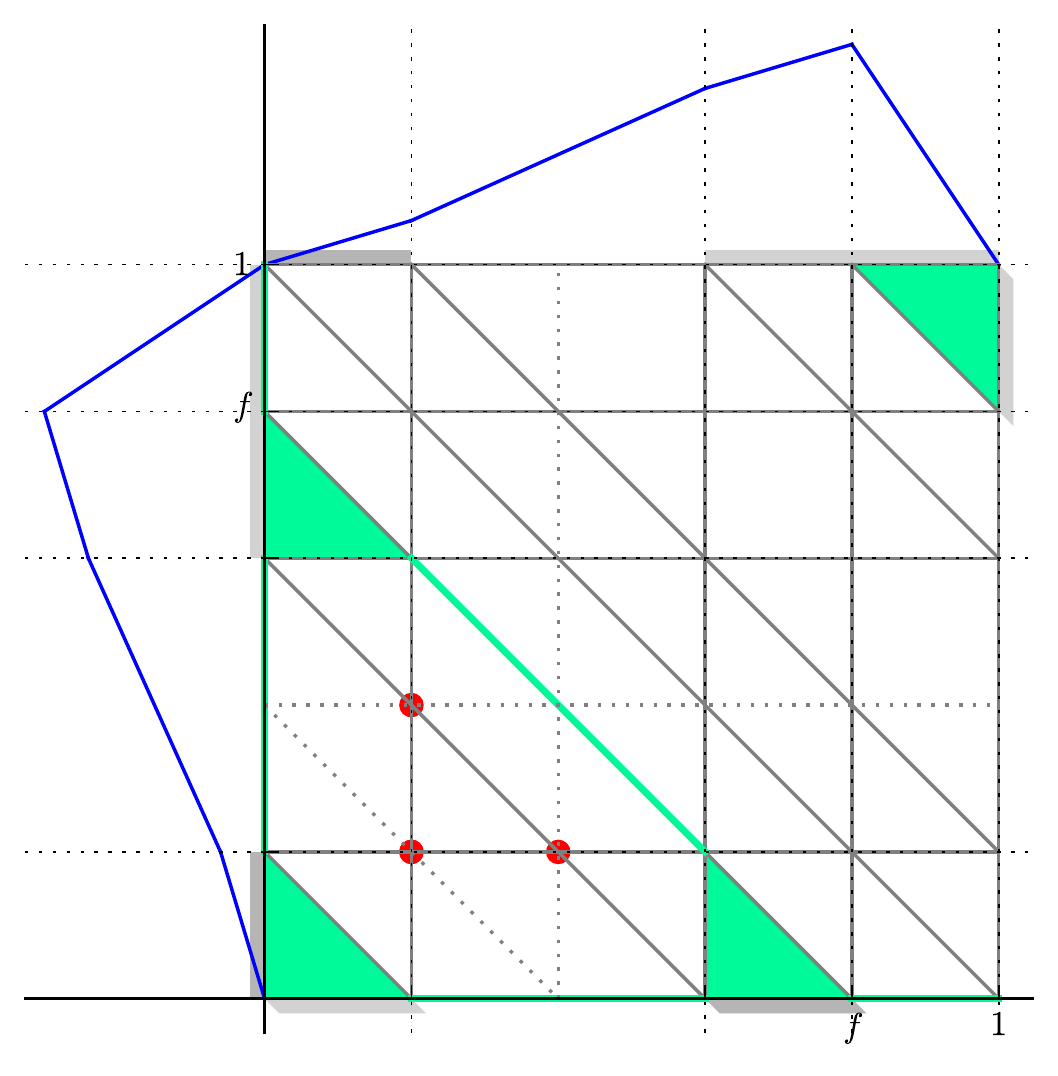}
  \caption{
    Two 
    diagrams of a function (\emph{blue graphs on the top and the left}) and its polyhedral complex $\Delta\P$ (\emph{gray
      solid lines}), as plotted by the command
    \sage{\sagefunc{plot_2d_diagram}(h)}.
    \emph{Left}, \sage{h = \sagefunc{gj_forward_3_slope}()} \emph{(left)}.
    \emph{Right},
    \sage{h = \sagefunc{not_minimal_2}()}.
    The set $E(\pi)$ in both cases is the union of the faces shaded in green.
    The \emph{heavy diagonal green line} $x + y = f$ 
    corresponds to the symmetry
    condition.  
    Vertices of $\Delta\P$ do not necessarily project (\emph{dotted gray lines}) to breakpoints; 
    compare with~\autoref{fig:uniform-spacing}.
    Vertices of the complex on which $\Delta\pi < 0$, i.e., subadditivity is
    violated, are shown as \emph{red dots}; see \autoref{minimality-check}.
    At the borders of each diagram, the projections $p_i(F)$ of two-dimensional
    additive faces are shown as \emph{gray shadows}: $p_1(F)$ at the top border, $p_2(F)$ at
    the left border, $p_3(F)$ at the bottom and the right borders.
  }
  \label{fig:delta-p}
\end{figure}

\begin{figure}[t]
\centering
\scalebox{.7}{\ifpdf
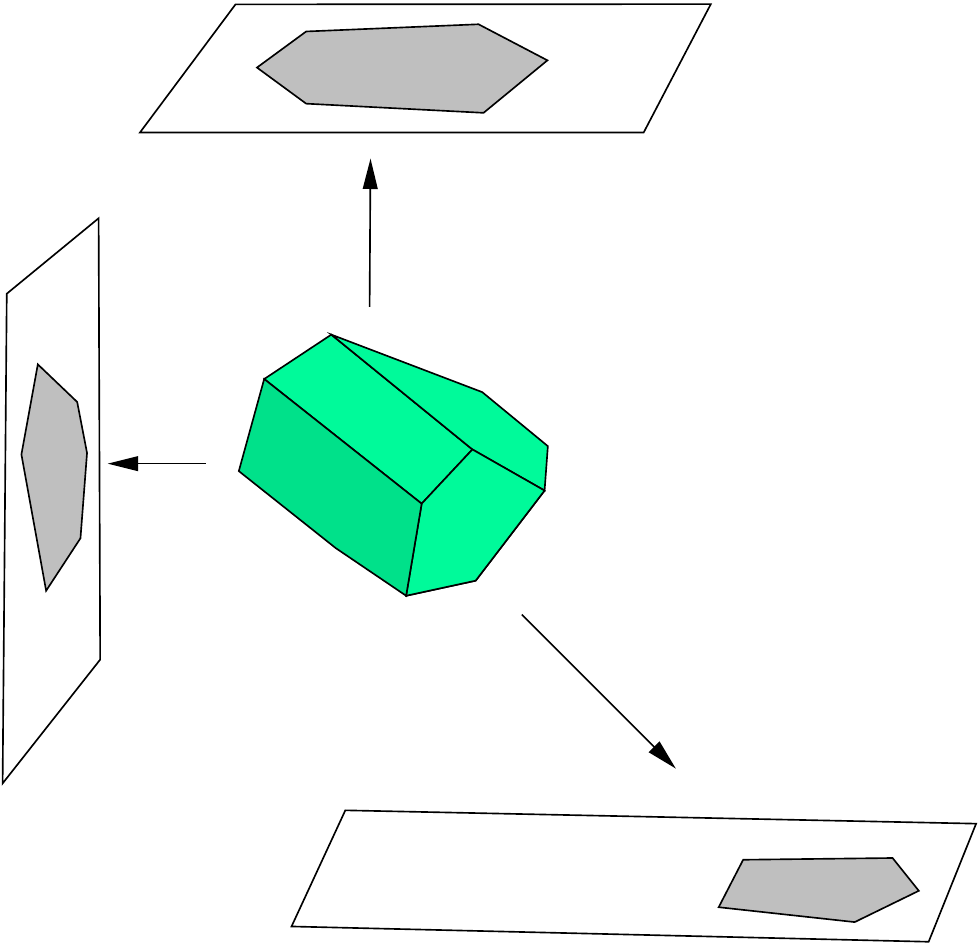
\else
\input{figureProjection.pstex_t}
\fi}
\caption{A face $F=F(I,J,K)$ and its projections $I' = p_1(F)$, $J' = p_2(F)$,
  $K' = p_3(F)$. This is an abstract picture; note that if $I',J',K' \subseteq
  \R^2$ are full-dimensional, then $F(I,J,K)$ is actually a full-dimensional set of $\R^4$.}
\label{figProjection}
\end{figure}

Define the projections $p_1,p_2,p_3\colon \R^k\times \R^k \to \R^k$ as
\begin{equation}
\label{eq:projections}
p_1(\x,\y) = \x, \quad p_2(\x,\y) = \y, \quad  p_3(\x,\y) = \x+\y; 
\end{equation}
see \autoref{figProjection}.
Now let $I,J,K \subseteq \R^k$ and let $F = F(I,J,K)$.
Simple formulas for the projections of $F$ are available \cite[\autoref{equi3:prop:projection}]{bhk-IPCOext}:
\begin{subequations}
  \begin{alignat}{2}
    I' := p_1(F(I,J,K)) &= (K + (-J)) \cap I && \subseteq I, \\
    J' := p_2(F(I,J,K)) &= (K+ (-I)) \cap J && \subseteq J, \\
    K' := p_3(F(I,J,K)) &= (I+J) \cap K && \subseteq K.
  \end{alignat}
\end{subequations}
The inclusions $I' \subseteq I$, $J' \subseteq J$, $K' \subseteq K$  may be strict. 
This possibility is illustrated by the largest shaded triangle in
\autoref{fig:delta-p} (left). 
We see that the projections $I'$, $J'$, $K'$ give us a canonical, minimal way of
representing~$F$ as $F(I', J', K')$ \cite[\autoref{equi3:lemma:setsIJKF}]{bhk-IPCOext}.
Note that $I'$, $J'$, $K'$ are not faces of~$\P$ in general, even if $I$, $J$,
$K$ were faces; see again \autoref{fig:delta-p} (left). 

We will study the function $\Delta\pi\colon \R^k \times \R^k \to \R$, as defined in~\eqref{eq:delta-pi-definition}, which measures the slack in the subadditivity constraints.  When $\pi$ is continuous piecewise linear over $\P$, we have that $\Delta \pi$ is continuous piecewise linear over $\Delta \P$ (\autoref{equi3:lem:delta-pi-cts} in \cite{bhk-IPCOext}).


\begin{remark}\label{rem:delta-p-finite}
  If $\pi$ and $\P$ are periodic modulo~$\Z^k$, then 
  $\Delta\pi$ and $\Delta\P$ are periodic modulo~$\Z^k\times\Z^k$.  
 Echoing \autoref{rem:p-finite}, one can make the description of $\Delta\pi$ finite by recording the values of $\Delta\pi$ on a smaller set;
for example, the set $\verts(\Delta\P) \cap ([0,1]^k \times [0,1]^k)$. One may also replace $[0,1]^k\times [0,1]^k$ by $ D \times  D$ for any $ D$ satisfying $ D + \Z^k = \R^k$.
\end{remark}

\subsection{Genuinely $k$-dimensional functions}\label{sec:gen-k-functions}

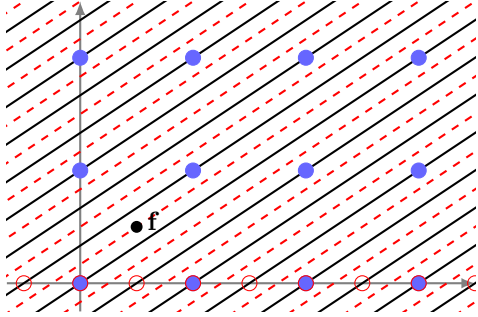
\begin{figure}[tp]
  \centering
  \begin{tikzpicture}[scale = 1.5]


    \def\xminAxis{-0.65}
    \def\xmaxAxis{3.5}
    
    \def\yminAxis{-0.25}
    \def\ymaxAxis{2.5}
    
    \def\xmin{-1}
    \def\xmax{5}
    
    \def\ymin{-2}
    \def\ymax{5}
    
    \def\latticePointSize{2}
    
    \begin{scope}
    \clip (\xminAxis,\yminAxis) -- (\xmaxAxis,\yminAxis) -- (\xmaxAxis,\ymaxAxis) -- (\xminAxis,\ymaxAxis) -- (\xminAxis,\yminAxis);

    \coordinate (Origin)   at (0,0);
    \coordinate (XAxisMin) at (\xminAxis,0);
    \coordinate (XAxisMax) at (\xmaxAxis,0);
    \coordinate (YAxisMin) at (0,\yminAxis);
    \coordinate (YAxisMax) at (0,\ymaxAxis);

    \draw [thick, gray,-latex] (XAxisMin) -- (XAxisMax);
    \draw [thick, gray,-latex] (YAxisMin) -- (YAxisMax);

     \foreach \x in {-15,...,6}{
            \pgfmathtruncatemacro{\y}{\x+15}%
           
             \draw[black,thick] (\x,-4) -- (\y,6);           
            \draw[black, thick] (\x,-5) -- (\y,5);   
            
             \draw[red, dashed, thick]   (\x,-5.5) -- (\y,4.5);   
             \draw[red, dashed, thick]   (\x,-4.5) -- (\y,5.5);   
     }
       
    \foreach \x in {\xmin,...,\xmax}{
      \foreach \y in {\ymin,...,\ymax}{
      
       		 \node[draw=blue!60,circle,inner sep=\latticePointSize pt,fill=blue!60] at (\x,\y) {};
      }
    }
    
            \foreach \x in {\xmin,...,\xmax}{
    		 \node[draw=red,circle,inner sep=\latticePointSize pt] at (\x,0) {};
    		}
          \foreach \x in {-3.5,...,7.5}{
                  \node[draw=red,circle,inner sep=\latticePointSize pt] at (\x,0) {};
                }

     \node at (0.65,0.55) {$\f$};
    \fill [black] (.5,.5) circle (1.5pt);
    
  \end{scope}  
  
%
             

  \end{tikzpicture}
  \caption{A piecewise linear function $\pi\colon \R^2 \to \R$ defined by
    interpolation between values of $0$ in on the black solid lines and $1$ on
    the dashed red lines.  The blue dots depict the lattice $\Z^2$. In
    particular, $\pi(x,0) = \min(4x, 2-4x)$ for $x \in [0,\tfrac{1}{2}]$.
    This is extended to the $x$-axis by $\pi(x,0) = \pi(x \bmod
    {\tfrac{1}{2}}, 0)$.  
    The points $(x,0)$ for $x\in\frac12\Z$ are shown as red circles.
    Finally, we can write $\pi(x,y) =\pi(\frac12(2x -3y),0)$ for
    all $(x,y) \in \R^2$.  
    This function is not genuinely two-dimensional, which is demonstrated by 
    a function $\phi\colon\R\to\R$ and a linear map $T\colon \R^2\to\R$ such
    that $\pi = \phi \circ T$.  Many choices for this pair $\phi$, $T$ are
    possible.  For $\phi(t) = \pi(\frac12t,0)$ and $T(x,y) = 2x - 3y$, 
    we have $T\Z^2 = \Z$, which satisfies the conditions in
    \autoref{prop:dim-reduction}.}
  \label{fig:not-gen-2-dim}
\end{figure}

In this subsection we show that when analyzing minimal functions it suffices to consider ``full-dimensional'' minimal functions. We formalize this in the following definition and proposition.


\begin{definition}\label{def:genk}
A function $\pi\colon \R^k \rightarrow \R$ is {\em genuinely $k$-dimensional}
if there does not exist a function $\phi \colon \R^{k-1} \rightarrow \R$ and a linear
map $T \colon \R^k \rightarrow \R^{k-1}$ such that $\pi = \phi\circ T$. 
\end{definition}

An example of a function that is not genuinely $k$-dimensional is described in Figure~\ref{fig:not-gen-2-dim}.

\begin{prop}[Dimension reduction; {\cite[\autoref{equi3:prop:dim-reduction}]{bhk-IPCOext}}]\label{prop:dim-reduction}
Let $\P$ be a pure and complete polyhedral complex in $\R^k$ that is periodic modulo $\Z^k$. Let $\pi\colon\R^k \to \R$ be a continuous piecewise linear function over $\P$, such that $\pi$ is nonnegative, subadditive, periodic modulo $\Z^k$ and $\pi(\0) = 0$. If $\pi$ is not genuinely $k$-dimensional, then there exists a natural number $0 \leq \ell < k$, a pure and complete polyhedral complex $\mathcal{X}$ in $\R^\ell$ that is periodic modulo $\Z^\ell$, a nonnegative and subadditive function $\phi\colon\R^\ell \to \R$ that is continuous piecewise linear over $\mathcal{X}$, and a point $\f' \in \R^\ell \setminus \Z^\ell$ with the following properties:
\begin{enumerate}
\item $\pi$ is minimal for $R_\f(\R^k, \Z^k)$ if and only if $\phi$ is minimal for $R_{\f'}(\R^\ell, \Z^\ell)$. 
\item $\pi$ is extreme for $R_\f(\R^k, \Z^k)$ if and only if $\phi$ is extreme for $R_{\f'}(\R^\ell, \Z^\ell)$.
\end{enumerate}
\end{prop}

The above idea first appears in~\cite[Construction 6.3]{dey3}, where the authors give a construction to obtain two-dimensional minimal functions from one-dimensional minimal functions, and show that all minimal functions for $k=2$ with $2$ slopes can be obtained using such a construction~\cite[Theorem 6.4]{dey3}. The construction is exactly via the use of a linear map as described in~\autoref{def:genk}. In fact, their result is a special case of \autoref{prop:dim-reduction} and the simple observation that subadditive, genuinely $k$-dimensional functions have at least $k+1$ slopes or gradient values (see also the conclusion of Theorem~\ref{thm:k+1slope}).
%

\begin{remark}[Dimension reduction;  {\cite[\autoref{equi3:remark:dimension-reduction}]{bhk-IPCOext}}]
\label{remark:dimension-reduction}

Using Proposition~\ref{prop:dim-reduction}, the extremality/minimality question for $\pi$ that is not genuinely $k$-dimensional can be reduced to the same question for a lower-dimensional genuinely $\ell$-dimensional function with $\ell < k$. When $\P$ is a rational polyhedral complex, this reduction can be done algorithmically. 
\end{remark}

Next, we show that genuinely $k$-dimensional functions that are continuous piecewise linear enjoy several regularity properties which can often simplify the investigation of minimal valid functions that are continuous piecewise linear functions.

%
\begin{theorem}[{\cite[\autoref{equi3:thm:faVertex-gen-k-dim}]{bhk-IPCOext}}]
\label{thm:faVertex-gen-k-dim}
Let $\P$ be a pure and complete polyhedral complex in $\R^k$ that is periodic modulo $\Z^k$. Let $\theta \colon \R^k \rightarrow \R$ be a minimal valid function for $R_\f(\R^k, \Z^k)$ that is continuous piecewise linear over $\P$, and is genuinely $k$-dimensional.  Then,
\begin{itemize}
\item[(i)] $\f \in \verts(\P)$.
\item[(ii)] The cells of $\P$ and $\Delta \P$ are full-dimensional polytopes. 
\end{itemize}
\end{theorem}

\subsection{Finite test for minimality}\label{section:minimalityTest}
  
One of the main advantages of working with minimal valid functions that are piecewise linear is their combinatorial structure, which avoids many analytical complexities. Moreover, it is possible to give a finite description of~$\pi$.  For example, it suffices to know the values of $\pi$ on the unit hypercube $ D = [0,1]^k$, which can in turn be broken into a finite number of polytopes over which $\pi$ is simply an affine function. Of course, any choice of $ D$ such that $ D + \Z^k = \R^k$ suffices to obtain such a finite description, and $ D = [0,1]^k$ is just one such choice. In certain situations, other choices of $ D$ may be more natural, and provide a shorter description. 


%
%

By Theorem~\ref{thm:minimal}, we can test whether a periodic function is minimal by
testing subadditivity, the symmetry condition, and the value at the origin. These properties are easy to test when the function is continuous piecewise linear.  The first of such tests came from Gomory and Johnson \cite[Theorem
7]{tspace} for the case $k=1$.\footnote{Note that in \cite{tspace}, the word ``minimal'' needs to be replaced by ``satisfies the symmetry condition'' throughout the statement of their
  theorem and its proof.} 
Richard, Li, and Miller \cite[Theorem
22]{Richard-Li-Miller-2009:Approximate-Liftings} extended it to the case of
discontinuous piecewise linear functions.\footnote{They present it in a setting
  of pseudo-periodic superadditive functions, rather than periodic subadditive
  functions.
} A test for subadditivity of continuous piecewise linear functions for the two-row problem was given in~\cite[Proposition 10]{dey3} that reduces to testing subadditivity at vertices, edges, and the so-called {\em supplemental vertices.}
We present a minimality test for continuous piecewise linear functions for
general $k$.  To simplify notation, we restrict ourselves to the continuous
case.\footnote{A discontinuous version of \autoref{minimality-check}
  appears in \cite[Theorem 2.5]{basu-hildebrand-koeppe:equivariant}, where it
  is stated for the case $k=1$; it extends verbatim to general~$k$.  
  All relevant limits of the function at discontinuities are taken care of by testing
  \begin{equation}
    \Delta\pi_F(u,v) = \lim_{\substack{(x,y) \to (u,v)\\ (x,y) \in \relint(F)}} \Delta\pi(x,y)
  \end{equation}
  for all faces $F\in\Delta\P$ that contain the vertex $(u,v)$.
  For $k=1$, by analyzing the possible faces~$F$, one recovers the explicit
  limit relations stated in \cite[Theorem
  22]{Richard-Li-Miller-2009:Approximate-Liftings}.
}
The test is stated in terms of the set of vertices
$\verts(\Delta\P)$ of the complex~$\Delta\P$; see again \autoref{fig:delta-p}
for an illustration. This uses the observation made in \autoref{rem:p-finite} that the function values for a continuous piecewise linear function can be obtained by interpolating the values at $\verts(\P)$.\footnote{A different approach is taken in~\cite[Proposition 10]{dey3} where the subadditivity test uses so-called {\em supplemental vertices} which are introduced to get around the problem of unbounded cells.}

%


\begin{theorem}[Minimality test {\cite[\autoref{equi3:minimality-check}, \autoref{equi3:rem:f-break-point}]{bhk-IPCOext}}]
\label{minimality-check}
Let $\P$ be a pure, complete, polyhedral complex  in $\R^k$ that is periodic
modulo $\Z^k$ and every cell of $\P$ is bounded.%
\footnote{This is not restrictive due to \autoref{thm:faVertex-gen-k-dim}(ii)
  and \autoref{prop:dim-reduction}(1).}
Let $\Delta D = [0,1]^k \times
[0,1]^k$.\footnote{Instead of $\Delta D = [0,1]^k \times [0,1]^k$, one can choose $\Delta  D =  D \times  D$ for any $ D$ such that $ D + \Z^k = \R^k$; see the discussion in~\cite{bhk-IPCOext}.}
Let $\pi\colon \R^k
\to\R$ be a nonnegative 
continuous
piecewise linear function over $\P$ that is periodic
modulo $\Z^k$.  Let $\f \in \verts(\P)$.\footnote{
  For $k=1$, necessarily $\f\in \verts(\P)$ \cite[Lemma
  2.4]{basu-hildebrand-koeppe:equivariant}.  The same is true 
  for genuinely $k$-dimensional functions
  (\autoref{thm:faVertex-gen-k-dim}). 
  If, however, $\f \notin \verts(\P)$, then the condition
  \eqref{symmetry-test} in the symmetry test must be replaced by a slightly more complicated condition
  (as stated in \cite[\autoref{equi3:minimality-check}, \autoref{equi3:rem:f-break-point}]{bhk-IPCOext}). 
  Let $S = \{\,(\u,\v)\st \u + \v \equiv \f \textstyle\pmod{\ve1}\,\}$.
  Then $\Delta\P\cap S := \{\, F\cap S: F \in \Delta\P\, \}$ is again a
  polyhedral complex.  The condition~\eqref{symmetry-test} is then replaced by:
  $$\Delta\pi(\u,\v) = 0 \quad\text{for all}\quad
  (\u,\v)\in \Delta D \cap \verts(\Delta \P \cap S).$$
}
Then $\pi$ is minimal for $R_\f(\R^k, \Z^k)$ if and only if the following conditions hold:
\begin{enumerate}
\item $\pi(\0) = 0$,
\item Subadditivity test: $\Delta\pi(\u,\v) \geq 0$ for all $(\u,\v) \in \Delta D \cap \verts(\Delta \P)$.
\item Symmetry test: $\pi(\f) = 1$ and 
  \begin{equation}
    \Delta\pi(\u,\v) = 0 \quad\text{for all}\quad
    (\u,\v)\in \Delta D \cap \verts(\Delta \P)\quad\text{with}\quad \u + \v
    \equiv \f \textstyle\pmod{\Z^k}.\label{symmetry-test}
  \end{equation}
  Here $({\bmod}\ \Z^k)$ denotes componentwise equivalence modulo~$1$. 
\end{enumerate}
\end{theorem}



\subsection{Combinatorializing the additivity domain}
\label{section:additivity-discretized}


Let $\pi \colon \R^k \to \R$ be a continuous piecewise linear function over a  pure, complete polyhedral complex $\P$.  
Recall the definition of  the \emph{additivity domain} of~$\pi$,
\begin{displaymath}
  E(\pi) 
  = \setcond{(\x, \y)}  {\Delta \pi(\x, \y) = 0}.
\end{displaymath}
We now give a combinatorial representation of this set
using the faces of~$\P$. Let 
$$
E(\pi, \P) 
= \setcond{F \in \Delta \P}{\Delta\pi|_F \equiv 0}.
$$
We consider $E(\pi, \P)$ to include $F=\emptyset$, on which $\Delta\pi|_F \equiv 0$
holds trivially.  Then $E(\pi, \P)$ is another polyhedral complex, a
subcomplex of $\Delta\P$.   
As mentioned, if $\pi$ is continuous, then $\Delta \pi$ is continuous.
Under this continuity assumption, we can consider only the set of maximal faces in $E(\pi, \P)$.  We define 
$$
E_{\max{}}(\pi,\P) = \bigl\{\,F\in E(\pi, \P)\bigst F \text{ is a maximal face by set inclusion in } E(\pi, \P)\,\bigr\}. 
$$

\begin{lemma}[{\cite[\autoref{equi3:lemma:covered-by-maximal-valid-triples}]{bhk-IPCOext}}]\label{lemma:covered-by-maximal-valid-triples}
\begin{equation*}
E(\pi) = \bigcup \{ F  \in E(\pi, \P)  \}
= \bigcup \{ F \in E_{\max{}}(\pi, \P)  \}.
\end{equation*}
\end{lemma}

This combinatorial
representation can then be made finite by choosing representatives 
as in \autoref{rem:delta-p-finite}.

\subsection{Perturbation functions}\label{subsec:perturbations}

We now discuss how to prove that a given minimal function is not a facet or not extreme. We consider the \emph{space of perturbation functions} with prescribed additivities $E\subseteq G \times G$
 \begin{equation}
 \arraycolsep=1.4pt
\bar \Pi^E(G,S) =  \left\{ \bar \pi\colon G \to \R \, \Bigg| \,
\begin{array}{rcll}
\bpi(\0) &=& 0\\
\bpi(\f) &=& 0\\
\bpi(\x) + \bpi(\y) &=& \bpi(\x + \y) & \text{ for all } (\x,\y) \in E\\
\bpi(\x) &=& \bpi(\x + \t) & \text{ for all } \x \in G,\, \t \in S
 \end{array} \right\}.
 \end{equation}
Later we will use this notation even if $G$ is not a group and only require that $\0,\f \in G$, and $S \subseteq G$.  Clearly $\bar \Pi^E(G,S)$ is a linear space. 

The third condition implies that $E \subseteq E(\bpi)$ for all $\bar \pi \in \bar \Pi^E(G,S)$. From Lemma~\ref{lem:minimality-of-pi1-pi2} it follows that $\pi$ is not extreme if and only if there exists a $\bar \pi \in  \Pi^{E(\pi)}(G,S) \setminus\{0\}$ such that $\pi^1 = \pi + \bar \pi$ and $\pi^2 = \pi - \bar \pi$ are minimal valid functions.  In a similar vein, 
if $\pi$ is not a facet of $R_\f(G,S)$, then  by the Facet Theorem, Theorem~\ref{thm:facet},  there exists a nontrivial $\bar \pi \in \bar \Pi^{E(\pi)}(G,S)$ such that $\pi' = \pi + \bar \pi$ is a minimal valid function.   Note that this last statement is not an if and only if statement.
%

Suppose $\pi$ is piecewise linear on a polyhedral complex $\P$.  We will often consider a refinement $\P'$ of $\P$ on which we can find a continuous piecewise linear perturbation $\bar \pi$ such that $\pi$ is not extreme.

The basic idea is that if one can find a non-zero function $\bar\pi$ in the linear
subspace of functions $\bar \Pi^{E(\pi)}(\R^k, \Z^k)$ then the finite, combinatorial
description of $\Delta\pi$ (since $\pi$ and therefore $\Delta \pi$ is
piecewise linear) 
allows small perturbations from $\pi$ in the direction of $\bar\pi$ while maintaining minimality.

\begin{theorem}[Perturbation {\cite[\autoref{equi3:corPerturb}]{bhk-IPCOext}}]
\label{corPerturb}
Let $\P$ be a pure, complete, polyhedral complex  in $\R^k$ that is periodic modulo $\Z^k$ and every cell of $\P$ is bounded.
Suppose $\pi$ is minimal and continuous piecewise linear over $\P$.  Suppose
$\bar \pi \not\equiv 0$ is continuous piecewise linear over a refinement $\P'$ of $\P$, is periodic
modulo $\Z^k$ and satisfies $\bar \pi \in \bar \Pi^{E}(\R^k, \Z^k)$ where $E = E(\pi)$.  Then $\pi$ is not extreme. Furthermore, given $\bar\pi$, there
exists an $\epsilon > 0$ such that $\pi^1 = \pi + \epsilon \bar \pi$ and
$\pi^2 = \pi - \epsilon \bar \pi$ are distinct minimal functions that are
continuous piecewise linear over $\P$ such that $\pi = \tfrac12(\pi^1 +
\pi^2)$. 
\end{theorem}


When $\pi$ is a continuous piecewise linear function over a polyhedral complex $\P$, for certain refinements $\T$ of $\P$ we can decompose perturbation functions $\bar \pi$ into piecewise linear perturbations over $\T$ and other perturbations that vanish on the vertices of $\T$.  
For a triangulation $\T$ define the vector spaces
\begin{align*}
  \bar \Pi^E_{\T}(\R^k, \Z^k) &:= \{\, \bar \pi \in \bar \Pi^E(\R^k, \Z^k) \st
  \bar \pi \text{ is continuous piecewise linear on } {\T}\,\}\\
  \intertext{and}
  \bar \Pi^E_{\textrm{zero}(\T)}(\R^k, \Z^k) &:= \{\, \bar \pi \in \bar \Pi^E(\R^k,
  \Z^k) \st \bar \pi|_{\verts(\T)} \equiv 0\,\}.
\end{align*}

\begin{lemma}[\TheoremNEWRESULT]
\label{lem:tri-restriction}
Suppose $\pi \colon \R^k \to \R$ is a minimal valid function that is piecewise linear over $\P$.  Suppose $\T$ is a triangulation of $\R^k$ such that there exists $q \in \mathbb{N}$ such that $\verts(\T) = \tfrac{1}{q}\Z^k$ and $p_i(\verts(\Delta \T)) \subseteq \tfrac{1}{q} \Z^k$ for $i=1,2,3$ and $\f \in \tfrac{1}{q}\Z^k$.  Let $E = E(\pi)$ and $E' = E(\pi) \cap \verts(\Delta \T)$, and suppose $\T$ is a refinement of $\P$.
\begin{enumerate}
\item 
\label{lem:tri-restriction-part1}
$\bar \pi \in \bar \Pi^E_\T(\R^k, \Z^k)$ if and only if $\bar \pi|_{\frac{1}{q} \Z^k} \in \bar \Pi^{E'}( \tfrac{1}{q}\Z^k, \Z^k)$,
\item 
\label{lem:tri-restriction-part2}
For every $\bar \pi \in \bar\Pi^E(\R^k, \Z^k)$, there exist unique $\bar \pi_\T \in \Pi^E_\T(\R^k, \Z^k)$ and $\bar \pi_{\mathrm{zero}(\T)} \in \bPizero$ such that 
$$
\bar \pi = \bar \pi_\T + \bar \pi_{\mathrm{zero}(\T)}.
$$ 
\end{enumerate}
\end{lemma}
\begin{proof}
Let $\bar \pi_\T$ be a continuous piecewise linear function over $\T$.
Since $\T$ is a refinement of $\P$, we have that $\pi$ is continuous piecewise linear over $\T$ as well.  
By Lemma~\ref{lemma:covered-by-maximal-valid-triples}, for any $\varphi$ that is continuous piecewise linear on $\T$ we have that $E(\varphi) = \bigcup \{\, F \in \Delta \T \st \Delta \varphi|_F \equiv 0\,\}$.  
Since $\Delta \varphi$ is affine on $F$, we have that  $\varphi|_F \equiv 0$ if and only if $\varphi|_{\verts(F)} \equiv 0$.  
Therefore, it follows that $E(\pi) \subseteq E(\bar \pi)$ if and only if $\Delta \pi|_{\verts(F)} \equiv 0$ implies that $\Delta \bar \pi|_{\verts(F)} \equiv 0$ for all $F \in \Delta \T$.  
Since $\verts(\T) = \tfrac{1}{q} \Z^k$, this establishes part~(1).  

Next, let $\bar \pi \in \bar \Pi^E(\R^k, \Z^k)$.   Let $\bar \pi_\T$ be the unique extension of $\bar \pi|_{\frac{1}{q}\Z^k}$ to $\R^k$ via the triangulation $\T$.  Note that $\bar \pi_\T$ is the unique piecewise linear function over $\T$ such that $(\bar \pi - \bar \pi_\T)|_{\frac{1}{q} \Z^k} \equiv 0$.  Define $\bar \pi_{\mathrm{zero}(\T)} = \bar \pi - \bar \pi_\T$.  It is left to show that $\bar \pi_\T, \bar \pi_{\mathrm{zero}(\T)} \in \bar \Pi^E(\R^k, \Z^k)$.  

Since $\bar \pi \in  \bar \Pi^E(\R^k, \Z^k)$, it follows that $\bar \pi_\T|_{\frac{1}{q}\Z^k} = \bar \pi|_{\frac{1}{q} \Z^k} \in \bar \Pi^{E'}(\tfrac{1}{q}\Z^k, \Z^k)$.  Therefore, by part~(1), $\bar \pi_\T \in  \bar \Pi^E(\R^k, \Z^k)$.  Since $\bar \Pi^E(\R^k, \Z^k)$ is a vector space containing $\bar \pi$ and $\bar \pi_\T$, we have that $\bar \pi_{\mathrm{zero}(\T)} = \bar \pi - \bar\pi_\T \in \bar\Pi^E(\R^k,\Z^k)$ which  establishes part~(2).
\end{proof}
Due to the decomposition in part~(2) of Lemma~\ref{lem:tri-restriction}, we can determine if a non-trivial perturbation function $\bar \pi \in \bar \Pi^{E(\pi)}(\R^k, \Z^k)$ exists by considering separately the spaces $\bar \Pi^{E(\pi)}_{\T}(\R^k, \Z^k)$ and $\bar \Pi^{E(\pi)}_{\textrm{zero}(\T)}(\R^k, \Z^k)$.  This is used in a procedure to test extremality described in ~\autoref{sec:direct-algorithm}.
\begin{remark}
\label{rem:tri-satisfy-lemma}
The polyhedral complexes $\P_{B}$ for $B = \tfrac{1}{q} \Z \cap [0,1)$ from Example~\ref{ex:1d-breakpoint-complex} and $\P_q$ from Example~\ref{ex:2d-standard-triangulation} are triangulations of $\R^1$ and $\R^2$, respectively, and satisfy the hypotheses of Lemma~\ref{lem:tri-restriction}.  This fact can be seen in Figure~\ref{fig:uniform-spacing} for the case of $\P_B$.  The polyhedral complex $\P_q$ will be discussed more in section~\ref{sec:one-two-dim}.
\end{remark}

\begin{figure}[t]
  \centering
  \includegraphics[width=.44\linewidth]{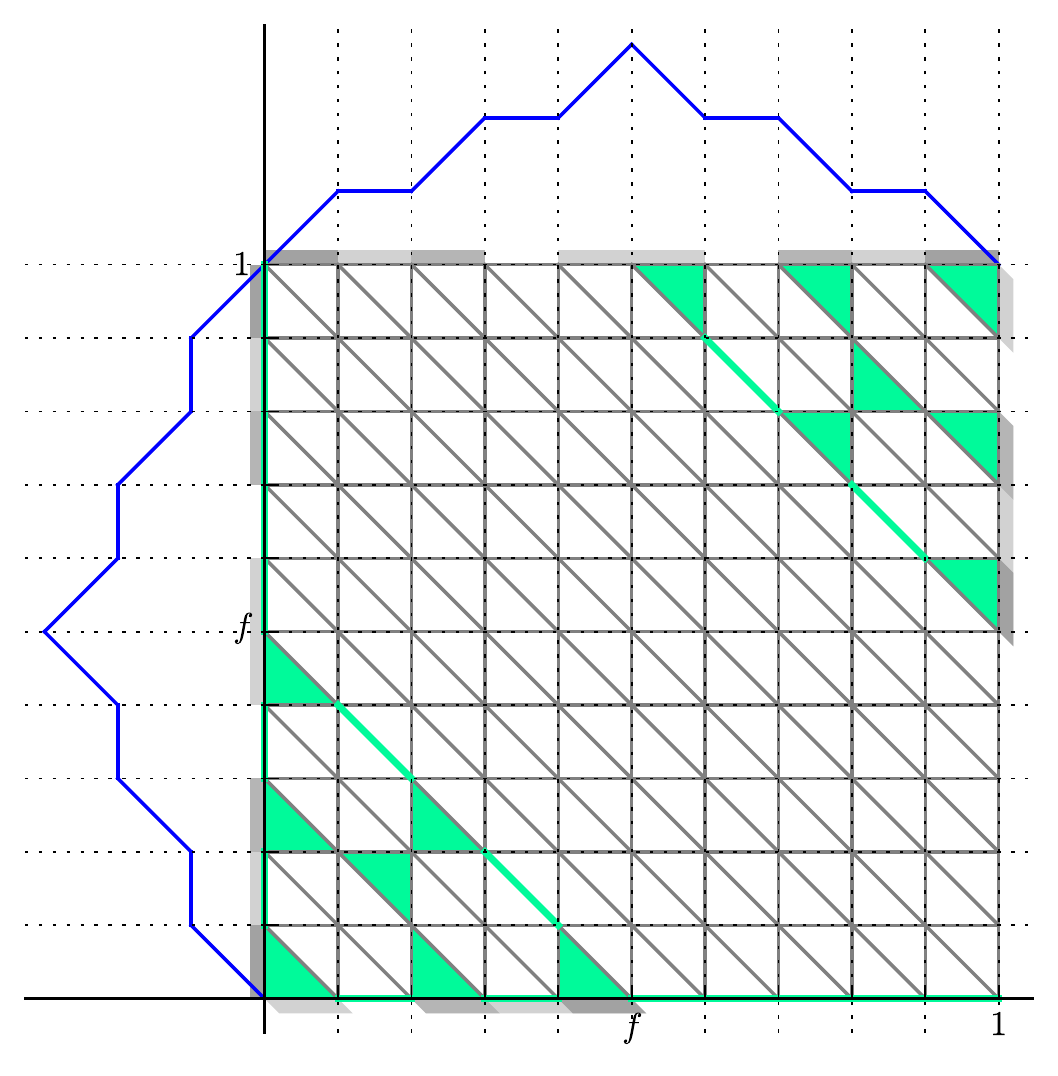}
  \caption{
    Diagram of a function (\emph{blue graphs on the top and the left}) on the
    evenly spaced complex~$\P_{\frac1{10}\Z}$ and the corresponding complex $\Delta\P_{\frac1{10}\Z}$ (\emph{gray
      solid lines}), as plotted by the command
    \sage{\sagefunc{plot_2d_diagram}(h)},
    where \sage{h = \sagefunc{not_extreme_1}()}.
    Faces of the complex on which $\Delta\pi = 0$, i.e., additivity holds, are \emph{shaded green}.
    The \emph{heavy diagonal green lines} $x + y = f$ and $x + y = 1+f$ correspond to the symmetry
    condition.  
    At the borders, the projections $p_i(F)$ of two-dimensional
    additive faces are shown as \emph{gray shadows}: $p_1(F)$ at the top border, $p_2(F)$ at
    the left border, $p_3(F)$ at the bottom and the right borders.
    Since the breakpoints of~$\P_{\frac1{10}\Z}$ are equally spaced, also $\Delta \P_{\frac1{10}\Z}$ is very
    uniform, consisting only of points, lines, and triangles, and the projections are
    either a breakpoint in $\P_{\frac1{10}\Z}$ or an interval in $\P_{\frac1{10}\Z}$; compare with
    \autoref{fig:delta-p}. 
  }
  \label{fig:uniform-spacing}
\end{figure}

\section{The Interval Lemma and its $k$-dimensional generalizations}\label{sec:foundations}


In order to prove that a given minimal valid function $\pi$ is a facet (or an
extreme function), we make use of the additivity domain $E(\pi)$ of a subadditive
function~$\pi\colon\R^k \to\R$. As discussed in the roadmap (\autoref{s:roadmap}), we would like to establish that $E(\pi) \subseteq E(\pi')$ implies $\pi = \pi'$ for every minimal valid function $\pi'$. An important ingredient in this step is to infer that $\pi'$ is an affine function when restricted to projections of $E(\pi)$. For this purpose, it is convenient to separate the additivity domain into convex sets, which we then study
independently.  In the important case of continuous piecewise linear functions, we
already know from \autoref{section:additivity-discretized} that it suffices to
study the maximal additive faces of the complex $\Delta\P$. 

The primary object of investigation is the {\em functional equation} known as the \emph{(additive) Cauchy functional equation}, which in its most general form is the study of real-valued functions $\theta$ satisfying
\begin{equation}\label{eq:cauchy-1d}
  \theta(\u)+\theta(\v) = \theta(\u+\v),\quad (\u,\v)\in F
\end{equation}
where $F$ is some subset of $\R^k \times \R^k$. We focus on convex sets $F$ that can be used as building blocks to cover $E(\pi)$ or other non-convex domains.  The simplest convex sets~$F$ of $\R^k\times\R^k$ are direct (Cartesian) products
$U\times V$, where $U$ and $V$ are convex sets of $\R^k$.  
For $k=1$, this
means we consider intervals $U\subseteq\R $ and $V\subseteq\R$ and set $F=U \times V$, i.e., we consider the functional equation $\theta(u) + \theta(v) = \theta(u+v)$ for all $u\in U$ and $v \in V$.


\subsection{The classical case: Cauchy's functional equation}
Classically (see, e.g., \cite[Chapter 2]{aczel1989functional}), \eqref{eq:cauchy-1d} is studied for the case
$F = \R \times \R$.  
In addition to the obvious \emph{regular solutions} to~\eqref{eq:cauchy-1d}, which are
the (homogeneous) linear functions $\theta(x) = cx$, there exist certain
pathological solutions, which are highly discontinuous~\cite[Chapter 2, Theorem 3]{aczel1989functional}; these were used in
Propositions~\ref{prop:neg-pi} and \ref{prop:nonneg-insufficient}. In order to rule out
these solutions, one imposes a regularity hypothesis.  Various such regularity
hypotheses have been proposed in the literature; for example, it is sufficient
to assume that the function~$\theta$ is bounded on bounded intervals
\cite[Chapter 2, Theorem 8]{aczel1989functional}.

\subsection{The bounded case: Gomory--Johnson's Interval Lemma in $\R^1$}

The so-called Interval Lemma was introduced by Gomory and Johnson in~\cite{tspace} (the result appears implicitly in the proof of \cite[Theorem 3.3]{infinite2}). This result
concerns the Cauchy functional equation~\eqref{eq:cauchy-1d} on a bounded domain,
i.e., the arguments $u$, 
$v$, and $u+v$ come from bounded intervals $U$, $V$, and their sum~$U+V$,
rather than the entire real line, i.e., additivity is on the set $F = U \times V$. In this case, we find that regular solutions are affine on these intervals; we lose homogeneity of the solutions. In fact, instead of equation~\eqref{eq:cauchy-1d}, one can consider the more general
equation~$f(u)+g(v) = h(u+v)$, with three functions~$f$, $g$, and $h$ instead
of one function~$\theta$.   
\begin{lemma}(Interval Lemma)~\cite[\autoref{equi3:one-dim-interval_lemma}]{bhk-IPCOext}\label{one-dim-interval_lemma} Given real numbers $u_1 < u_2$ and $v_1 < v_2$, let $U = [u_1, u_2]$, $V = [v_1, v_2]$, and $U + V = [u_1 + v_1, u_2 + v_2]$.
Let $f \colon U \rightarrow
\mathbb{R}$, $g \colon  V \rightarrow
\mathbb{R}$, $h \colon U+V  \rightarrow
\mathbb{R}$ be bounded functions. \\ If $f(u)+g(v) = h(u+v)$ for every
$(u, v) \in U \times V$, then there exists $c\in \R$ such that $f(u)=f(u_1)+c(u-u_1)$ for every $u\in U$, $g(v)=g(v_1)+c(v-v_1)$ for every $v\in V$, $h(w)=h(u_1+v_1)+c(w-u_1-v_1)$ for every $w\in U+V$. In other words, $f$, $g$ and $h$ are affine with gradient $c$ over $U$, $V$, and $U+V$ respectively.
\end{lemma}
We provide a brief justification of this result under the assumption that
$f,g$ and $h$ are in $C^2(\R)$ (continuous first and second derivatives). We
differentiate the relation $f(u) + g(v) = h(u+v)$ with respect to $u$ (holding
$v$ fixed in the interval $V$) to obtain $f'(u) = h'(u+v)$ for all $u \in
\intr(U)$. Since the choice of $v$ was arbitrary, this actually means $f'(u) =
h'(u+v)$ for all $u \in \intr(U)$ and $v\in \intr(V)$. But then
differentiating this relation with respect to $v$ we obtain $0 =
h''(u+v)$. This implies that $h$ is affine over $U+V$, and $f$ is affine with
the same slope over $U$. Similarly, fixing $u$ in $U$ and differentiating with
respect to $v$ we obtain $g'(v) = h'(u+v)$ for all $v \in \intr(V)$, implying
that $g$ is affine with the same slope over $V$. The result under the weaker
assumption of boundedness of the functions is obtained by making a discrete
version of these derivative arguments; the details are complicated and we
refer the reader
to~\cite[\autoref{equi3:one-dim-interval_lemma}]{bhk-IPCOext} for a full proof. 


\subsection{The full-dimensional Cartesian case: Higher-dimensional Interval Lemma}
\label{s:generalized-interval-lemma}

%

We now discuss generalization of the Interval Lemma (\autoref{one-dim-interval_lemma}) presented
in the previous section to the $k$-dimensional setting. 
The first higher dimensional versions of \autoref{one-dim-interval_lemma} in the
literature appear in \cite{dey3, 3slope} for the case of $k=2$ and in
\cite{bhkm} for general $k$, all of which apply when either $U$ or $V$ contains the origin.  
The result in \cite{dey3} applies allows also for so-called
\emph{star-shaped} sets that contain the origin.
We will follow the results of~\cite{bhk-IPCOext}, which all for more general types of convex sets.  Similar proofs of these results allow for star-shaped sets as well, but this is not presented here.



\begin{theorem}[Higher-dimensional Interval Lemma, full-dimensional
  version {\cite[\autoref{equi3:theorem:generalized_interval_lemma_fulldim}]{bhk-IPCOext}}]\label{theorem:generalized_interval_lemma_fulldim} 
Let $f,g,h \colon\allowbreak \R^k\to \R$ be bounded functions. Let $U$ and $V$ be convex
subsets of $\R^k$ such that $f(\u) + g(\v) = h(\u+\v)$ for all $(\u, \v) \in U\times V$. Assume that $\aff(U) = \aff(V) = \R^k$. Then there exists a vector $\cve\in \R^k$ such that
%
%
%
$f$, $g$ and $h$ are affine over $U$, $V$ and $W = U+V$, respectively, with the
same gradient $\cve$.
\end{theorem}

\subsection{The full-dimensional convex case: Cauchy's functional equation on convex additivity domains in~$\R^k$} The most direct generalization applies to full dimensional convex sets $F$.
The general idea of the proof is to consider a point $(\x,\y)$ in such a convex additivity domain $F$, and consider a finite set of smaller subsets $F_1, \ldots, F_k \subseteq F$ that are Cartesian products, such that $\x \in F_1$, $\y\in F_k$ and $\intr(F_i) \cap \intr(F_{i+1}) \neq \emptyset$ for each $i=1, \ldots, k=1$. Applying Theorem~\ref{theorem:generalized_interval_lemma_fulldim} on each $F_i$, we can deduce that the functions are affine over all of $F$. This idea of ``patching'' together simple additivity domains to obtain affine properties over a more complicated domain was first introduced in~\cite[Proposition 23]{dey3}, and then used again in~\cite[Lemma 10]{3slope} and~\cite[Lemmas 3.5, 3.6]{bhkm}.

%
%
%


\begin{theorem}[Convex additivity domain lemma, full-dimensional
  version {\cite[\autoref{equi3:lem:projection_interval_lemma_fulldim}]{bhk-IPCOext}}]\label{lem:projection_interval_lemma_fulldim} 
  Let $f,g,h \colon \R^k \to \R$ be bounded functions. 
  Let $F \subseteq \R^k \times \R^k$ be a full-dimensional convex set
  such that $f(\u) + g(\v) = h(\u+\v)$ for all $(\u, \v) \in F$. 
  Then there exists a vector $\cve\in \R^k$ such that $f, g$ and $h$ are
  affine with the same gradient $\cve$ over $\intr(p_1(F))$,
  $\intr(p_2(F))$ and $\intr(p_3(F))$, respectively. 
\end{theorem}

This theorem is obtained by applying the ``patching'' idea to subsets $F_i$
that are Cartesian
products. \autoref{theorem:generalized_interval_lemma_fulldim} is applied to
the individual subsets $F_i$ to deduce affine properties. 

It is notable that we can only deduce affine linearity over the
\emph{interiors} of the projections in Theorem~\ref{lem:projection_interval_lemma_fulldim}, as opposed to the conclusion of Theorem~\ref{theorem:generalized_interval_lemma_fulldim}.  This is best possible, as is illustrated in~\cite[Remark~\ref{equi3:remark:projection_interval_lemma_interiors_only}]{bhk-IPCOext}. If continuity is assumed for the functions, then one easily extends the affine-ness property to the boundary (\autoref{s:boundary}).

\subsection{The lower-dimensional case: Affine properties with respect to subspaces $L$}
\label{section:convex-additivity-with-L}

Theorems~\ref{theorem:generalized_interval_lemma_fulldim} and
\ref{lem:projection_interval_lemma_fulldim} can be established in a significantly more general setting, which
takes care of  situations in which the set $F$ is not full-dimensional
(Theorems~\ref{lem:generalized_interval_lemma} and
\ref{lem:projection_interval_lemma}). Affine properties are deduced with
respect to certain subspaces, which is important for the classification of extreme functions in two or more dimensions.

We start with a result obtained in~\cite{bhk-IPCOext}, in which the additivity domain is $U\times V$ for convex sets $U\subseteq\R^k$ and $V\subseteq\R^k$, which are not necessarily of the same dimension.  In this general setting we cannot expect to deduce that the solutions are affine over
$U$, $V$, and $U+V$.  
\begin{remark}
  Indeed, if $U + V$ is a direct sum,
  i.e., for every $\w \in U + V$ there is
  a unique pair $\u \in U$, $\v \in V$ with $\w = \u+\v$, then 
  $f(\u) + g(\v) = h(\u+\v)$ merely expresses a form of separability of~$h$
  with respect to certain subspaces,
  and $f$ and $g$ can be arbitrary functions; 
  see \autoref{fig:higher-dim-interval-lemma}\,(c). 
\end{remark}
\begin{figure}[t!]
 \begin{center}
  \includegraphics[scale=0.4]{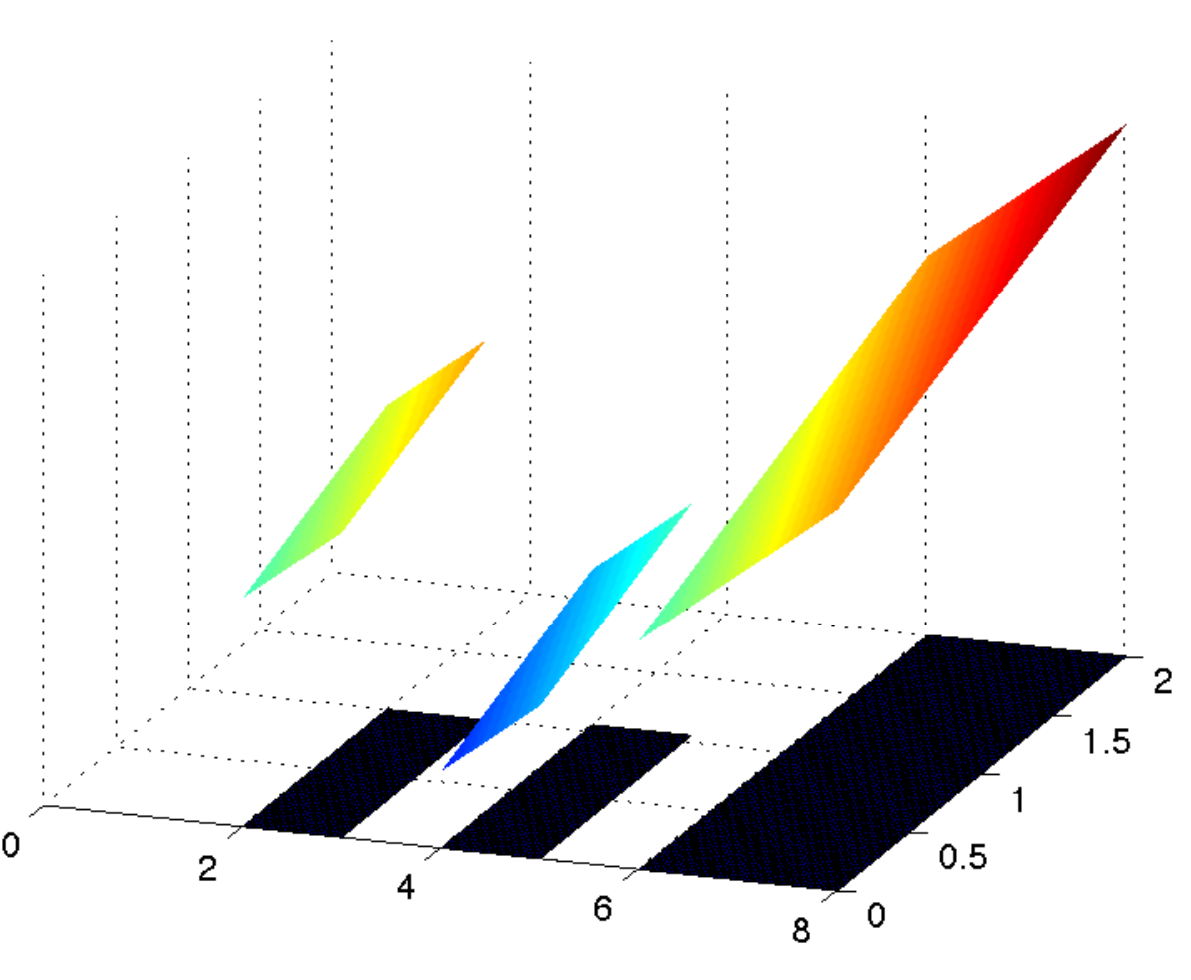}
   \includegraphics[scale=0.4]{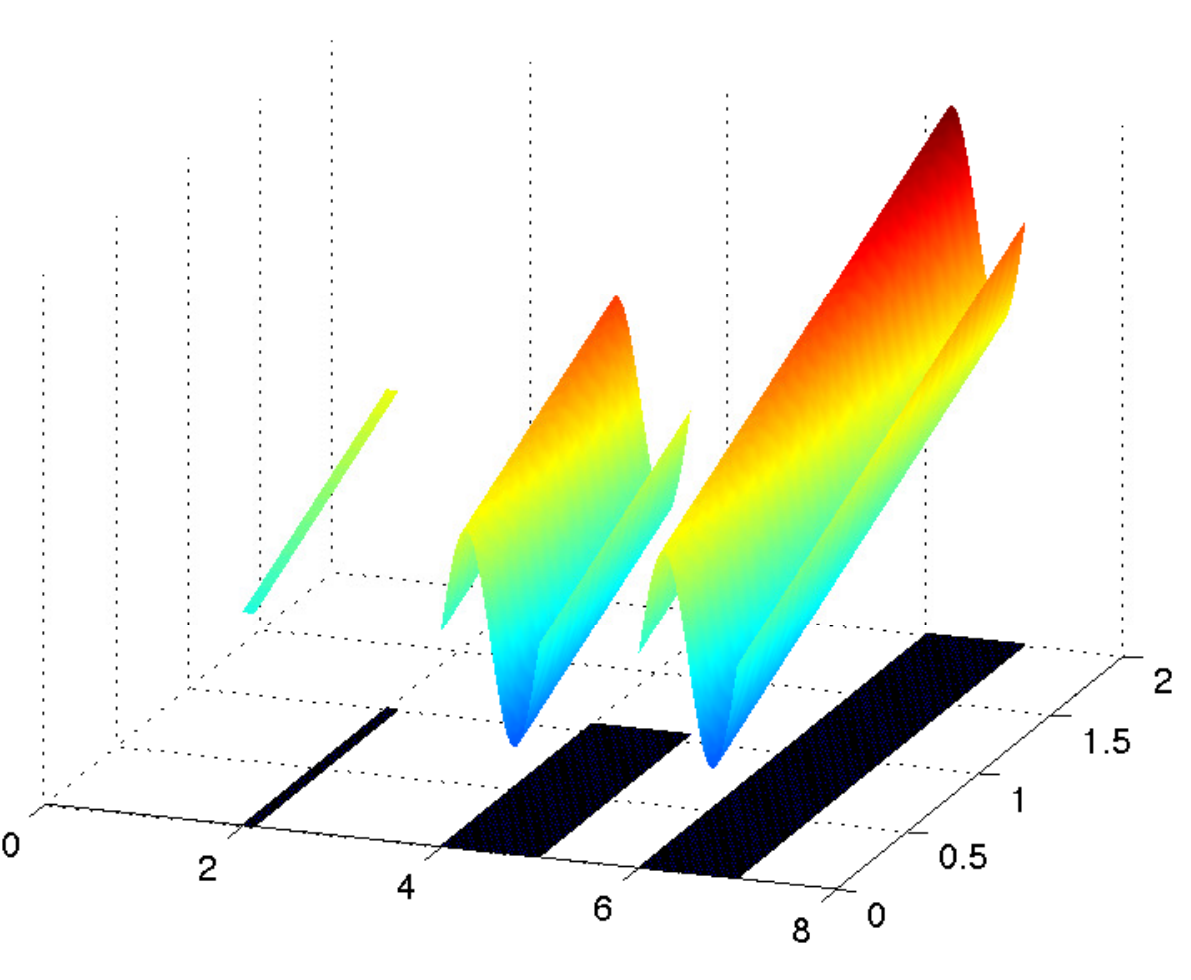}
  \includegraphics[scale=0.4]{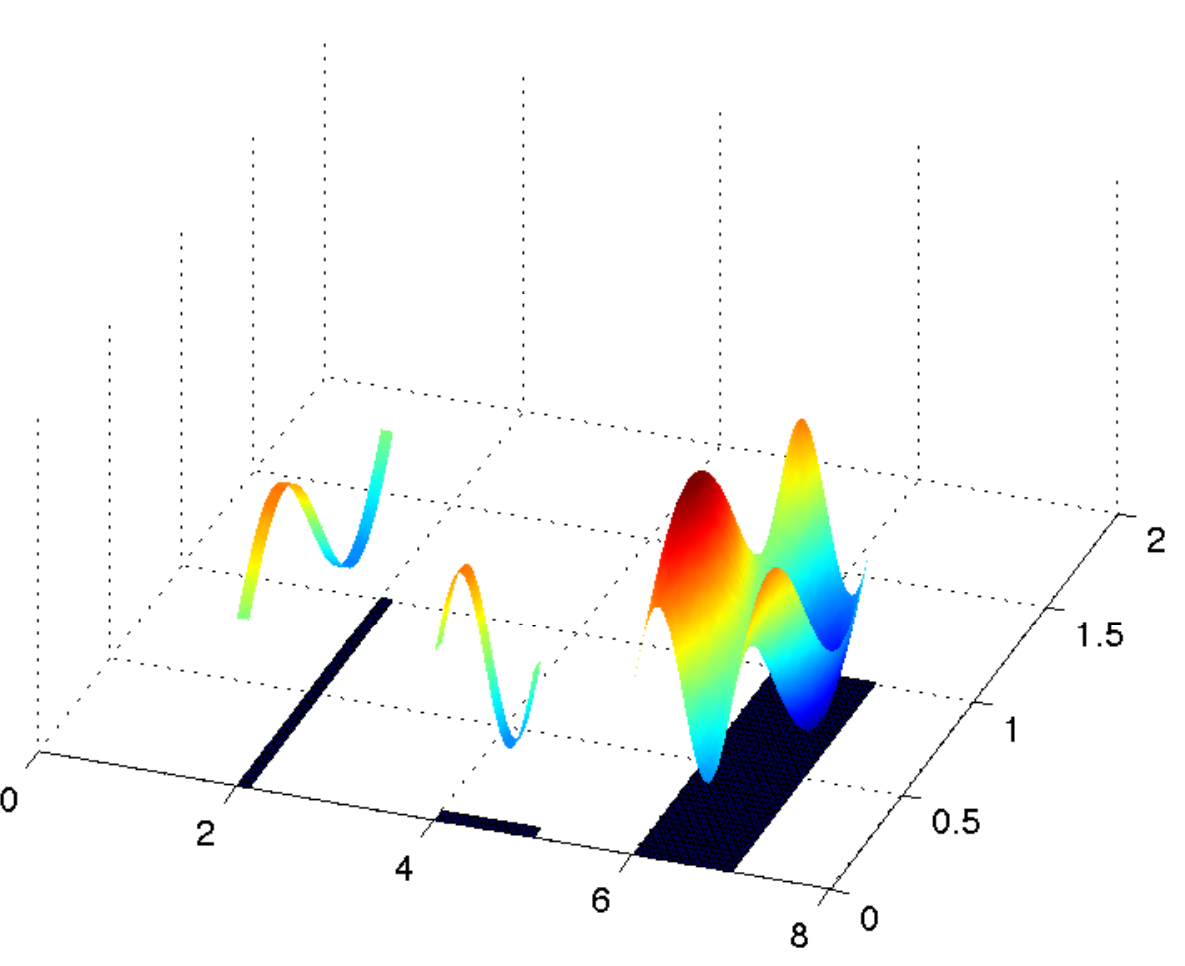} 
     \\
      \end{center}
      \vspace{-0.5cm}
      \begin{flushleft}
   \hspace{2.5cm}   (a) \hspace{4.3cm} (b) \hspace{4.3cm} (c) \hspace{5cm}
      \end{flushleft}
  \caption{Cauchy's functional equation on bounded domains. In each part (a), (b), and (c), we depict 3 domains in the plane, $U,V,U+V$, left to right, and an function that is additive over these domains.  (a)
    Full-dimensional situation. (b) Sum of a one-dimensional and a
    two-dimensional set; not a direct sum.  
   (c) Direct sum of (non-parallel) one-dimensional sets.}
  \label{fig:higher-dim-interval-lemma}
\end{figure}


\begin{definition}
Let $U \subseteq \R^k$. 
Given a linear subspace $L \subseteq \R^k$, we say $\pi\colon U \to \R$ is
{\em affine with respect to $L$ over $U$} if there exists $\cve \in \R^k$ such
that $\pi(\u^2) - \pi(\u^1) =  \cve\cdot( \u^2 - \u^1)$ for any $\u^1, \u^2
\in U$ such that $\u^2 - \u^1 \in L$.

\end{definition}

\begin{theorem}[Higher-dimensional Interval Lemma; {\cite[ \autoref{equi3:lem:generalized_interval_lemma}]{bhk-IPCOext}}]
\label{lem:generalized_interval_lemma}
Let $f,g,h \colon \R^k \to \R$ be bounded functions. Let $U$ and $V$ be convex subsets of $\R^k$ such that $f(\u) + g(\v) = h(\u+\v)$ for all $(\u,\v) \in F = U\times V$. Let $L$ be a linear subspace of $\R^k$ such that $(L + U) \times (L + V) = (L\times L) + F \subseteq \aff(F) = \aff(U) \times \aff(V)$. Then there exists a vector $\cve\in \R^k$ such that
$f$, $g$ and $h$ are affine with respect to $L$  over $p_1(F) = U$, $p_2(F) = V$ and $p_3(F) = U+V$
respectively, with gradient $\cve$.
\end{theorem}

Theorem~\ref{theorem:generalized_interval_lemma_fulldim} follows when $L=\R^k$.

\begin{definition}
  For a linear space $L \subseteq \R^k$ and a set $U \subseteq \R^k$ such that
  for some $\u \in \R^k$ we have $\aff(U) \subseteq L + \u$, we will denote by
  $\intr_L(U)$ the interior of $U$ in the relative topology of $L + \u$.
\end{definition}
Note that $\intr_L(U)$ is well defined because either $\aff(U) = L + \u$, or
$\intr_L(U) = \emptyset$.   We now state our most general theorem relating to equation~\eqref{eq:cauchy-1d} on a convex domain.

\begin{theorem}[Convex additivity domain lemma; {\cite[\autoref{equi3:lem:projection_interval_lemma}]{bhk-IPCOext}}]
\label{lem:projection_interval_lemma}
Let $f,g,h \colon \R^k \to \R$ be bounded functions. Let $F \subseteq \R^k \times \R^k$ be a convex set
such that $f(\u) + g(\v) = h(\u+\v)$ for all $(\u, \v) \in F$. 
Let $L$ be a linear subspace of $\R^k$ such that $(L \times L) + F \subseteq \aff(F)$.
Let $(\u^0, \v^0) \in \relint(F)$. 
Then there exists a vector $\cve\in \R^k$ such that $f, g$ and $h$ are affine with gradient $\cve$ over $\intr_L((\u^0 + L) \cap p_1(F))$, $\intr_L((\v^0 + L) \cap p_2(F))$ and $\intr_L((\u^0 + \v^0 + L) \cap p_3(F))$, respectively.
\end{theorem}

Theorem~\ref{lem:projection_interval_lemma_fulldim} follows when $L=\R^k$.

\subsection{Continuity at the boundary}\label{s:boundary}

The one-dimensional Interval Lemma, Lemma~\ref{one-dim-interval_lemma},
includes affine properties on the boundaries.  Using this, it is easy to prove
that a similar Interval Lemma holds on all non-degenerate intervals
$U,V\subseteq \R$ that are any of open, half-open, or closed.  Only in special
cases in higher dimensions is it possible to extend affine properties in
Theorem~\ref{lem:projection_interval_lemma} to the boundary; in general this is not possible (see \cite[Remark~\ref{equi3:remark:projection_interval_lemma_interiors_only}]{bhk-IPCOext}).

Of course, if we use the stronger regularity assumption that $f$, $g$, and $h$
are continuous functions (rather than merely bounded functions), then the
affine properties extend to the boundary as well.

\begin{corollary}[Convex additivity domain lemma for continuous functions; {\cite[\autoref{equi3:lem:projection_interval_lemma-corollary}]{bhk-IPCOext}})]
\label{lem:projection_interval_lemma-corollary}
Let $f,g,h \colon \R^k \to \R$ be continuous functions. Let $F \subseteq \R^k \times \R^k$ be a convex set
such that $f(\u) + g(\v) = h(\u+\v)$ for all $(\u, \v) \in F$. 
Let $L$ be a linear subspace of $\R^k$ such that $L \times L + F \subseteq \aff(F)$.
Let $(\u^0, \v^0) \in \relint(F)$. 
Then there exists a vector $\cve\in \R^k$ such that $f, g$ and $h$ are affine with gradient $\cve$ over $(\u^0 + L) \cap p_1(F)$, $(\v^0 + L) \cap p_2(F)$ and $(\u^0 + \v^0 + L) \cap p_3(F)$, respectively.
\end{corollary}

\section{Sufficient conditions for extremality in the $k$-row infinite
  group problem}\label{s:sufficient-cond}

\subsection{The $(k+1)$-Slope Theorem}

We have already mentioned the classic Gomory--Johnson 2-Slope Theorem
(\autoref{th:2-slope}), which states
that for $k=1$, if a continuous piecewise linear minimal function has only 2 slopes, then
it is extreme. An analogous 3-Slope Theorem for $k=2$ was proved by
Cornu\'ejols and 
Molinaro~\cite{3slope}.
We present here the $(k+1)$-Slope Theorem for the case of general~$k$ by Basu,
Hildebrand, K\"oppe and Molinaro~\cite{bhkm}, along
with the main ingredients of its proof. 

\begin{theorem}[{\cite[Theorem 1.7]{bhkm}}]\label{thm:k+1slope}
Let $\pi\colon \R^k \to \R$ be a minimal valid function that is continuous piecewise linear and genuinely $k$-dimensional\/\footnote{See \autoref{def:genk}.} with at most $k+1$ slopes, i.e., at most $k+1$ different values for the gradient of $\pi$ where it exists. Then $\pi$ is extreme and has exactly $k+1$ slopes.
\end{theorem} 

The proof will follow the basic roadmap of \autoref{s:roadmap} and use
\autoref{lem:tightness}; we give an outline here, before diving into the
details. For the rest of this section, $\pi$ is a continuous piecewise linear
minimal function that is genuinely $k$-dimensional with at most $k+1$
slopes. Let $\P$ be the associated polyhedral complex.  

\begin{enumerate}
\item Subadditivity and the property of being genuinely $k$-dimensional is used
  to first establish that $\pi$ has exactly $k+1$ gradient values $\gp^1,
  \ldots, \gp^{k+1} \in \R^k$. This is a relatively easy step, and we refer to
  the reader to~\cite[Lemma 2.11]{bhkm} for the details.
\item Consider any minimal valid functions $\pi^1, \pi^2$ such that $\pi = \frac{\pi^1 + \pi^2}{2}$. 
\item (Compatibility step) For each $i = 1, \ldots, k+1$, define
  $\P_i\subseteq \P$ to be the polyhedral complex formed by all the cells (and
  their faces) of $\P$ where the gradient of $\pi$ is $\gp^i$. Show that there
  exist $\gt^1, \ldots, \gt^{k+1}$ such that $\pi^1$ is affine over every cell
  in $\P_i$ with gradient $\gt^i$. 
\item (Gradient matching step) We then highlight certain structures of genuinely $k$-dimensional functions with $k+1$ slopes that lead to a system of $k(k+1)$ equations that are satisfied by the coefficients of $\gp^1, \ldots, \gp^{k+1}$ and $\gt^1, \ldots, \gt^{k+1}$. Then, it is established that this system of equations has a unique solution, and thus, $\gp^i = \gt^i$ for every $i=1, \ldots, k+1$.
\item For every $\rr \in \R^k$ there exist $\mu_1, \mu_2, \ldots, \mu_{k+1}$ such that $\mu_i$ is the fraction of the segment $[\0,\rr]$ that lies in $\P_i$. Thus, $$\pi(\rr) = \pi(\0) + \sum_{i=1}^{k + 1} \mu_i (\gp^i \cdot \rr) = \pi^1(\0) + \sum_{i=1}^{k+1} \mu_i (\gt^i\cdot \rr) = \pi^1(\rr).$$ This proves that $\pi = \pi^1$ and thus, $\pi = \pi^1 = \pi^2$, concluding the proof of \autoref{thm:k+1slope}. 
\end{enumerate}

\medskip

\paragraph{\em Compatibility Step.} The following observation is crucial:

\begin{lemma}\label{lem:k+1-useful}
Let $U, V \subseteq \R^k$ be full-dimensional convex sets such that $\0 \in U$.  Let $F = F(U,V,V)$\footnote{See the definition in~\eqref{eq:F-def}.}. 
Then $\0 \in p_1(F)$, $V = p_2(F) = p_3(F)$ and $p_1(F)$ is full-dimensional. 
Furthermore, if $\pi \colon \R^k \to \R$ is such that $\pi(\0) = 0$ and is affine on $U,V$ with the same slope, then $F \subseteq E(\pi)$.
\end{lemma}

\begin{proof}
By definition, $p_1(F) \subseteq U, p_2(F), p_3(F) \subseteq V$.  Since $\0 \in U$ and $\{\0\} + V = V$, we see that $p_2(F), p_3(F) = V$ and $\0 \in p_1(F)$.  Now, let $\v \in \intr(V)$.  Therefore there exists a ball $B(\v,\epsilon) \subseteq V$.  Since $U$ is full-dimensional, there exist $k$-linearly independent vectors $\u^1, \dots, \u^k \in U$ with $\| \u^i \| \leq \epsilon$.  But then $\u^i + \v \in V$.  Therefore, $\u^i \in p_1(F)$.  Finally, since $F$ is convex and the projection of convex sets is convex, we have that $p_1(F)$ is full-dimensional.

For the second part of the lemma, observe that there exist $\g \in \R^k$ and $\delta \in \R$ be such that $\pi(\u) = \g\cdot \u$ (follows since $\0\in U$ and $\pi(\0) =0$) for all $\u \in U$ and $\pi(\v) = \g\cdot \v + \delta$ for all $\v \in V$. Then for any $\u \in U, \v\in V$ with $\u + \v \in V$, we have $\pi(\u) + \pi(\v) - \pi(\u + \v) = (\g\cdot \u) + (\g\cdot \v + \delta) - (\g\cdot(\u+\v) + \delta) = 0$.
\end{proof}

The analysis of step (1) also shows that for every $i=1, \ldots, k+1$, there exist $C_i \in \P_i$ such that $\0 \in C_i$ (in other words, for every gradient value, there is a cell containing the origin with that gradient). Fix an arbitrary $i \in \{1, \ldots, k+1\}$ and consider any cell $P \in \P_i$. By \autoref{lem:k+1-useful} with $U=C_i$ and $V=P$, we obtain that $F = F(C_i, V,V) \subseteq E(\pi)$. By \autoref{lem:tightness} (ii), $F \subseteq E(\pi^1)$ and by \autoref{lem:tightness} (iii), $\pi^1$ is continuous. By \autoref{lem:projection_interval_lemma_fulldim} and continuity of $\pi^1$, we obtain that $\pi^1$ is affine on $C_i$ and $P$ with the same gradient. Since the choice of $P$ was arbitrary, this establishes that for every cell $P\in P_i$, $\pi^1$ is affine with the same gradient; this is precisely the desired $\gt^i$.

\medskip

\paragraph{\em Gradient matching step.} The system for step (4) has two sets of constraints, the first of which follows from the condition that $\pi(\f + \w) = 1$ for every $\w \in \Z^k$.  The second set
of constraints is more involved. Consider two adjacent cells $P, P' \in
\mathcal{P}$ that contain a segment $[\x, \y] \subseteq \mathbb{R}^k$ in their
intersection. Along the line segment $[\x,\y]$, the gradients of $P$ and $P'$
projected onto the line spanned by the vector $\y - \x$ must agree; the second
set of constraints captures this observation. We will identify a set of vectors $\rr^1, \ldots, \rr^{k+1}$ such that every subset of $k$ vectors is linearly independent and such that each vector $\rr^i$ is contained in $k$ cells of $\P$ with different gradients. We then use the segment $[\0,\rr^i]$ to obtain linear equations involving the gradients of $\pi$ and $\pi'$. The fact that every subset of $k$ vectors is linearly independent will be crucial in ensuring the uniqueness of the system of equations.

      \begin{lemma}[{\cite[Lemma 3.10]{bhkm}}]\label{lemma:kkmDirections}
                There exist vectors $\rr^1, \rr^2, \ldots, \rr^{k + 1} \in \R^k$ with the following properties:
                \begin{enumerate}[\rm(i)]
                        \item For every $i,j,\ell \in \{1, \ldots, k + 1\}$ with $j, \ell$ different from $i$, the equations $\rr^i\cdot \gp^{j} = \rr^i\cdot \gp^{\ell}$ and $\rr^i\cdot \gt^{j} = \rr^i\cdot \gt^{\ell}$ hold.
                        
                        \item Every $k$-subset of $\{\rr^1, \ldots, \rr^{k+1}\}$ is linearly independent. 
                \end{enumerate}
        \end{lemma}
        
        \begin{figure}[tp]
\centering
\vspace{-1ex}
\scalebox{0.7}{\ifpdf
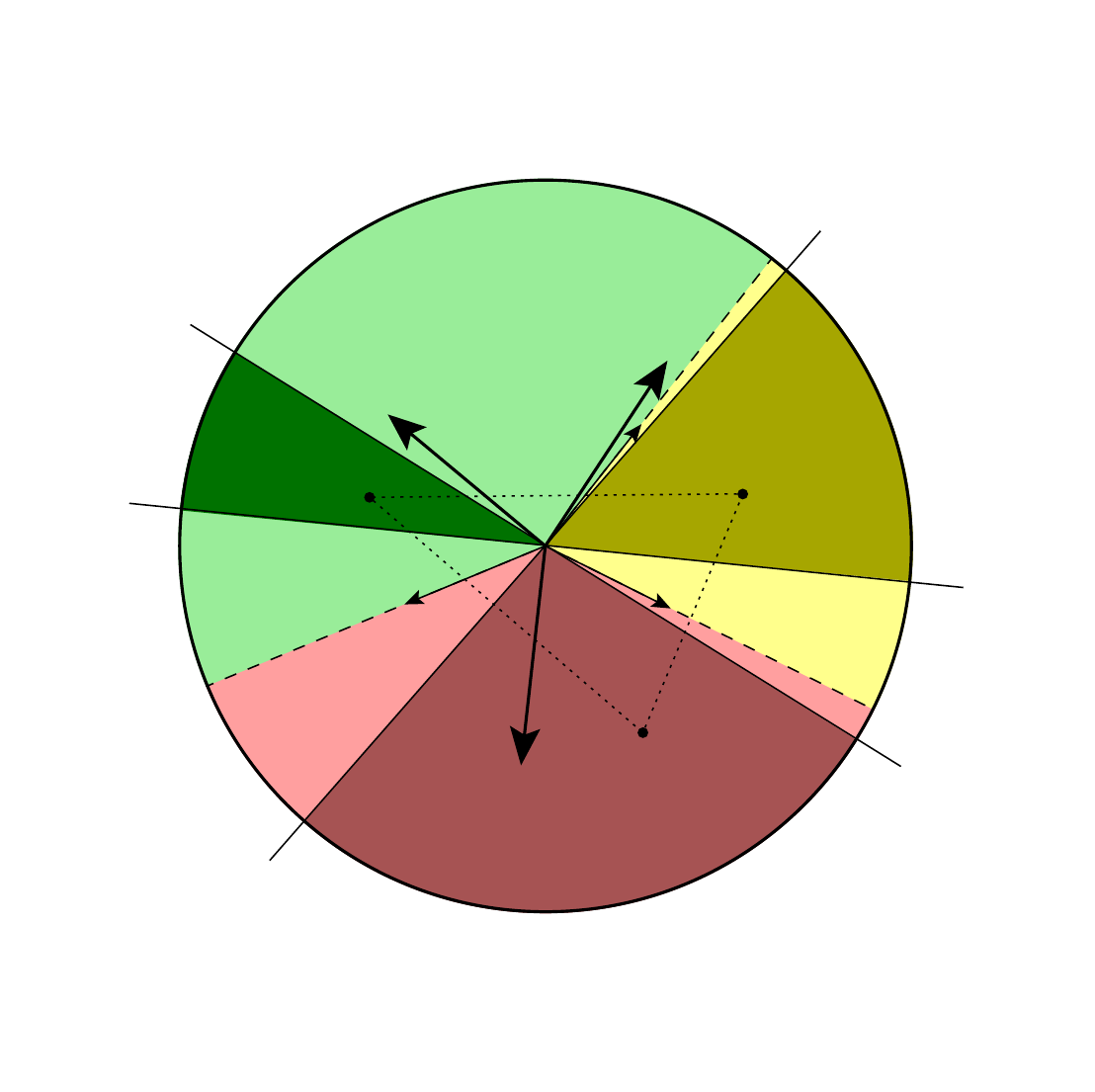
\else
\input{figureKKMcircleSmallerAdjust.pstex_t}
\fi}
\vspace*{-5ex}\par
\caption{The geometry of the proof of \autoref{lemma:kkmDirections}.   
  Each cone $C_i$ (shaded in dark colors) is the intersection of the
  halfspaces~$H_j$ (defined by the gradients~$\bar \g^j$) for $j\neq i$. 
  Near the origin (within the ball~$\Ball(\0,\varepsilon)$), 
  each point of~$C_i$ lies in the set $F_i$ of points
  where the function~$\pi$ has gradient~$\bar \g^i$ (shaded in light colors).
  Picking points $\ve v_i$ near the origin in the interior of~$C_i$, 
  we construct a simplex~$\Delta$ with $\0$ in its interior. 
  By applying the KKM Lemma to each of its facets~$\Delta_i$, we
  show the existence of the vectors~$\rr^i$ with the desired properties.
} 
\label{fig:KKM}
\end{figure}
        
       The proof of \autoref{lemma:kkmDirections} 
       uses a nontrivial result known as the {\em Knaster--Kuratowski--Mazurkiewicz Lemma (KKM Lemma)} from fixed point theory, which exposes a nice structure in the gradient pattern of $\pi$. The KKM lemma states that if a $d$-dimensional simplex is covered by $d+1$ closed sets satisfying certain combinatorial conditions, then there is a point in the intersection of all $d+1$ sets. This lemma is applied to the facets of a certain simplex $S$ containing the origin, where the closed sets form $S \cap \P_i$. The fixed points on the $k+1$ facets of this simplex give the vectors $\rr^1, \ldots, \rr^{k+1}$ from \autoref{lemma:kkmDirections}. The bulk of the technicality lies in showing that the hypothesis of the KKM lemma are satisfied by the gradient structure of $\pi$. A few more details are offered in \autoref{fig:KKM}.


        We finally present the system of linear equations that we consider.
        
        \begin{corollary}[{\cite[Corollary 3.13]{bhkm}}]\label{lemma:linear-system}
                Consider any $k+1$ affinely independent vectors $\a^1, \a^2, \ldots, \a^{k + 1} \in \mathbb{Z}^k + \f$. Also, let $\rr^1, \rr^2, \ldots, \rr^{k + 1}$ be the vectors given by \autoref{lemma:kkmDirections}. Then there exist $\mu_{ij} \in \R_+$, $i, j \in \{1, \ldots, k+1\}$ with $\sum_{j = 1}^{k + 1} \mu_{ij} = 1$ for all $i \in \{1, \ldots, k+1\}$ such that both $\gt^1, \ldots, \gt^{k+1}$ and $\gp^1, \ldots, \gp^{k+1}$ are solutions to the linear system 
        \begin{equation}\label{eq:linear-system}
\begin{aligned}
\textstyle\sum_{j=1}^{k+1} (\mu_{ij}\a^i)\cdot \gs^j & = 1 &\qquad& \text{for all }  i \in \{1, \ldots, k+1\}, \\
\rr^i\cdot \gs^{j} - \rr^i\cdot \gs^{\ell} & = 0 & \qquad& \text{for all } i,j, \ell \in \{1, \ldots, k+1\} \textrm{ such that } i \neq j, \ell, 
\end{aligned}
\end{equation}
with variables  $\g^1, \ldots, \g^{k+1}\in \R^k$.
        \end{corollary}

        We remark that we can always find vectors $\a^1, \a^2, \ldots, \a^{k + 1} \in \mathbb{Z}^k + \f$ such that the set $\a^1, \ldots, \a^{k+1}$ is affinely independent, so the system above indeed exists. Property (ii) in \autoref{lemma:kkmDirections} and the fact that $\a^1, \ldots, \a^{k+1}$ are affinely independent can be used to show that \eqref{eq:linear-system} has either no solutions or a unique solution. 
        Since $\gp^1, \ldots, \gp^{k+1}$ is a solution, the conclusion is that the system has a unique solution and so $\gt^j = \gp^j$ for each $j=1, \ldots, k+1$.

\begin{remark}\label{rem:facet-extreme}
Along almost identical lines, one can show that a $(k+1)$-slope function $\pi$ is a facet -- this is done in~\cite{bhkm}. The only difference is that the continuity of $\pi^1$ in the proof above was obtained easily via Lemma~\ref{lem:tightness} (iii). For the facetness proof, this continuity argument is slightly more involved.
\end{remark}

\subsection{Construction of extreme functions with the sequential-merge procedure}

\label{s:sequential-merge}

Dey and Richard~\cite{dey2} gave the first examples of facets in higher dimensions by combining facets from lower dimensions.  We outline these concepts here.  For a more detailed discussion, we refer the reader to~\cite{dey2} and also the survey~\cite{Richard-Dey-2010:50-year-survey}. 

Let  $\varphi \colon \R\to \R$ be valid for $R_{f_{k+1}}(\R,\Z)$ and $\pi \colon \R^k \to \R$ valid for $R_\f(\R^k, \Z^k)$.  The \emph{sequential merge} of $\varphi$ and $\pi$ is the function $\varphi \merge \pi \colon \R^k\times \R \to \R$ given by
$$
(\varphi \merge \pi)(\x, x_{k+1}) = \frac{ \pi(\x) \sum_{i=1}^k f_i + f_{k+1} \varphi\Big( \sum_{i=1}^{k+1}x_i - \pi(\x) \sum_{i=1}^k f_i\Big) }{\sum_{i=1}^{k+1} f_i}.  
$$
Here we assume, without loss of generality, that $f_{k+1} \in (0,1)$, $\f \in [0,1)^k \setminus \{\0\}$.
The \emph{lifting space representation} of a function $\pi \colon \R^k \to \R$ is given by $[\pi]_\f(\x) = \sum_{i=1}^k x_i -  \pi(\x) \sum_{i=1}^k f_i$.\footnote{This is a superadditive pseudo-periodic function in the terminology of~\cite{Richard-Li-Miller-2009:Approximate-Liftings}.}

Dey and Richard showed that $(\varphi \merge \pi)(\x, x_{k+1})$ is a facet for
$R_{(\f,f_{k+1})}(\R^{k+1}, \Z^{k+1})$ provided that  $\varphi$ and $\pi$ are
facets, their lifting representations are non-decreasing, and the perturbation
spaces\footnote{See \autoref{subsec:perturbations}.} $\bar \Pi^{E(\varphi)}(\R,\Z)$ and $\bar \Pi^{E(\pi)}(\R^k, \Z^k)$ both contain only trivial solutions~\cite[Theorem 5]{dey2}.  They also show how to extend these sequential merge facets to facets of the mixed-integer problem~\cite[Proposition 15]{dey2}.
This produces a simple method to construct facets in higher dimensions from facets in lower dimensions.  

Some sequential merge functions can be projected as well.  Let $\xi \colon \R
\to \R$ be the \sagefunc{gmic} function and let $\pi \colon \R\to \R$ be a valid function for $R_f(\R,\Z)$.  For any $n \in \Z_+$ such that $0 < f < 1/n$, we define the \emph{projected sequential merge} function $\pi \mergeProj \xi \colon \R \to \R$ as $(\pi \mergeProj \xi)(x) = (\pi \merge \xi)(nx,x)$.  
Provided that $\pi$ is a facet of $R_f(\R,\Z)$ and $[\pi]_f$ is non-decreasing
and $\bar \Pi^{E(\pi)}(\R,\Z)$ has only the trivial solution, we have that $\pi \mergeProj \xi$ is a facet for
$R_{nf}(\R,\Z)$.  See~\autoref{tab:compendium-procedures} for an example of a
projected sequential merge inequality,
\sagefunc{dr_projected_sequential_merge_3_slope}. Also \sagefunc{dg_2_step_mir} from \autoref{tab:compendium-1} can be seen as the projected sequential merge function $\xi \mergeProj \xi$.  
We can state this idea in the following more general way.   Consider $(\pi_1 \merge (\pi_2 \merge \dots (\pi_{k-1} \merge \pi_k) \ldots ))$, where $\pi_i$ is a facet for $R_f(\R,\Z)$ and $[\pi_i]_f$ is non-decreasing, and $\bar\Pi^{E(\pi^i)}(\R,\Z)$ has only the trivial solution for $i=1, \dots, k$.  Let $n \in \Z_+$ such that $0 < f_k < \tfrac{1}{n}$.  Then  
$(\pi_1 \merge (\pi_2 \merge \dots \merge (\pi_{k-1} \merge (\pi_k \mergeProj \xi))\ldots ))$ is a facet for $R_{\f'}(\R^k, \Z^k)$ where $\f' = (f_1, \dots, f_{k-1}, nf_k)$~\cite[Theorem 6]{dey2}.

%

\section{Sequences of minimal valid and extreme functions}\label{s:limits}

\subsection{Minimality of limits of minimal valid functions}\label{s:minimal-limits}
The most basic topology on the space $\R^G$ of real-valued functions on $G$ is
the product topology, or the topology of pointwise convergence. We first
note that the properties in the characterization of minimal valid functions
(Theorem~\ref{thm:minimal}) are preserved under 
pointwise convergence. 
\begin{proposition}[{\cite[Proposition
4]{dey1}}]\label{prop:sequence-minimal}
Let $\pi_i \in \R^G$, $i \in \N$ be a sequence\footnote{The statement of
  \autoref{prop:sequence-minimal} remains true for generalizations of
  sequential limits; for example, we may consider the convergence of nets of
  minimal functions.} 
of minimal valid functions that converge pointwise to $\pi \in \R^G$. Then $\pi$ is a minimal valid function.
\end{proposition}

\begin{proof}
Since each $\pi_i$ is nonnegative, $\pi$ is nonnegative. We simply verify the conditions in Theorem~\ref{thm:minimal} for $\pi$.
\begin{enumerate}
\item For any $\ve w \in S$, $\pi(\ve w) = \lim_{i\to \infty} \pi_i(\ve w) = \lim_{i\to \infty} 0 = 0$.
\item For any $\ve x, \ve y \in G$, $\pi(\ve x + \ve y) = \lim_{i\to \infty} \pi_i(\ve x + \ve y) \leq \lim_{i 
\to \infty} (\pi_i(\ve x) + \pi_i(\ve y)) =  \lim_{i \to \infty} \pi_i(\ve x) +  \lim_{i \to \infty}\pi_i(\ve y) = \pi(\ve x) + \pi(\ve y)$.
\item For any $\ve x$, $\pi(\ve x) + \pi( \ve f - \ve x) = \lim_{i \to \infty} \pi_i(\ve x) +  \lim_{i \to \infty}\pi_i(\ve f -\ve x) = \lim_{i \to \infty} (\pi_i(\ve x) + \pi_i(\ve f - \ve x)) = \lim_{i\to \infty} 1 = 1.$\qedhere\pushQED{\relax}
\end{enumerate}\end{proof}

This result can be used to prove \autoref{obs:compact} regarding the
compactness of the set of minimal functions.

\begin{proof}[Proof of \autoref{obs:compact}] Theorem~\ref{thm:minimal}
  implies that all minimal valid functions $\pi$ satisfy $0 \leq \pi \leq
  1$. The set of functions in $\R^G$ bounded between $0$ and $1$ is compact by
  Tychonoff's theorem. \autoref{prop:sequence-minimal} applies to nets of
  minimal functions also, which is a generalization of sequences; this shows
  that the set of minimal valid functions is a closed subset of the set of
  functions in $\R^G$ bounded between $0$ and $1$. As a closed subset of a
  compact set, the set of minimal functions is compact.
\end{proof}

\subsection{Failure of extremality of limits of extreme functions}
\label{s:limits-extremality-not-preserved}


While minimality is preserved by limits, this is not true in general for
extremality. 

Dey and Wolsey~\cite[section 2.2, Example 2]{dey1} give an example where a
sequence of continuous piecewise linear extreme functions of type
\sagefunc{gj_2_slope_repeat} converges
pointwise to a 
discontinuous piecewise linear minimal valid function that is not extreme
(\autoref{fig:drlm_gj_2_slope_extreme_limit_to_nonextreme}).%
\footnote{The sequence and its limit can be constructed using \sagefunc{drlm_gj_2_slope_extreme_limit_to_nonextreme}.}
\newcommand\LimitFigures[3]{
  \begin{minipage}[c]{0.3\linewidth}\includegraphics[width=\linewidth]{#1}\end{minipage}%
  \begin{minipage}[c]{0.3\linewidth}\includegraphics[width=\linewidth]{#2} \end{minipage}%
  \begin{minipage}[c]{0.07\linewidth}{\Huge$\longrightarrow$}\end{minipage}%
  \begin{minipage}[c]{0.3\linewidth}\includegraphics[width=\linewidth]{#3} \end{minipage}
}
\newcommand\LimitFiguresFour[4]{
  \begin{minipage}[c]{0.23\linewidth}\includegraphics[width=\linewidth]{#1}\end{minipage}%
  \begin{minipage}[c]{0.23\linewidth}\includegraphics[width=\linewidth]{#2} \end{minipage}%
  \begin{minipage}[c]{0.23\linewidth}\includegraphics[width=\linewidth]{#3} \end{minipage}%
  \begin{minipage}[c]{0.07\linewidth}{\Huge$\longrightarrow$}\end{minipage}%
  \begin{minipage}[c]{0.23\linewidth}\includegraphics[width=\linewidth]{#4} \end{minipage}
}
\begin{figure}[t]
  \centering
  \LimitFigures{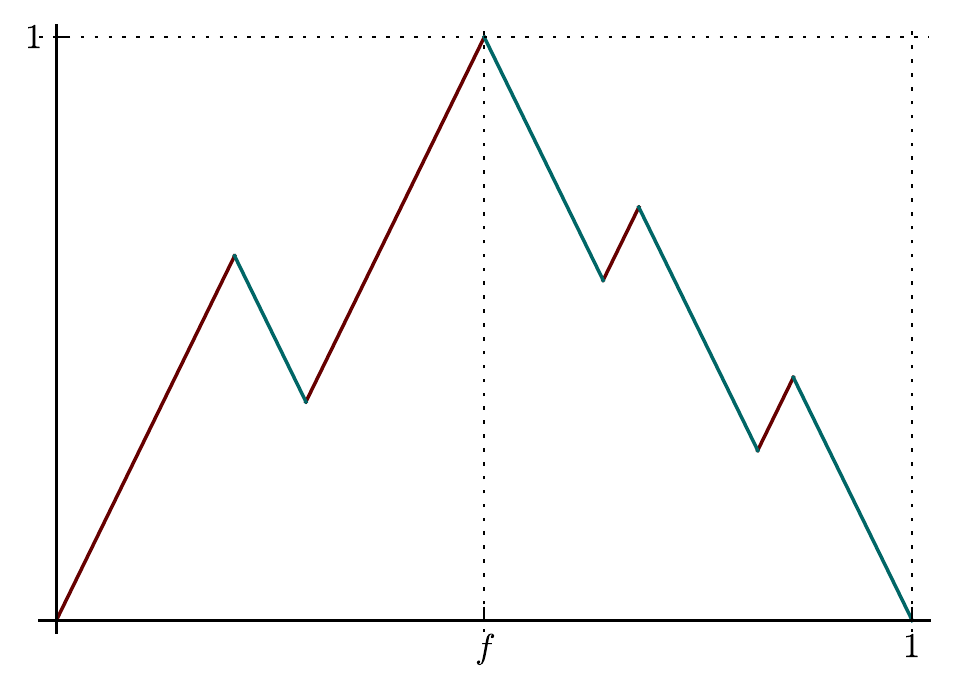}{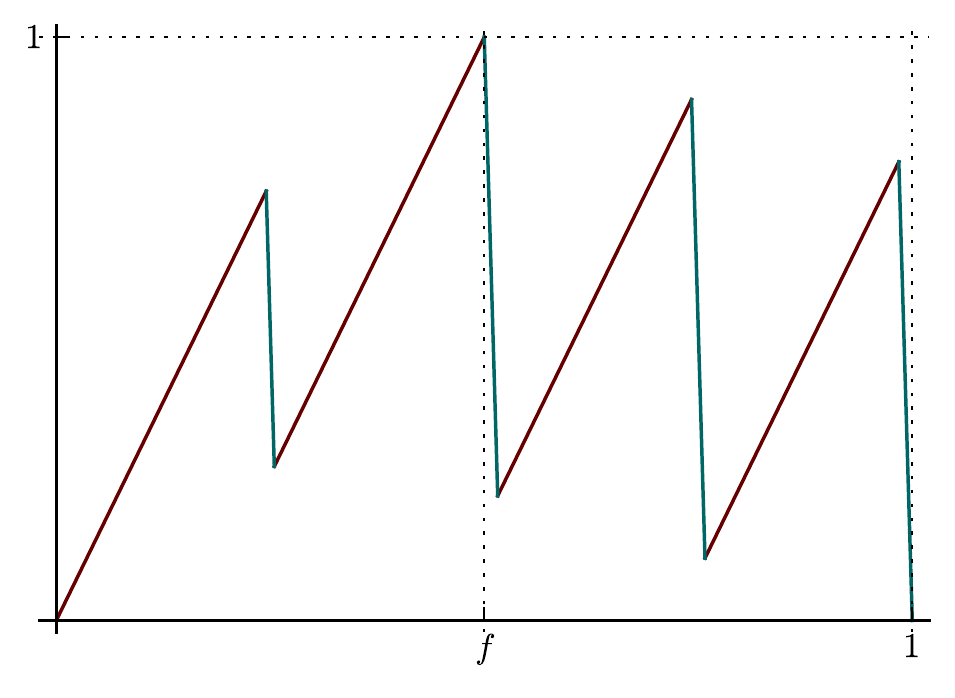}{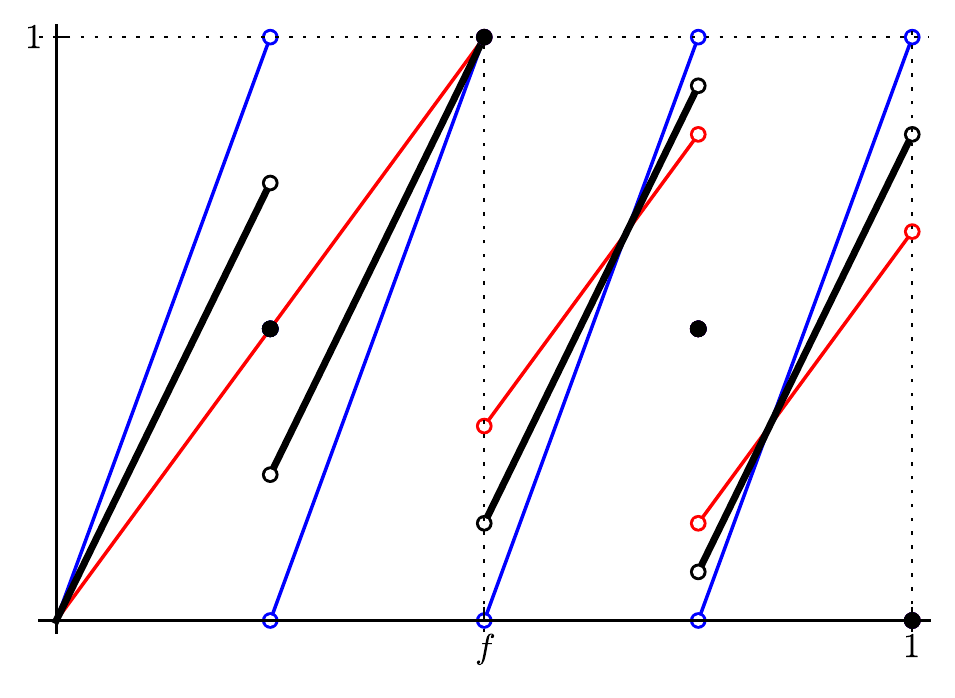}
  \caption{A pointwise limit of extreme functions that is not extreme
    \cite[section 2.2]{dey1}.  Consider the sequence of continuous extreme
    functions of type
    \sagefunc{gj_2_slope_repeat} set up for any $n \in \Z_+$ by
    \sage{h =
      \sagefunc{drlm_gj_2_slope_extreme_limit_to_nonextreme}(n)}.  For example,  $n = 3$ \emph{(left)} and $n=50$ \emph{(center)}.  This sequence 
    converges to a
    non-extreme discontinuous minimal valid function, set up with \sage{h =
      \sagefunc{drlm_gj_2_slope_extreme_limit_to_nonextreme}()}
    \emph{(right)}.  The limit function $\pi$ (\emph{black}) is shown with two
    minimal functions $\pi^1$ (\emph{blue}), $\pi^2$ (\emph{red}) such that
    $\pi = \frac12(\pi^1+\pi^2)$.} 
  \label{fig:drlm_gj_2_slope_extreme_limit_to_nonextreme}
\end{figure}

This convergence, of course, is not uniform. 
One may then ask whether extremality is preserved by stronger notions of
convergence.  However, even uniform convergence (i.e., convergence in the
sense of the space $C(\R)$ of continuous functions) 
or convergence in the sense of the Sobolev space\footnote{See, for example,
  \cite{hunter-nachtergaele} for an introduction to Sobolev spaces.}
$W^{1,1}_{\mathrm{loc}}(\R)$ do not suffice to ensure extremality of the limit function
(\autoref{fig:bhk_irrational_extreme_limit_to_rational_nonextreme}).

\begin{prop}[\TheoremNEWRESULT]\label{prop:uniform-conv}
There exists a sequence of continuous extreme functions of type \sagefunc{bhk_irrational} \cite[section~5]{basu-hildebrand-koeppe:equivariant} that converges uniformly to a continuous non-extreme function of the
same type.  Further, even the sequence of generalized derivatives converges
in the sense of $L^1_{\mathrm{loc}}(\R)$; thus we have convergence in
$W^{1,1}_{\mathrm{loc}}(\R)$.
\end{prop}
 
The functions from~\autoref{prop:uniform-conv} have the intriguing property that extremality depends, in
addition to some inequalities in the parameters, on the $\Q$-linear
independence of two real parameters \cite[Theorems 5.3 and
5.4]{basu-hildebrand-koeppe:equivariant}.\footnote{These parameters are
  collected in the list \sage{delta}, which is an argument to the function
  \sagefunc{bhk_irrational}. 
  The parameters are $\Q$-linearly independent for example when one
  parameter is rational, e.g., \sage{1/200}
  , the other irrational,
  e.g., \sage{sqrt(2)/200}
  . When the irrational number is algebraic (for example, when it is
  constructed using square roots), the code will 
  construct an appropriate real number field that is a field extension of the
  rationals.  In this field, the computations are done in exact arithmetic.}
 Thus it is easy to construct a sequence of parameters satisfying this
 condition whose limit is rational, making the limit function
 non-extreme.\footnote{Such a sequence and the limit can be constructed using
   \sagefunc{bhk_irrational_extreme_limit_to_rational_nonextreme}.} 

\begin{figure}[t]
  \centering
  \LimitFigures{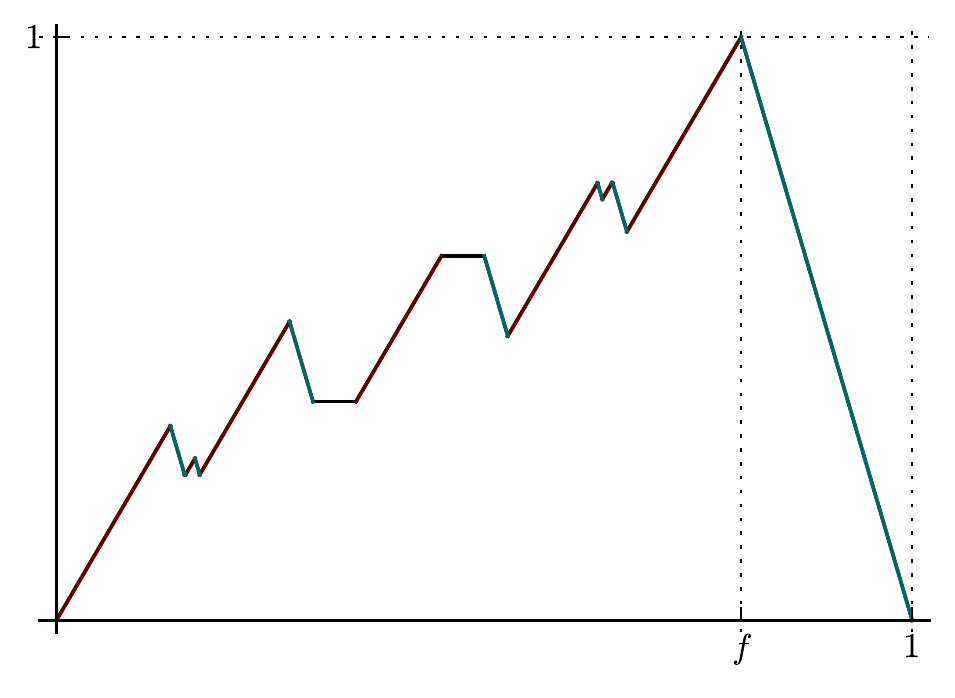}{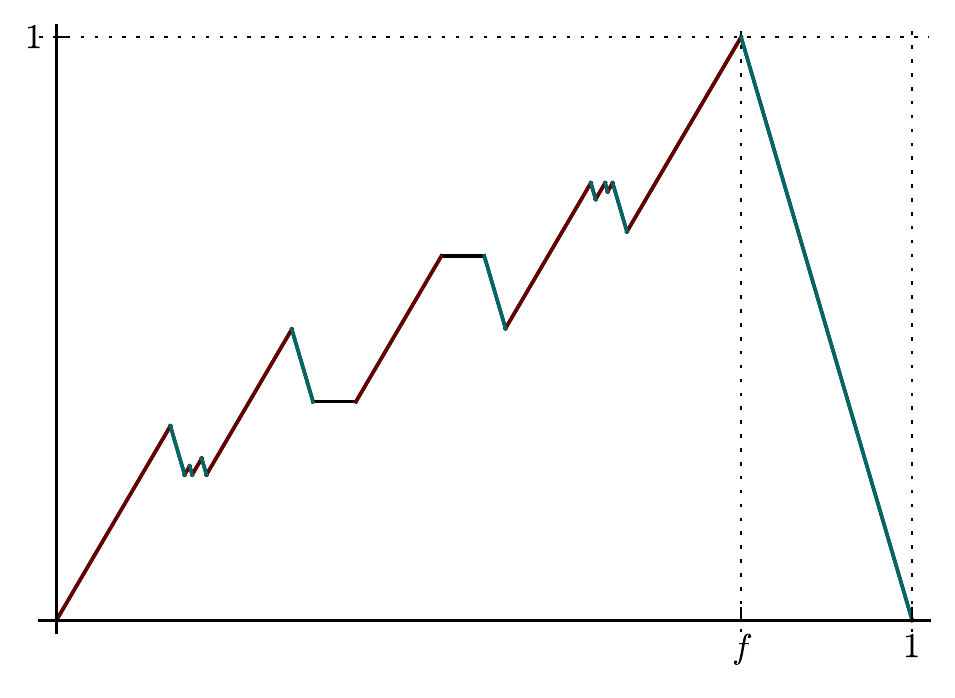}{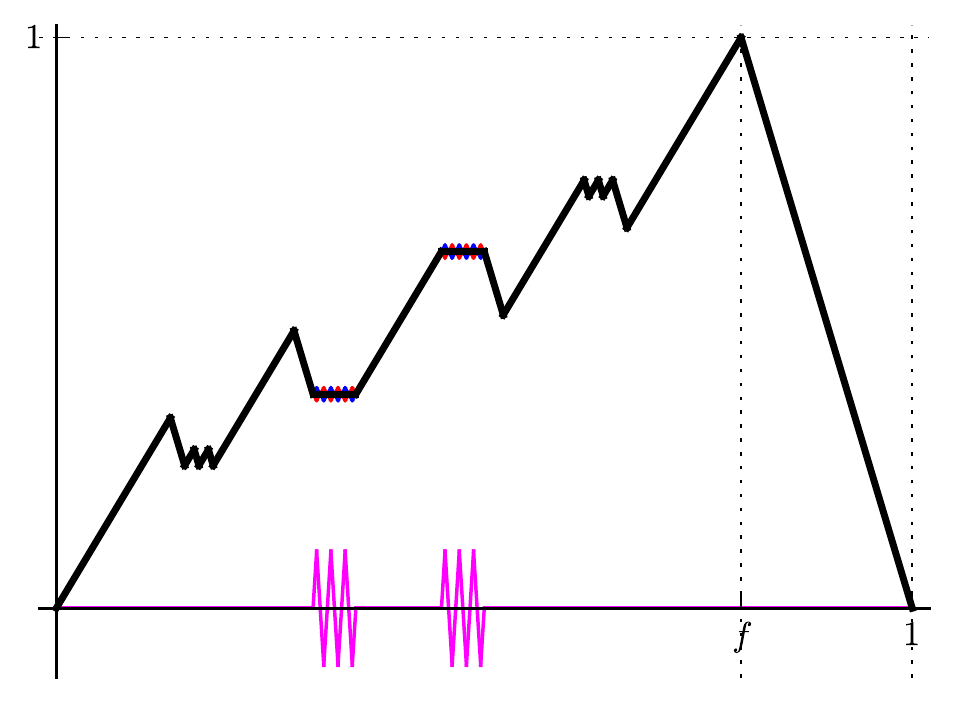} 
  \caption{A uniform limit of extreme functions that is not extreme.  
    The sequence of extreme functions of type \sagefunc{bhk_irrational}, set up with \sage{h =
      \sagefunc{bhk_irrational_extreme_limit_to_rational_nonextreme}(n)} where $n =
    1$ \emph{(left)}, $n=2$ \emph{(center)}, \dots\ converges to a non-extreme
    function, set up with \sage{h =
      \sagefunc{bhk_irrational_extreme_limit_to_rational_nonextreme}()}
    \emph{(right)}.
    The limit function $\pi$ (\emph{black}) is shown with two
    minimal functions $\pi^1$ (\emph{blue}), $\pi^2$ (\emph{red}) such that
    $\pi = \frac12(\pi^1+\pi^2)$ and a scaling of the perturbation
    function~$\bar\pi = \pi^1 - \pi$ (\emph{magenta}).} 
  \label{fig:bhk_irrational_extreme_limit_to_rational_nonextreme}
\end{figure}

\subsection{Discontinuous extreme piecewise linear limit functions} 

Dey and Wolsey~\cite{dey1} give some general conditions under which the limit is indeed extreme. Recall that a function $\pi \in \R^\R$ is called {\em piecewise linear} (not necessarily continuous) if we can express $\R$ as the union of closed intervals with non-overlapping interiors such that any bounded subset of $\R$ intersects only finitely many intervals, and the function is affine linear over the {\em interior} of each interval.

\begin{theorem}[{\cite[Theorem 7]{dey1}}]
\label{thm:dey-convergence}
Let $\pi_i \in \R^\R$, $i \in \N$ be a sequence of continuous piecewise linear, extreme valid functions for $R_{f}(\R, \Z)$ and let $\phi$ be the pointwise limit of the sequence $\pi_i$, $i\in \N$ such that the following conditions hold:

\begin{itemize}
\item[(i)] $\phi$ is piecewise linear (not necessarily continuous).
\item[(ii)] $\phi$ has a finite right derivative at $0$.\footnote{This can also be done with a finite left derivative.  Note that not all extreme functions have a finite left or right derivative at the origin.  That is, there exist extreme functions that are discontinuous on both sides of the origin.  See Table~\ref{tab:compendium-3} for examples.}
\item[(iii)] There is a sequence of integers $k_i$, $i \in \N$ with $\lim_{i\to\infty} k_i = \infty$ such that for each $i\in \N$,
\begin{itemize}
\item[(a)] $\phi(u) = \pi_i(u)$ for all $u \in \frac1{k_i}\Z$ and
\item[(b)] the set of nondifferentiable points of $\pi_i$ is contained in $\frac1{k_i}\Z$.
\end{itemize}
\end{itemize}
Then $\phi$ is extreme.

\end{theorem}

The authors of~\cite{dey1} use the above theorem to construct families of
discontinuous piecewise linear extreme functions for the single-row
infinite group problem; see \autoref{tab:compendium-2} for a list.  The use of
\autoref{thm:dey-convergence} does not seem to be essential, however; the
extremality of all of these functions can also be established by following the
algorithm of \autoref{sec:direct-algorithm}. 

\subsection{Non--piecewise linear extreme limit functions} 
\label{sec:bccz_counterexample}

We now describe a construction based on limits of extreme functions that yields an extreme function that is not piecewise linear. The extremality of this limit function cannot be obtained by an application of~\autoref{thm:dey-convergence} since the limit function is not piecewise linear.


This construction is motivated by a conjecture of Gomory and Johnson from 2003 that all facets are piecewise linear~\cite[section 6.1]{tspace}. If true, this would justify focusing attention on piecewise linear minimal valid functions, for which we have developed many tools for analysis (see \autoref{sec:piecewise}). However, even for $k=1$, this conjecture was disproved by Basu, Conforti, Cornu\'ejols and Zambelli~\cite{bccz08222222}. We present their counterexample and a brief argument for its extremality. 

\begin{remark}\label{rem:facet-non-piece} The arguments for its facetness are almost identical; the only difference is that some technical continuity arguments can be avoided in the proof of extremality because of Lemma~\ref{lem:tightness} (iii). \end{remark} 

We first define a sequence of valid functions $\psi_i \colon  \R
\rightarrow \mathbb{R}$ 
that are piecewise linear, and then consider
the limit $\psi$ of this sequence, which will be extreme but not piecewise linear.

\begin{figure}[tp]
  \centering
  \LimitFiguresFour{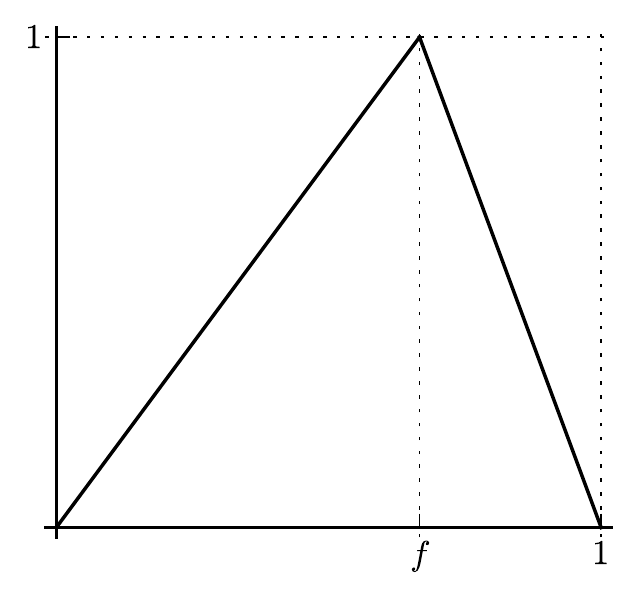}{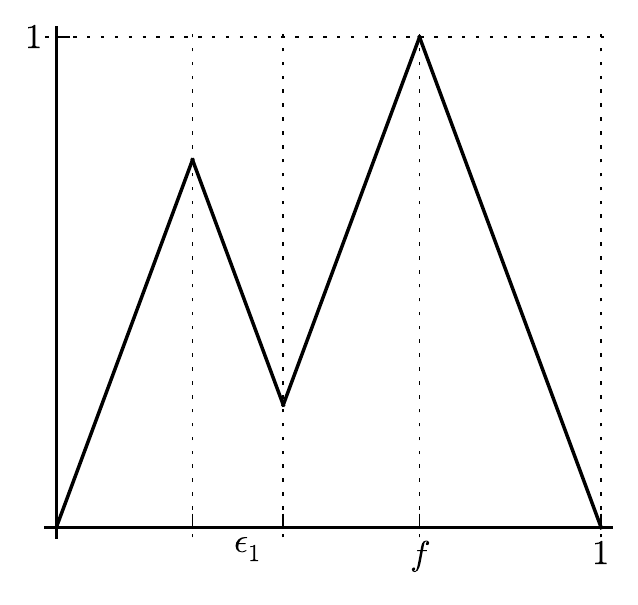}{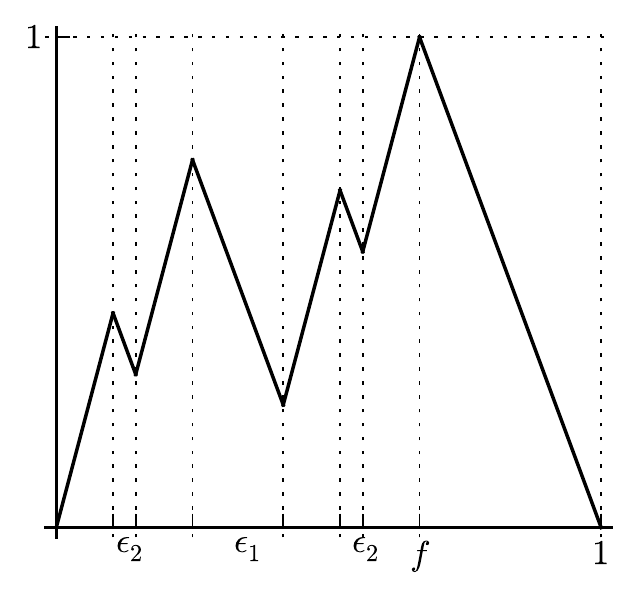}{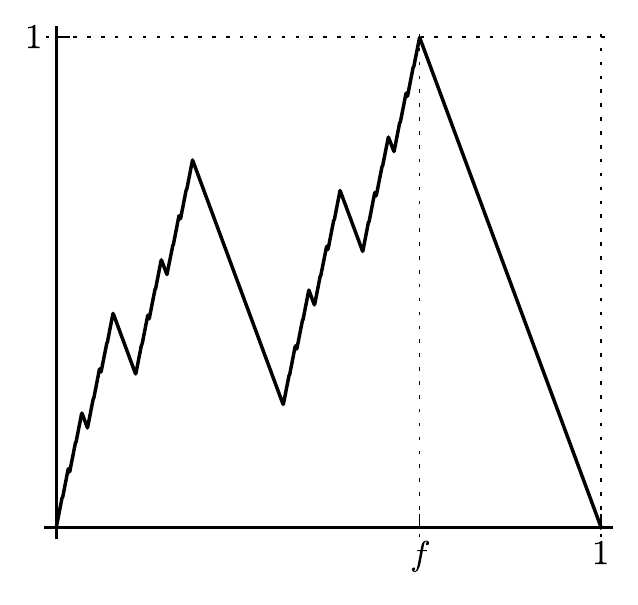}
  \caption{First steps ($\psi_0 = \sage{\sagefunc{gmic}()}$, $\psi_1$, $\psi_2$) in the
    construction of the continuous non--piecewise linear limit function $\psi
    = \sage{\sagefunc{bccz_counterexample}()}$.}
  \label{fig:limit_function}
\end{figure}

Let $0 < f < 1$.  Consider a geometric sequence of real numbers $\epsilon_1 > \epsilon_2 >
\dots$ such that $\epsilon_1 \leq 1-f$ and
\begin{equation} \label{series}
  \mu^- = (1-f) + \sum_{i=1}^{+\infty} 2^{i-1} \epsilon_i \leq 1
\end{equation}
holds.\footnote{The first $n$ terms of such a sequence of~$\epsilon_i$
  are generated by \sage{e = \sagefunc{generate_example_e_for_psi_n}(n=$n$)}.}
We distinguish two cases:  $\mu^- < 1$ \cite{bccz08222222} and $\mu^- =1$
\cite{zhou:extreme-notes}. 
An example for the first case is the sequence $ \epsilon_i = (\frac{1}{4})^i
f$ for $0 < f \leq \frac{4}{5}$; for the second case,
$\epsilon_i = 2 (\frac{1}{4})^i f$ for $0 < f \leq \frac12$.
Let $\psi_0$ be the \sagefunc{gmic} function with peak at~$f$. 
%
We construct $\psi_{i+1}$ from $\psi_i$ by modifying each segment
with positive slope in the graph of $\psi_i$ in the manner of the
\sagefunc{kf_n_step_mir} construction \cite{kianfar1} as follows.\footnote{The
  construction of $\psi_n$ is furnished by
  \sage{h = \sagefunc{psi_n_in_bccz_counterexample_construction}(e=e)},
  where \sage{e} is the list \sage{[$\epsilon_1, \dots, \epsilon_n$]}.
} 
For every inclusion-maximal 
interval $[a,b]$ where $\psi_i$ has constant positive slope we
replace the line segment from $(a,\psi_i(a))$ to $(b,\psi_i(b))$
with the following three segments:
\begin{itemize}
\item a positive slope segment connecting $\bigl(a,\psi_i(a)\bigr)$
and
$\bigl(\frac{(a+b)-\epsilon_{i+1}}{2},\psi_i(\frac{a+b}{2})+\frac{\epsilon_{i+1}}{2(1-f)}\bigr)$,
\item a negative slope segment connecting  $\bigl(\frac{(a+b)-\epsilon_{i+1}}{2},
\psi_i(\frac{a+b}{2})+\frac{\epsilon_{i+1}}{2(1-f)}\bigr)$ and
$\bigl(\frac{(a+b)+\epsilon_{i+1}}{2},
\psi_i(\frac{a+b}{2})-\frac{\epsilon_{i+1}}{2(1-f)}\bigr)$,
\item a positive slope segment connecting $\bigl(\frac{(a+b)+\epsilon_{i+1}}{2},
\psi_i(\frac{a+b}{2})-\frac{\epsilon_{i+1}}{2(1-f)}\bigr)$ and $\bigl(b,
\psi_i(b)\bigr)$.
\end{itemize}
Figure~\ref{fig:limit_function} shows the transformation of $\psi_0$
to $\psi_1$ and $\psi_1$ to $\psi_2$.  Each $\psi_i$ is nonnegative, subadditive and
satisfies the symmetry condition \cite[Lemma 4.5 and Fact 4.6]{bccz08222222},
and thus is a minimal valid function.  By construction, the new negative
slopes match the existing negative slopes, and the new positive slopes of each
function have all the same slope.  Thus $\psi_i$ is a (continuous piecewise
linear) 2-slope function and  
hence extreme. The functions $\psi_i$ are therefore extreme functions by the
Gomory--Johnson 2-Slope Theorem (\autoref{th:2-slope}).


The function $\psi$ which we show to be extreme but not piecewise
linear is defined as the pointwise limit of this sequence of functions, namely
\begin{equation}
\psi(x) = \lim_{i \rightarrow \infty} \psi_i(x).\label{PSI}
\end{equation}
This limit is well defined when \eqref{series} holds.\footnote{The function
  can be created by \sage{h = \sagefunc{bccz_counterexample}()}; however, \sage{h(x)} can be
  exactly evaluated only on the set $\bigcup_{i=0}^\infty \cl N_i$ defined
  below; for other values, the function will return an approximation.}  
In fact, $\psi_i$
converges uniformly to $\psi$. 
Since each $\psi_i$ is continuous, this implies
that $\psi$ is also continuous.\footnote{In fact, if $\mu^- < 1$, then $\psi$
  is actually Lipschitz continuous and thus 
  absolutely continuous and hence almost everywhere differentiable. The convergence then holds even in the sense of the
  space~$W^{1,1}_{\mathrm{loc}}(\R)$.} The limit function has the following
intriguing properties:
\begin{enumerate}
\item By \autoref{prop:sequence-minimal}, $\psi$ is minimal. 
\item For each integer $i \geq 0$, define $X^-_i$ to be the subset of points of
  $[0,1]$  
on which the function $\psi_i$ is differentiable with a negative slope. From the construction of $\psi_i$, $X^-_i$ is the union of $2^i$ open intervals \cite[Fact 4.1]{bccz08222222}. Furthermore, $X^-_i\subseteq X^-_{i+1}$ for every $i\in\mathbb{N}$. The set $X^-\subseteq [0,1]$ defined by
$X^-=\bigcup_{i=0}^\infty X^-_i$ is thus the set of points over which $\psi$ has
negative slope, and it is an open set since it is the union of open intervals.
The set $X^-$ is dense in $[0,1]$ \cite[Fact 5.4]{bccz08222222}.  Its Lebesgue measure
is~$\mu^-$. 
\item $\psi$ is not piecewise linear. This is because each $\psi_i$ is nonnegative, and therefore so is $\psi$. If $\psi$ is piecewise linear, by definition of continuous piecewise linear functions from~\autoref{sec:def-complex-piecewise} there exists $0 < \epsilon$ such that $\psi$ is affine linear on $[0,\epsilon]$. Since $X^-$ is dense, there exists a point from $X^-$ in $(0,\epsilon)$ and so $\psi$ has negative slope on this entire segment. But since $\psi(0) = 0$, this contradicts the fact that $\psi \geq 0$.
\item The complement $[0,1]\setminus X^-$ is a closed set, which does not
  contain any interval; hence $[0,1]\setminus X^-$ is a nowhere dense set.  It
  has Lebesgue measure $\mu^+ = 1 - \mu^-$.  
  Removing from $[0,1]\setminus X^-$ the countably many breakpoints of the negative-slope intervals, we
  obtain the set $X^+ = [0,1]\setminus \bigcup_{i=0}^\infty \cl X^-_i$, which is still a
  nowhere dense set of measure~$\mu^+$.
\item If $\mu^- < 1$, the set $X^+$ is of positive measure, and thus a
  \emph{fat Cantor set}; in this case the derivative of~$\psi$ exists for all
  points in $X^+$ and
  equals the limit of the positive slopes of the functions~$\psi_i$.
  Thus $\psi$ is an absolutely continuous, measurable, non--piecewise
  linear ``2-slope function.'' 
\item On the other hand, if $\mu^- = 1$, the measure of~$X^+$ is zero, and 
  so the derivative of~$\psi$ equals the negative slopes of the
  functions~$\psi_i$ Lebesgue--almost everywhere.  Thus $\psi$
  is a continuous (but not absolutely continuous), measurable, non--piecewise
  linear ``1-slope function.''  This case is discussed in \cite{zhou:extreme-notes}.
\end{enumerate}
The proof of extremality of $\psi$ proceeds along the roadmap of \autoref{s:roadmap} as follows.
\begin{enumerate}
\item Consider any minimal valid functions $\pi^1, \pi^2$ such that $\psi =
  \frac{\pi^1 + \pi^2}{2}$. Since $\psi$ is affine over the segments in $\cl X^-_i$, the additivity properties on these segments are inherited by $\pi^1$ using a one-dimensional version of \autoref{lem:k+1-useful} and Lemma~\ref{lem:tightness} (ii).
\item One uses the Interval Lemma (\autoref{one-dim-interval_lemma}) on $\pi^1$ to obtain that $\pi^1$ is affine over $X^-$, and moreover, since $\pi^1(0) = \psi(0) = 0$ and $\pi^1(f) = \psi(f) =1$, one recursively establishes that $\pi^1(x) = \psi(x)$ for all $x\in X^-$. 
\item Since $X^-$ is dense in $[0,1]$, $\psi$ is continuous and $\pi^1$ is
  continuous by Lemma~\ref{lem:tightness} (iii) 
  we obtain that $\pi^1 = \psi$. Therefore, $\pi^1 = \pi^2 = \psi$, establishing that $\psi$ is extreme.
\end{enumerate}

We end this section with a conjecture about limits of minimal functions, whose positive resolution would emphasize the importance of piecewise linear functions. 

\begin{conjecture}[{\cite[Conjecture 6.1]{bccz08222222}}] Every extreme function (resp. facet) $\pi \colon \R^k \to \R$ is either piecewise linear or the limit of a sequence of piecewise linear extreme functions (resp. facets). 
\end{conjecture}

\section{Algorithmic characterization of extreme functions
}\label{sec:one-two-dim}
In this section we discuss recent algorithmic results for proving piecewise linear functions are either extreme or not extreme for the infinite group problem $R_\f(\R^k,\Z^k)$.  In~\cite{basu-hildebrand-koeppe:equivariant}, the first algorithmic test for
extremality was given for the single-row infinite group problem $R_f(\R,\Z)$, followed by an extension to two-row infinite group problem $R_\f(\R^2,\Z^2)$ in~\cite{bhk-IPCOext}.  We summarize these algorithmic ideas here in two lights.  We will first discuss a general procedure to test for extremality and then in section~\ref{sec:alg-restriction-finite-groups} discuss specific classes of functions that have relations to finite group problems where extremality can be tested easily using linear algebra.

\subsection{General procedure outline}
\label{sec:direct-algorithm}
%

We will outline here a general procedure for testing extremality of a continuous piecewise linear function $\pi\colon \R^k \to \R$ defined on a polyhedral complex $\P$.  Similar techniques may apply to testing extremality and even facetness of discontinuous piecewise linear functions as well.

Let $E = E(\pi)$.
Recall that  $\pi$ is not extreme if and only if there exists a nontrivial function $\bar \pi$ such that $\pi \pm \bar \pi$ is minimal.
From Lemma~\ref{lem:tightness} parts~\eqref{lem:tightness-minimal} and~\eqref{lem:tightness-subadditive} it follows that $\pi$ is not extreme if and only if there exists a nontrivial continuous function $\bar \pi \in \bar \Pi^E(\R^k, \Z^k)$ such that $\pi \pm \bar \pi$ is minimal.

Let $\T$ be a triangulation of $\R^k$ that satisfies the hypotheses of Lemma~\ref{lem:tri-restriction}, i.e., there exists $q \in \mathbb{N}$ such that $\verts(\T) = \tfrac{1}{q}\Z^k$ and $p_i(\verts(\Delta \T)) \subseteq \tfrac{1}{q} \Z^k$ for $i=1,2,3$ and $\f \in \tfrac{1}{q}\Z^k$.
The following algorithmic ideas are based on the decomposition in Lemma~\ref{lem:tri-restriction} part~\eqref{lem:tri-restriction-part2} of perturbations $\bar \pi \in \bar \Pi^E(\R^k, \Z^k)$ into $\bar \pi = \bar \pi_\T + \bar \pi_{\mathrm{zero}(\T)}$ with $\bar \pi_\T \in \bar \Pi^E_\T(\R^k, \Z^k)$ and $\bar \pi_{\mathrm{zero}(\T)} \in \bPizero$.  Since $\bar \pi$ and $\bar \pi_\T$ are continuous, $\bar\pi_{\mathrm{zero}(\T)}$ is also continuous.


\subsubsection{Finite-dimensional linear algebra for $\bar \Pi^E_\T(\R^k, \Z^k)$}
\label{section:system}
We begin by looking for a perturbation function in $\bar \Pi^E_\T(\R^k, \Z^k)$.
By Lemma~\ref{lem:tri-restriction}, $\bar \pi \in \bar \Pi^E_\T(\R^k, \Z^k)$
if and only if $\bar \pi|_{\frac{1}{q}\Z^k} \in \bar \Pi^{E'}(
\tfrac{1}{q}\Z^k, \Z^k)$ where $E' = E(\pi) \cap \tfrac{1}{q}\Z^k$.  Thus we
consider the linear system $\bar \Pi^{E'}( \tfrac{1}{q}\Z^k, \Z^k)$, which is
finite-dimensional if we identify the variables $\bar \pi(\x)$ and $\bar
\pi(\x + \t)$ for all $\t \in \Z^k$.  Hence, this is a finite-dimensional
linear system and any nontrivial solution can be computed by analyzing the
null space of this system.  
If such a nontrivial solution exists, it can be interpolated to a piecewise
linear function $\bpi \in \bar \Pi^E_\T(\R^k, \Z^k)$ because $\T$ is a triangulation of $\R^k$ that satisfies the hypotheses of Lemma~\ref{lem:tri-restriction}, and so by Theorem~\ref{corPerturb} $\pi$ is not extreme.  This is demonstrated in Figure~\ref{fig:minimalNotExtreme} for the case of $\T = \P = \P_{\frac{1}{q} \Z}$ where a perturbation is found on the complex $\T$. 

Otherwise we have that $\bar \Pi^E_\T(\R^k, \Z^k) = \{ 0\}$.  This scenario is depicted in Figure~\ref{fig:drlm_not_extreme_1} where $\bar \Pi^{E'}_{\P_{\frac{1}{q} \Z}}(\tfrac{1}{q} \Z, \Z) = \{0\}$ with $E' = E(\pi) \cap \tfrac{1}{q} \Z^2$.  




\subsubsection{Projections and additivity for $\bPizero$}

From Lemma~\ref{lem:tri-restriction}, any $\bar \pi \in \bPizero$ satisfies $\bar \pi|_{\frac{1}{q}\Z^k} \equiv 0$.  

We consider full-dimensional faces $F\in E(\pi, \P)$.  By 
 Corollary~\ref{lem:projection_interval_lemma-corollary}, these full-dimensional faces imply that any $\bar \pi \in \bPizero$ is affine on the projections $p_1(F), p_2(F)$, and $p_3(F)$.  If a projection $p_i(F)$ contains $k+1$ affinely independent points in $\frac{1}{q} \Z^k$, then we conclude that $\bar \pi|_{p_i(F)} \equiv 0$ on this projection.  This is because $\bar \pi|_{\frac{1}{q} \Z^k} \equiv 0$.  Therefore, we learn certain polyhedral regions where $\bar \pi$ vanishes and we record these.

In the next step, we consider any faces $F$ of $E(\pi, \P)$ such that two of $p_1(F), p_2(F), p_3(F)$ are full-dimensional and one is zero-dimensional.  In particular, if one of these full-dimensional projections intersects a region where $\bar \pi$ is zero, then that property is transferred to the other full-dimensional projection.  
For example, the relations $\pi(\x) + \pi(\t) = \pi(\x+ \t)$ for all $\x \in I$ corresponds to the face $F = F(I, \{\t\}, I + \{\t\})$ where $p_1(F) = I$, $p_2(F) = \{\t\}$, $p_3(F) = I + \{\t\}$.  Hence, if $I$ is full-dimensional in $\R^k$ then $I + \{\t\}$ is full-dimensional in $\R^k$.
In this way the function values of $\bar \pi$ in $I + \{\t\}$ are dependent on the function values on $I$.  For example, if we know that $\bar \pi$ is affine over $I$, then it is also affine over $I + \{\t\}$.
This is the key step in this procedure.
We continue transferring properties until no new affine properties are discovered.

If the procedure terminates, it may either show that $\bar \pi \equiv 0$, in which case $\pi$ is extreme.  
Otherwise, we hope to find a perturbation function $\bar \pi$ that shows that $\pi$ is not extreme.   In fact, in certain cases, we can find a $\bar \pi$ that is piecewise linear on a refinement of $\T$.
Showing termination of this procedure is non-trivial and it is an open
question under what conditions this procedure is guaranteed to
terminate. Subsections~\ref{subsec:single-row-rational-breakpoints}
and~\ref{subsec:standard-tri-R2} discuss cases in which the procedure provably terminates.

The above procedure only considers certain faces of $E(\pi, \P)$.  Other faces
of $E(\pi,\P)$, as shown in Theorem~\ref{lem:generalized_interval_lemma},
establish other affine properties about $\bar \pi$, but not necessarily
full-dimensional affine properties.  These properties can sometimes combine to
create full-dimensional affine properties.  This effect is investigated in the
forthcoming paper~\cite{basu-hildebrand-koeppe:equivariant-general-2dim} for
the case of the two-row problem and general continuous piecewise linear functions
over the complex $\P_q$.

\subsection{One-row case with rational breakpoints}
  \label{subsec:single-row-rational-breakpoints}
%
 
We will consider the one-dimensional polyhedral
complex $\P_{B}$ for $B = \frac{1}{q} \Z \cap [0,1)$ as defined in Example~\ref{ex:1d-breakpoint-complex};  we will call this complex $\P_{\frac{1}{q} \Z}$.  Therefore, we consider piecewise linear functions (possibly discontinuous) with breakpoints in $\frac{1}{q} \Z$.

\begin{theorem}[{\cite[Theorem 1.3]{basu-hildebrand-koeppe:equivariant}}]
\label{thm:one-dim}
Consider the following problem.  
\begin{quote}
  Given a minimal 
  valid function $\pi$ for $R_f(\R,\Z)$ that is piecewise
  linear with a set of rational breakpoints with the least common
  denominator~$q$, decide if $\pi$ is extreme or not.
\end{quote}
There exists an algorithm for this problem whose running time is bounded by a
polynomial in $q$. 
\end{theorem}
Since the above algorithm is polynomial in the least common
denominator~$q$, it is only a pseudo-polynomial time algorithm.

\begin{openquestion}
Does there exist a polynomial time algorithm to determine extremality of piecewise linear functions for $R_f(\R,\Z)$?
\end{openquestion}

A more general version of the above algorithm is implemented in
\cite{infinite-group-relaxation-code} for the case of piecewise linear
functions, which are allowed to be continuous or discontinuous, and whose data may be algebraic irrational numbers.\footnote{If \sage{h} is the function~$\pi$, e.g., after typing
  \sage{h = \sagefunc{dg_2_step_mir}()}, then the algorithm is invoked by typing
  \sage{\sagefunc{extremality_test}(h, show\underscore{}plots=True)}.
  In the irrational case no proof of finite convergence of the
  procedure is known.}
%
%
The implementation will be described in more detail in a forthcoming article.

\subsection{Two-row case using a standard triangulation of~$\R^2$}
\label{subsec:standard-tri-R2}

For the case of the standard triangulations~$\P_q$ of~$\R^2$
(\autoref{ex:2d-standard-triangulation}),
\cite{basu-hildebrand-koeppe:equivariant-2,bhk-IPCOext} 
describe an algorithm of the above scheme for a special class of piecewise
linear functions over this complex, which are said to be \emph{diagonally
  constrained}. 

Let 
$$
A =  \begin{bmatrix}  1 & -1 &  0 & 0 &  1 & -1 \\
 0 & 0 & 1 & -1  & 1 & -1  \end{bmatrix}^T.
$$
Then for every face $I \in \P_q$, there exists a vector $\ve b \in \frac1q\Z^6$ such that $I = \{\, \x \st A \x \leq \ve b\,\}.$ Furthermore, for every vector $\ve b \in \frac1q\Z^6$, the set $\{\, \x \st A \x \leq \ve b\,\}$ is a union of faces of $\P_q$ (possibly empty), since each inequality corresponds to a hyperplane in the arrangement $\mathcal{H}_q$.
The matrix $A$ is totally unimodular and this fact plays a key role in proving the following lemma.


%
%

\begin{lemma}
  \label{lemma:vertices}
  Let $F \in \Delta \P_q$. Then the projections $p_1(F)$, $p_2(F)$, and
    $p_3(F)$ are faces in the complex~$\P_q$.
  In particular, let $(\x,\y)$ be a vertex of $\Delta\P_q$.  Then
    $\x,\y$ are vertices of the complex $\P_q$, i.e., $\x,\y\in
    \frac{1}{q}\Z^2$.
\end{lemma}


%
%
%
%


Extremality is more easily studied if we restrict ourselves to a setting determined by the types of faces $F \in E_{\max{}}(\pi, \P_q)$. Recall that $$
E_{\max{}}(\pi,\P_q) 
= \setcond{F\in E(\pi, \P_q)}{F \text{ is a maximal face by set inclusion in } E(\pi, \P_q)}. 
$$

\begin{definition}  A continuous piecewise linear function $\pi$ on $\P_q$ is
  called \emph{diagonally constrained} if for all $F \in E_{\max{}}(\pi, \P_q)$ and $i=1,2,3$, the projection $p_i(F)$ is either a vertex, diagonal edge, or triangle from the complex $\P_q$.
  %
\end{definition}
The properties in Lemma~\ref{lemma:vertices} provide an easy method to compute $E(\pi, \P_q)$ and test if a function is \emph{diagonally constrained} by using simple arithmetic and set
membership operations on vertices of~$\P_q$.  



\begin{figure}[t]

\begin{center}
\includegraphics[scale=0.47]{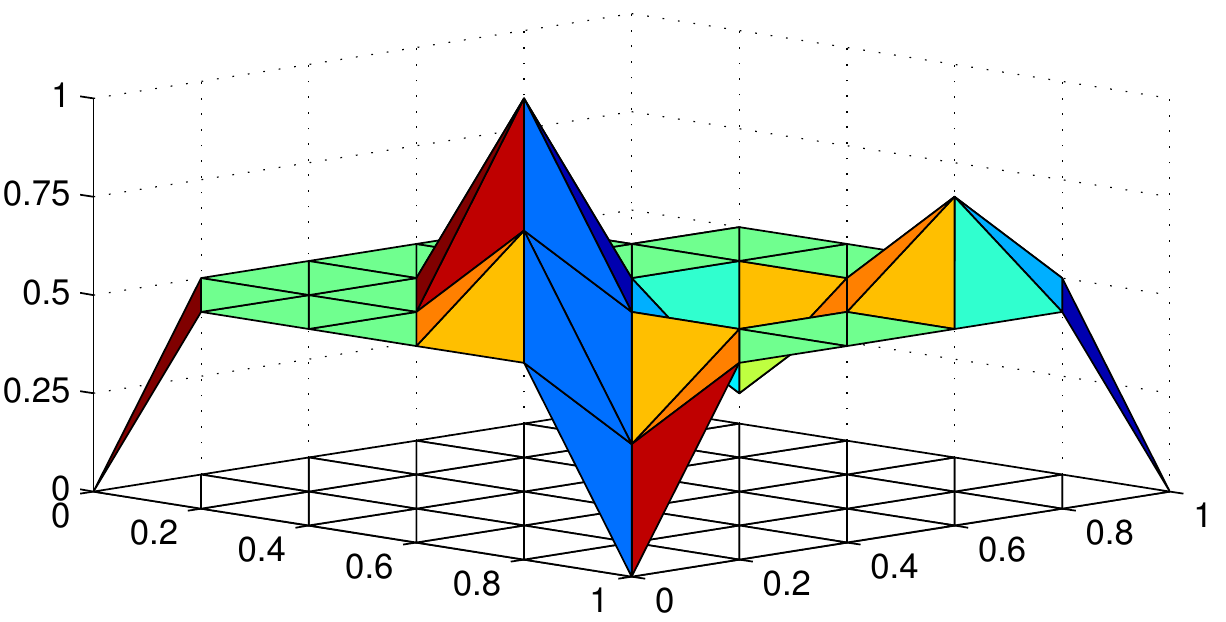} \ \ \ \ 
\begin{tikzpicture}[scale=0.7]
\begin{scope}[xshift = 2.5cm, yshift=2.5cm]
    \node {\includegraphics[scale=.18]{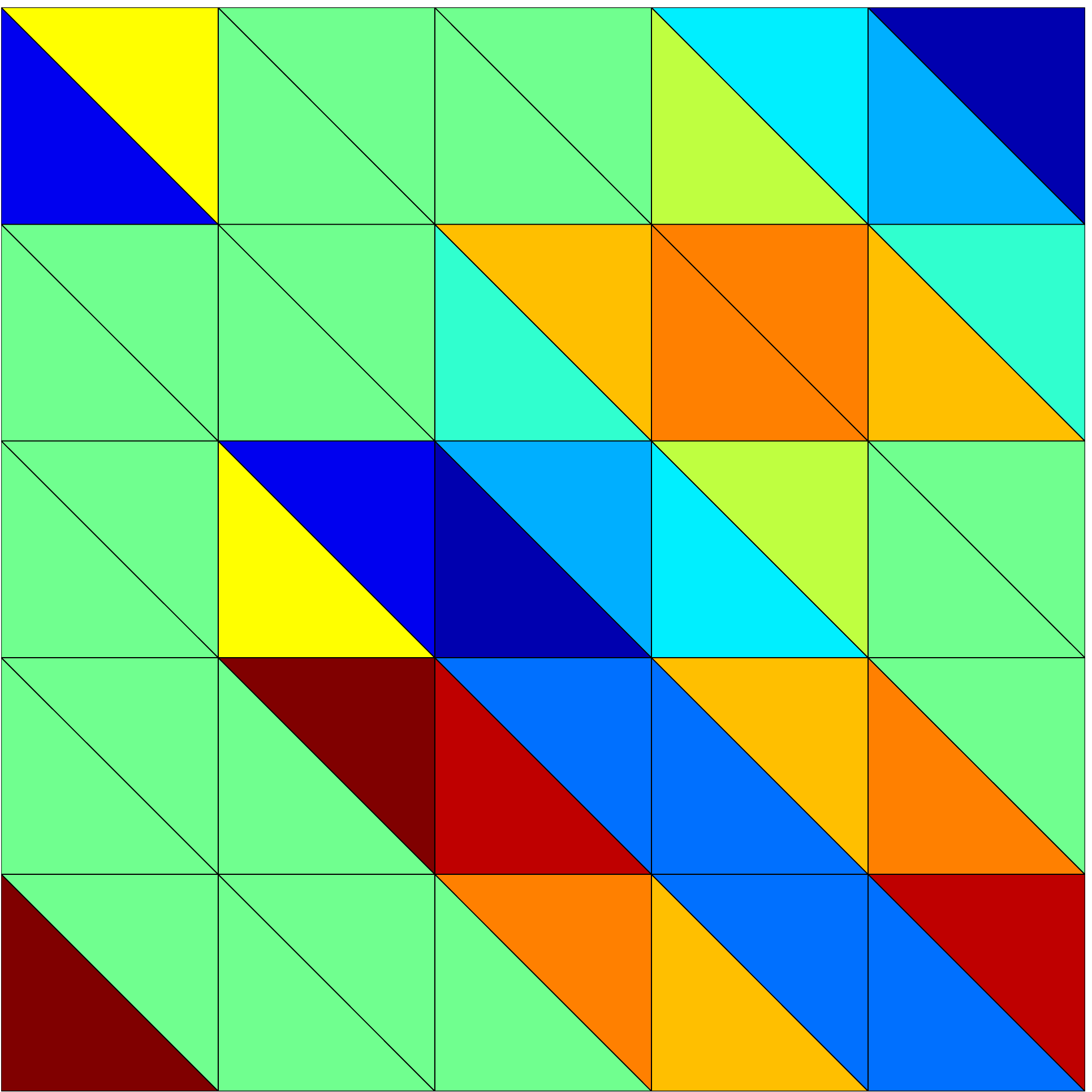}};
    \end{scope}
 
\begin{scope}[shift={(0,0)}]
\def\functionValues{{
      { 0,     2,     2,     2,     2,     0},
    { 2,     2,     2,     3,     1,     2},
    { 2,     2,     4,     2,     2,     2},
    { 2 ,    2,     2,     1,     2,     2},
    { 2,     2,     2,     2,     3,     2},
    { 0,     2,     2,     2,     2,     0}}}

\foreach \i in {0,...,5} {%
	\foreach \j in {0,...,5} {%
	\node[draw, fill = white, scale = 0.65] at (\j,\i) {\pgfmathparse{\functionValues[\i][\j]}\pgfmathresult};
	}
	}
	
	\end{scope}
\end{tikzpicture}
\end{center}
\caption{A minimal valid, continuous, piecewise linear function over the
  polyhedral complex $\P_5$, which is diagonally constrained (\autoref{subsec:standard-tri-R2}).
  \emph{Left}, the three-dimensional plot of the function on $D = [0,1]^2$. \emph{Right},
  the complex $\P_5$, restricted to $D$ and colored according to slopes to match the 3-dimensional plot, and decorated
  with values $v$ at each vertex of $\P_5$ where the function takes value
  $\tfrac{v}{4}$. }

\label{figure:diagonallyConstrained-new-figure}
\end{figure}
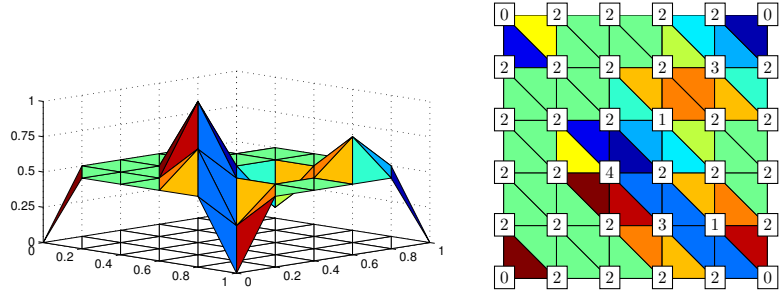

\begin{example}\label{ex:diag-constrained-function} \label{ex:diag-constrained-function-continued2}
  \autoref{figure:diagonallyConstrained-new-figure} shows the complex~$\P_5$
  with an example of a minimal valid continuous piecewise linear function
  on~$\P_5$ with $\f = \ColVec[5]{2}{2}$ that is periodic modulo~$\Z^2$. Note that, due the
  periodicity of the function modulo~$\Z^2$, the values of the function on the
  left and the right edge (and likewise on the bottom and the top edge) of $D = [0,1]^2$ match. 
  
  It can be checked that no relations appearing in the list of all maximal additive faces involve a vertical or
  horizontal edge; thus, the function is diagonally constrained. See
  \cite[sections~\ref{equi3:s:geometry-Pq}
  and~\ref{equi3:subsection:maximal-faces}]{bhk-IPCOext}. \end{example} 


\begin{theorem}[{\cite[\autoref{equi3:thm:main}]{bhk-IPCOext}}]
\label{thm:main}
Consider the following problem.  
\begin{quote}
  Given a minimal 
  valid function $\pi$ for $R_{\f}(\R^2,\Z^2)$ that is 
  piecewise linear continuous on~$\P_q$ and diagonally constrained with $\f \in \verts(\P_q)$, 
  decide if $\pi$ is extreme.
\end{quote}
There exists an algorithm for this problem whose running time is
bounded by a polynomial in $q$.
\end{theorem}

As before, this algorithm is only a pseudo-polynomial time algorithm.

\begin{openquestion}
For any fixed $k$, does there exist a polynomial time algorithm to determine extremality of piecewise linear functions for $R_\f(\R^k,\Z^k)$?
\end{openquestion}
Unlike in the one-row problem, even with all rational input, no algorithm is known for determining extremality of piecewise linear functions for $R_\f(\R^k, \Z^k)$ for $k\geq 3$ and, as mentioned in Theorem~\ref{thm:main}, only for certain cases is an algorithm known for $k=2$.   




\section{Algorithm using restriction to finite group problems}
\label{sec:alg-restriction-finite-groups}
In this section, we discuss connections between infinite group problems and finite group problems.  We begin with a discussion of testing extremality for finite group problems.  Later we show that in certain settings, a function is extreme for an infinite group problem if and only if its restriction to a finite group is extreme for the finite group problem.  Hence, this connection provides an alternative algorithm from those described in section~\ref{sec:one-two-dim} for testing extremality and facetness.
\subsection{Algorithm for finite group problem}\label{s:finite-group}
When $S$ has finite index in $G$, we call $R_\f(G,S)$ a \emph{finite group problem}.
As we noted in Remark~\ref{rem:finite-group}, $R_\f(G,S)$ and $R_\f(G/S,0)$
are closely related by aggregation of variables, and it is convenient to study
the finite-dimensional problem $R_{\bar\f}(G/S,0)$.  
The fundamental theorem of finitely generated abelian groups shows that $G/S
\cong (\tfrac{1}{q_1} \Z \times \dots \times \tfrac{1}{q_k} \Z)/\Z^k$ for some
$q_i \in \N$ for $i=1, \dots, k$.  Therefore, it suffices to consider
$G = \tfrac{1}{q_1} \Z \times \dots \times \tfrac{1}{q_k} \Z$ and $S = \Z^k$
where $q_i \in \N$.  
In the case of one row, $G/S = \tfrac{1}{q_1} \Z/\Z \cong \Z/q_1\Z$ is a cyclic
group.   Cyclic group problems were originally studied by
Gomory~\cite{gom} and have been the subject of many later studies.
See~\cite{Richard-Dey-2010:50-year-survey} for an excellent survey on these results.



  
 The set of minimal valid functions $\pi \colon G/S \to \R$ 
 is a (finite-dimensional) convex polytope~\cite{gom}.  Extreme functions are
 thus extreme points of this polytope.  As we noted in
 \autoref{subsec:relation-between-three-notions}, 
 standard polyhedral theory reveals that extreme functions are
 equivalent to weak facets and facets.  
 Furthermore, extreme points of
 polytopes are characterized by points where the tight inequalities are of
 full rank.  Therefore, testing extremality of a function for a finite group
 problem can be done with simple linear algebra.  

 Note that there is a bijection
 between the minimal valid functions of $R_\f(G,S)$ and minimal valid functions
 for $R_{\bar \f}(G/S,0)$.  This is because 
 minimal valid functions for $R_\f(G,S)$ are $S$-periodic functions by
 Theorem~\ref{thm:minimal}.  
 Hence the extremality test 
 translates into the following statement about $\bar \Pi^{E(\pi)}(G,\Z^k)$. 
\begin{theorem}
\label{thm:finite-extremality}
 Let $G = \tfrac{1}{q_1} \Z \times \dots \times \tfrac{1}{q_k} \Z$ and let $\f \in G$.  Let $\pi \colon G \to \R$ be a minimal valid function for $R_\f(G, \Z^k)$.  Then $\pi$ is extreme if and only if $\bar \Pi^E(G,\Z^k) = \{0\}$ where $E = E(\pi)$.
\end{theorem}

For any discrete group $G\supseteq \Z^k$ and subgroup $G'$, the set
$R_\f(G'/\Z^k,0)$ is a face of the polyhedron $R_\f(G/\Z^k,0)$.  This
observation implies the following theorem via the above bijection.
\begin{theorem}
Let  $G = \tfrac{1}{q_1} \Z \times \dots \times \tfrac{1}{q_k} \Z$, let $G'$ be any subgroup of $G$, and let $\f \in G'$.  
Let $\pi \colon G \to \R$.  
\begin{enumerate}
\item If $\pi$ is minimal for $R_\f(G,\Z^k)$, then $\pi$ is minimal for $R_\f(G', \Z^k)$.
\item If $\pi$ is extreme for $R_\f(G,\Z^k)$, then $\pi$ is extreme for $R_\f(G',\Z^k)$.
\end{enumerate}
\end{theorem}

\subsection{Restriction and interpolation in the one-row problem}\label{s:interpolation}
Gomory and Johnson devised the infinite group problem as a way to study the finite group problem.  
They studied interpolations of valid functions of the finite group problems $R_f(\tfrac{1}{q} \Z, \Z)$ in order to connect the problems, but they never completed this program.  Due to the ease of testing extremality in the finite group problems, having this connection is useful for algorithms.
We encapsulate their results on this connection in the following theorem.

\begin{theorem}[\cite{infinite}]
\label{thm:GJ-restrictions}
Let $\pi$ be a continuous piecewise linear function with breakpoints in
$\tfrac{1}{q} \Z$ for some $q \in \Z_+$ and let $f \in \tfrac{1}{q}
\Z$.\footnote{Under these hypotheses, $\pi$ is the continuous interpolation of $\pi|_{\frac{1}{q}\Z}$.}
  Then the following hold:
\begin{enumerate}
\item $\pi$ is minimal for $R_f(\R,\Z)$ if and only if $\pi_{\frac{1}{q}\Z}$ is minimal for $R_f(\tfrac{1}{q} \Z, \Z)$.
\item If $\pi$ is extreme for $R_f(\R,\Z)$, then $\pi|_{\frac{1}{q} \Z}$ is extreme for $R_f(\tfrac{1}{q} \Z,\Z)$.
\end{enumerate}
\end{theorem}

Part (1) shows that minimality can be tested on just points in $\tfrac{1}{q} \Z$, while part (2) yields a method of proving a function is not extreme.  That is, if $\pi|_{\frac{1}{q} \Z}$ is not extreme for $R_f(\tfrac{1}{q} \Z,\Z)$, then $\pi$ is not extreme for $R_f(\R,\Z)$.  However, it is not true in general that if $\pi|_{\frac{1}{q} \Z}$ is extreme
for $R_f(\frac{1}{q} \Z, \Z)$, then $\pi$ is extreme for $R_f(\R,\Z)$.  See Figure~\ref{fig:drlm_not_extreme_1} for an example.   To obtain such a characterization, it turns out that we must restrict to a finer grid.  
The first result in this direction of relating the infinite and the finite group problems appeared in \cite{dey1}; we state it in our notation.  
\begin{theorem}[{\cite[Theorem 6]{dey1}}] \label{thm:infinite_test}
Let $\pi$ be a
piecewise linear minimal valid function for $R_f(\R,\Z)$ with set $\B$ of rational
breakpoints with the least common denominator~$q$. 
Then
$\pi$ is extreme if and only if the restriction $\pi|_{ \frac{1}{2^nq} \Z}$ is extreme
for $R_f( \frac{1}{2^nq} \Z, \Z)$ for all $n\in \N$.
\end{theorem} 
The above condition cannot be checked in a finite number of steps and hence
cannot be converted into a computational algorithm, because it potentially needs to test
infinitely many finite group problems.  In fact, this result holds even when just considering $n=2$.

\begin{theorem}[{\cite[Theorem 1.5]{basu-hildebrand-koeppe:equivariant}}]
\label{thm:1/4q-restriction}
If the function~$\pi$ is continuous, then $\pi$ is extreme for
  $R_f(\R,\Z)$ if and only if the restriction $\pi \big|_{\frac1{4q}\Z}$ is
  extreme for the finite group problem $R_f(\frac1{4q}\Z,\Z)$.
\end{theorem}

This result demonstrates a tight connection between finite and infinite group problems, and in particular, yields an alternative algorithm to Theorem~\ref{thm:one-dim} for testing extremality.  That is, to test extremality of $\pi$, simply test if $\pi|_{\frac{1}{4q}
\Z}$ is extreme for $R_f(\tfrac{1}{4q}\Z,\Z)$ using linear algebra, as discussed in subsection~\ref{s:finite-group}.   
To prove Theorem~\ref{thm:1/4q-restriction}, the authors construct certain perturbations functions that are piecewise linear with breakpoints in $1/4q$.  In fact, this result can be improved by a different choice of perturbation function, to have the piecewise linear function have breakpoints in $1/3q$, or $1/mq$ for any fixed $m \in \Z_{\geq 3}$.  This observation yields the following result for which we provide a proof.

\begin{theorem}[\TheoremNEWRESULT]\label{thm:extreme-facet-and-restriction}
\label{thm:extreme-restriction-m}
  Let $m \in \Z_{\geq 3}$.  Let $\pi$ be a continuous piecewise linear  minimal
  valid function for $R_f(\R,\Z)$ with breakpoints in $\tfrac{1}{q} \Z$ and suppose $f \in \tfrac{1}{q} \Z$.    The following are equivalent:
\begin{enumerate}
\item $\pi$ is a facet for $R_f(\R,\Z)$,
\item $\pi$ is extreme for $R_f(\R,\Z)$,
\item $\pi|_{\frac{1}{mq}\Z}$ is extreme for $R_f(\tfrac{1}{mq} \Z, \Z)$.
\end{enumerate}
\end{theorem}

\begin{proof}
As mentioned in subsection~\ref{subsec:relation-between-three-notions}, facets are extreme functions~\cite[Lemma 1.3]{bhkm}, and hence $1 \Rightarrow 2$.
By Theorem~\ref{thm:GJ-restrictions},  $2 \Rightarrow 3$.  
We now show $3 \Rightarrow 1$. 

Set $E = E(\pi)$.  Let $\pi|_{\frac{1}{mq} \Z}$ be extreme for $R_f(\tfrac{1}{mq} \Z, \Z)$
and suppose, for the sake of deriving a contradiction, that $\pi$ is not a facet for $R_f(\R,\Z)$.  
Then, by the Facet Theorem (Theorem~\ref{thm:facet}), $\bar \Pi^E(\R, \Z)$ 
  contains a nontrivial element (see subsection~\ref{subsec:perturbations}).  
 Since $\pi|_{\frac{1}{mq} \Z}$ is extreme for $R_f(\tfrac{1}{mq} \Z,\Z)$, by Theorem~\ref{thm:finite-extremality}, $\bar \Pi^{E'}(\tfrac{1}{mq} \Z, \Z) = \{0\}$ for $E' = E \cap \tfrac{1}{mq} \Z^2$.  By Lemma~\ref{lem:tri-restriction} part~\ref{lem:tri-restriction-part1} with $\T = \P_{\frac{1}{mq} \Z}$, we have that $\bar \Pi_\T(\R,\Z) = \{0\}$.  Therefore $\bar \Pi^E(\R, \Z) = \bPizeroOne$.  Furthermore, Lemma~\ref{lem:tri-restriction} part~\ref{lem:tri-restriction-part2} shows that 
\begin{equation}\label{eq:barpi-is-zero-on-mq}
 \bar \Pi^E(\R, \Z) = \bar \Pi^E(\R,\Z) \cap  \{ \bar \pi \st \bar \pi|_{\frac{1}{mq}\Z} \equiv 0\}.
\end{equation}
We divide $E(\pi)$ by the faces of $\Delta \P_{\frac{1}{q} \Z}$ using Lemma~\ref{lemma:covered-by-maximal-valid-triples}.  For $i=1,2,3$, define 
$$
E_i := \bigcup \{ F \in E(\pi, \P_{\frac{1}{q} \Z}) \st \dim(F) = i\,\}.
$$
So $E = E_0 \cup E_1 \cup E_2$.
\smallbreak
\noindent {\emph{Step 1.  Remove $E_0$:}}   
We claim that $\bar \Pi^E(\R, \Z) = \bar \Pi^{E_1 \cup E_2} (\R, \Z) \cap  \{ \bar \pi \st \bar \pi|_{\frac{1}{q} \Z} \equiv 0 \}$.  

First, for any $\bar \pi \in  \bar \Pi^{E_1 \cup E_2} (\R, \Z) \cap \{ \bar \pi \st \bar \pi|_{\frac{1}{q} \Z} \equiv 0 \}$, we have that $\tfrac{1}{q} \Z^2 \subseteq E(\bar \pi)$.  Furthermore, since $\verts(\Delta \P_{\frac{1}{q}}) = \tfrac{1}{q} \Z^2$, we have that $E_0 \subseteq \tfrac{1}{q} \Z^2$.  Therefore, $E_0 \subseteq E(\bar \pi)$.  Hence $\bar \pi \in \bar \Pi^{E_0 \cup E_1 \cup E_2}(\R,\Z) = \bar \Pi^E(\R,\Z)$.

On the other hand, for any $\bar \pi \in \bar \Pi^E(\R,\Z)$, trivially $\bar \pi \in \bar\Pi^{E_1\cup E_2}(\R,\Z)$.  From~\eqref{eq:barpi-is-zero-on-mq}, we see that $\bar \pi \in \{ \bar \pi \st \bar \pi|_{\frac{1}{q} \Z} \equiv 0 \}$.
 \smallbreak
\noindent {\emph{Step 2. Remove $E_2$:}}  
Define $X := \bigcup \{p_i(E_2) \colon i=1,2,3\}$.  The set $X$ is called the ``covered intervals" in~\cite{basu-hildebrand-koeppe:equivariant}.
We claim that $ \bar \Pi^{E_1 \cup E_2} (\R, \Z) \cap  \{ \bar \pi \st \bar \pi|_{\frac{1}{q} \Z} \equiv 0 \} = \bar \Pi^{E_1}(\R,\Z) \cap  \{ \bar \pi \st \bar \pi|_{\frac{1}{q} \Z \cup X} \equiv 0 \}$. 

For any $\bar \pi \in \bar \Pi^{E_1}(\R,\Z) \cap  \{ \bar \pi \st \bar \pi|_{\frac{1}{q} \Z \cup X} \equiv 0 \}$, we see that $E_2 \subseteq E(\bar \pi)$ since $\bar \pi|_X \equiv 0$.  Therefore $\bar \pi \in  \bar \Pi^{E_1 \cup E_2} (\R, \Z) \cap  \{ \bar \pi \st \bar \pi|_{\frac{1}{q} \Z} \equiv 0 \}$.

On the other hand, let $\bar \pi \in \bar \Pi^{E_1 \cup E_2} (\R, \Z) \cap  \{ \bar \pi \st \bar \pi|_{\frac{1}{q} \Z} \equiv 0 \}$.  By Step 1 and~\eqref{eq:barpi-is-zero-on-mq}, $\bar \pi|_{\frac{1}{mq} \Z} \equiv 0$.  
For any $F \in E(\pi, \P_{\frac{1}{q}\Z})$ with $\dim(F) = 2$, by \autoref{lem:projection_interval_lemma} the function $\bar \pi$ is affine on the projections $\intr(p_i(F))$ for $i=1,2,3$.
The projections $p_i(F)$ are full intervals in the complex $\P_{\frac{1}{q} \Z}$ (see Figure~\ref{fig:uniform-spacing}).  In particular, their endpoints lie in $\tfrac{1}{q}\Z$.  Thus, $\intr(p_i(F)) \cap \tfrac{1}{mq}\Z$ contains at least two points since $m\geq 3$.  
Since $\bar \pi|_{\frac{1}{mq}\Z} \equiv 0$ and $\bar \pi$ is affine on $\intr(p_i(F))$, it follows that $\bar \pi|_{\intr(p_i(F))} \equiv 0$.  
Furthermore, since the endpoints of $p_i(F)$ are in $\tfrac{1}{q} \Z$, we also have that $\bar \pi|_{p_i(F)} \equiv 0$.  Finally, since $E_2$ is the union of all $F \in E(\pi, \P_{\frac{1}{q}\Z})$ with $\dim(F) = 2$, it follows that $\bar \pi|_X \equiv 0$, and hence $\bar \pi \in \bar \Pi^{E_1}(\R,\Z) \cap  \{ \bar \pi \st \bar \pi|_{\frac{1}{q} \Z \cup X} \equiv 0 \}$.
\smallbreak
\noindent {\emph{Step 3. Write down $E_1$ relations:}} The additivity set $E_1$ corresponds to one-dimensional faces in $\Delta\P_{\frac{1}{q} \Z}$.  These faces represent one the following two relations: 
\begin{equation*}
\arraycolsep=1.4pt
\begin{array}{rccclrr}
\pi(x) &+& \pi(t) &=& \pi(x + t) && \text{ for all } x \in I,\\
 \pi(x) &+& \pi(r-x) &=& \pi(r) && \text{ for all } x \in I,
\end{array}
\end{equation*} 
for some $I \in \I_q$ and $r,t \in \tfrac{1}{q} \Z$.   Since $\bar \pi|_{\frac{1}{mq}} \equiv 0$, we have $\bar \pi(t) = 0$ and $\bar \pi(r) = 0$.  Considering this, we can find   sets $R_I,T_I \subseteq \tfrac{1}{q} \Z$ for every interval $I \in \P_{\frac{1}{q} \Z, \EquiOneDimEdge}$ (see Example~\ref{ex:1d-breakpoint-complex} for notation $\P_{\frac{1}{q} \Z, \EquiOneDimEdge}$) such that 
  \begin{equation}
  \label{zero}
\bar \Pi^E(\R, \Z)  = \left\{ \bar \pi\colon \R \to \R \, \Bigg| \,
\begin{array}{lr}
 \bpi(x) = 0 & \text{ for all } x \in X \cup \frac{1}{q} \Z\\
\bpi(x) = \bpi(x+t) &   \text{ for all } x \in I, t \in T_I, I \in \P_{\frac{1}{q} \Z, \EquiOneDimEdge}\\
\bpi(x) = -\bpi(r-x) &   \text{ for all } x \in I, r \in R_I, I \in \P_{\frac{1}{q} \Z, \EquiOneDimEdge}
 \end{array} \right\}.
 \end{equation}
Note that taking $T_I \supseteq \Z$ for all $I \in \P_{\frac{1}{q} \Z, \EquiOneDimEdge}$ covers the periodicity conditions.
\smallbreak

\noindent \emph{Step 4. Derive contradiction:}
We define the orbit $\mathcal O(x) = (\{x\} \cup \{-x\}) + \frac{1}{q} \Z $.  Thus, for any interval $I \in \P_{\frac{1}{q} \Z, \EquiOneDimEdge}$ and $x \in I$, we have $x + t, r - x \in \mathcal O(x)$ for all $t \in T_I, r \in R_I$. 
Notice that $\mathcal O([0,\frac{1}{2q}]) :=  \bigcup_{x \in [0, \frac{1}{2q}]} \mathcal O(x) = \R$.  

Let $\bar \pi \in \bar \Pi^E(\R,\Z)$ such that $\bar \pi \not\equiv 0$.  By~\eqref{eq:barpi-is-zero-on-mq}, $\bar \pi|_{\frac{1}{mq} \Z} \equiv 0$.  Since $\bar\pi \not\equiv 0$ and $\mathcal O([0,\frac{1}{2q}])= \R$, there exists an $x_0 \in [0,\tfrac{1}{2q}]\setminus \tfrac{1}{mq} \Z$ such that $\bar \pi|_{\mathcal O(x_0)} \not\equiv 0$.  
Define $\bar \pi_{x_0} \colon \R \to \R$ as 
$$
\bar \pi_{x_0}(x) = \begin{cases}
\bar \pi(x) & \text{ if }  x \in \mathcal O(x_0),\\
0 & \text{ otherwise.}
\end{cases}
$$
The key idea here that we need to use is that in \eqref{zero}  the value of $\bar \pi$ at $x$ is related only to the value at points in $\mathcal O(x)$.  From that, it follows from~\eqref{zero} that $\bar \pi_{x_0} \in \bar \Pi^E(\R,\Z)$.  We next will transform $\bar \pi_{x_0}$.
By definition of $x \in \mathcal O(x_0)$ we have $x = x_0 + t$ for some $t \in \tfrac{1}{q} \Z$ or $x = - x_0 + r$ for some $r \in \tfrac{1}{q}\Z$.  If $x_0 \in \frac{1}{2q} \Z$, both decompositions are possible, but otherwise, only one such decomposition is possible.

We now consider the orbit $\mathcal O(\tfrac{1}{mq}) = \{ \tfrac{1}{mq}, -\frac{1}{mq}\} + \tfrac{1}{q} \Z$ and define $\varphi \colon \mathcal O(\tfrac{1}{mq}) \to \R$ as
$$
\varphi(\tfrac{i}{mq} + t) = \begin{cases}
\bar\pi_{x_0}(x_0 + t) & \text{ if } i = 1,\\
\bar\pi_{x_0}(-x_0 + t) & \text{ if } i = -1,
\end{cases}
$$
for all $t \in \tfrac{1}{q}\Z$.
The description of $\varphi$ transfers values of $\bar \pi$ in $\mathcal O(x_0)$ to values in $\mathcal O(\tfrac{1}{mq})$. 
Since $\bar \pi|_{\mathcal O(x_0)} \not\equiv 0$, we also have that $\varphi \not\equiv 0$.  
Finally, define $\bar \varphi \colon \R \to \R$ as
$$
\bar \varphi(x) = \begin{cases}
\varphi(x) & \text{ if } x \in \mathcal O(\tfrac{1}{mq}),\\
0 & \text{ otherwise}.
\end{cases}
$$
Then, using the representation in~\eqref{zero}, the fact that $\bar \pi_{x_0} \in \bar \Pi^E(\R,\Z)$ implies that $\bar \varphi \in \bar \Pi^E(\R,\Z)$.  But notice that $\bar \varphi|_{\frac{1}{mq} \Z} \not\equiv 0$ since $\varphi \not\equiv 0$ and $\mathcal O(\tfrac{1}{mq}) \subseteq \tfrac{1}{mq}\Z$, which contradicts~\eqref{eq:barpi-is-zero-on-mq}.
Therefore, we conclude that $3 \Rightarrow 1$.
\end{proof}

Figure~\ref{fig:drlm_not_extreme_1} gives an example of a function $\pi$ that
is not extreme for $R_f(\R,\Z)$, but $\pi|_{\frac{1}{q}\Z}$ is extreme for
$R_f(\tfrac{1}{q} \Z,\Z)$.  

Using computer-based search, K\"oppe and Zhou
\cite{koeppe-zhou:extreme-search} found a function that is not extreme for
$R_f(\R,\Z)$, but whose restriction to $\tfrac{1}{2q}\Z$ is extreme for
$R_f(\tfrac{1}{2q}\Z, \Z)$.\footnote{The function is available in the
  electronic compendium \cite{electronic-compendium}  as
  \sagefunc{kzh_2q_example_1}.}
This proves the following result.
\begin{proposition}[K\"oppe and Zhou \cite{koeppe-zhou:extreme-search}]
  The hypothesis $m \geq 3$ in \autoref{thm:extreme-restriction-m} is best
  possible. The theorem does not hold for $m=2$.
\end{proposition}

%
%
%
%
%

%
%

\begin{figure}[tp]
\begin{center}
\includegraphics[width=.31\linewidth]{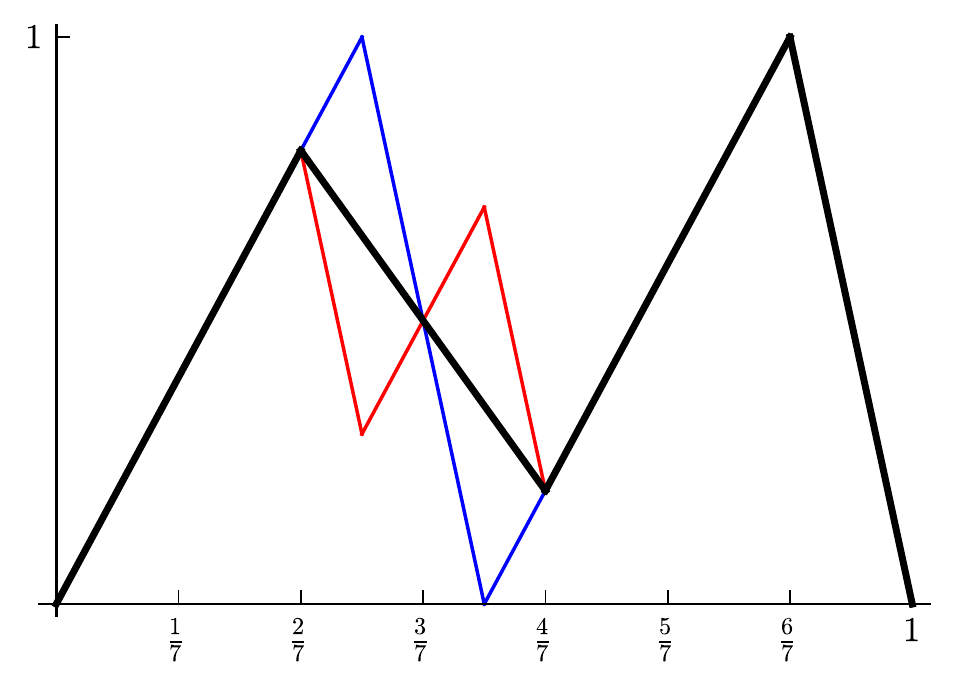}\quad
\includegraphics[width=.31\linewidth]{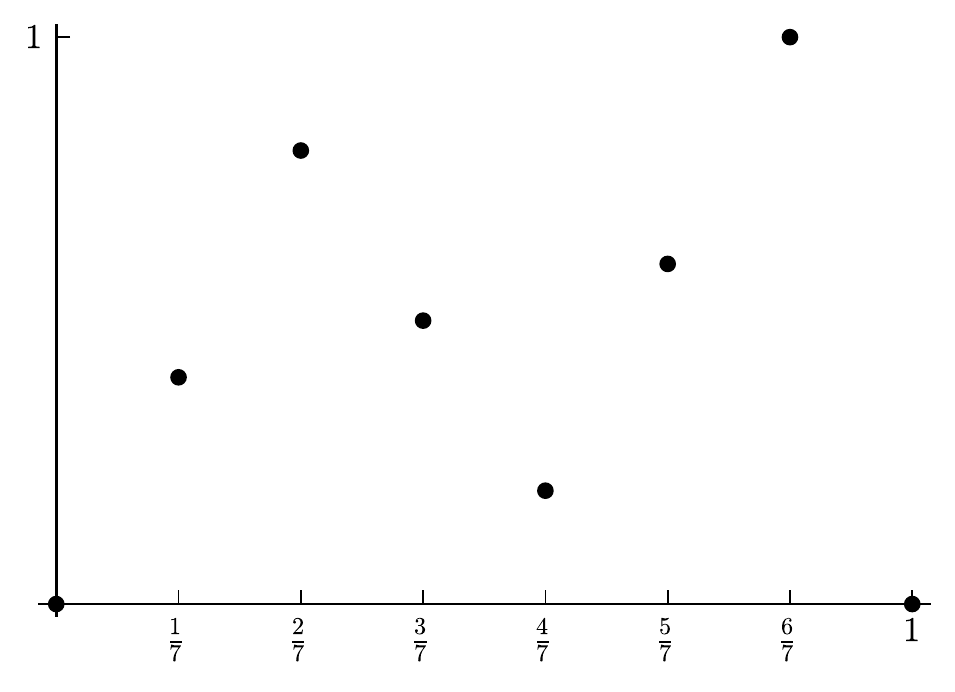}\quad
\includegraphics[width=.31\linewidth]{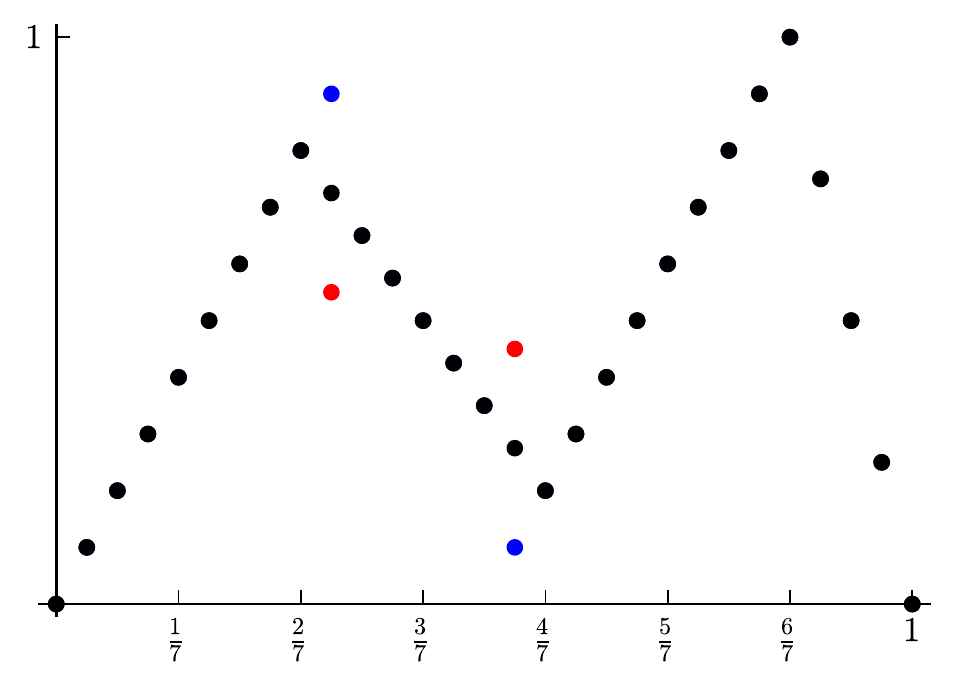}
\end{center}
\caption{This function (\sage{h = \sagefunc{drlm_not_extreme_1}()}) is
  minimal, but not extreme (and  hence also not a facet), as proved by
  \sage{\sagefunc{extremality_test}(h, show\underscore{}plots=True)}
  by demonstrating a perturbation.  The red and blue perturbations describe the minimal functions $\pi^1, \pi^2$ that verify that $\pi$ is not extreme.  These minimal functions necessarily have more breakpoints than $\pi$.  This is because $\pi|_{\frac{1}{q} \Z}$ with $q = 7$, as depicted in the middle figure, is extreme for the finite group problem $R_f(\tfrac{1}{q} \Z,\Z)$.  However, $\pi|_{\frac{1}{2q} \Z}$ is not extreme for $R_f(\tfrac{1}{2q} \Z, \Z)$.  The discrete perturbations, depicted on the right, are interpolated to obtain the continuous functions $\pi^1, \pi^2$.
}
\label{fig:drlm_not_extreme_1}
\end{figure}

\subsection{Restriction and interpolation for $k\geq 2$}

Some similar restriction results can be proved for the case of $k$ rows, but
this area is much more open.  Restrictions seem to require the use of nice
polyhedral complexes.  The only results known are for the polyhedral complex
$\P_q$ (Example~\ref{ex:2d-standard-triangulation}) in $\R^2$. 

\begin{theorem}[{\cite[Theorem~\ref{equi3:minimality-check-2d} and Theorem~\ref{equi3:theorem:systemNotUnique}]{bhk-IPCOext}}]
\label{thm:GJ-restrictions-2D}
Let $\pi\colon \R^2 \to \R$ be a continuous piecewise linear function over $\P_q$ and suppose $\f \in \tfrac{1}{q} \Z^2$.  Then the following hold:
\begin{enumerate}
\item $\pi$ is minimal for $R_\f(\R^2,\Z^2)$ if and only if $\pi_{\frac{1}{q}\Z^2}$ is minimal for $R_\f(\tfrac{1}{q} \Z^2, \Z^2)$.
\item If $\pi$ is extreme for $R_\f(\R^2,\Z^2)$, then $\pi|_{\frac{1}{q} \Z^2}$ is extreme for $R_\f(\tfrac{1}{q} \Z^2,\Z^2)$.
\end{enumerate}
\end{theorem}
For $k\geq3$ rows, it is unclear when similar results are possible.

\begin{openquestion}
Can Theorem~\ref{thm:GJ-restrictions-2D} be generalized to other
triangulations of $\R^k$ for $k\geq2$?
\end{openquestion}


In the special case of diagonally constrained functions in $\R^2$, there is a similar result to Theorem~\ref{thm:extreme-restriction-m}.

\begin{theorem}[{\cite[\autoref{equi3:thm:1/4q}]{bhk-IPCOext}}]
\label{thm:1/4q}
Let $\pi$ be a minimal continuous piecewise linear function over $\P_q$ that is diagonally constrained and $\f \in \verts(\P_q)$.  Fix $m \in \Z_{\geq 3}$.
Then $\pi$ is extreme for $R_{\ve f}(\R^2, \Z^2)$ if and only if the restriction $\pi\big|_{\frac{1}{mq}\Z^2}$ is extreme for $R_{\ve f}(\frac{1}{mq} \Z^2, \Z^2)$.  
\end{theorem}
If we know that $\pi$ is diagonally constrained, then this theorem produces an alternative algorithm to Theorem~\ref{thm:main} to test extremality of $\pi$ by simply restricting to $\tfrac{1}{3q}\Z^2$ and testing extremality in the finite dimensional setting.
A generalization of this theorem that removes the condition of being diagonally constrained will appear in a forthcoming article~\cite{basu-hildebrand-koeppe:equivariant-general-2dim}.

\appendix
\section{Updated compendium of extreme functions}
\label{s:appendix-compendium}

The following tables contain the updated compendium of extreme functions.

\newcommand\CompendiumNEWRESULT{\hfill\par\vspace{1ex}\hfill \emph{Previously unpublished} \NEWRESULT}

\newcommand\CompendiumGraphics[1]{\vspace{0pt}\includegraphics[width=\linewidth]{#1}}
\newcommand\CompendiumGraphicsBig[1]{\multicolumn{2}{p{.25\linewidth}}{\vspace{0pt}\includegraphics[width=\linewidth]{#1}}}
\newcommand\CompendiumGraphicsTikz[1]{\vspace{0pt}\scalebox{.3}{\input{#1}}}

\newcommand\CompendiumSlopes[1]{\vspace{0pt}\centering#1}
\newcommand\CompendiumContinuity[1]{\vspace{0pt}\centering#1}

\newenvironment{CompendiumTabular}{%
    \begin{minipage}{\linewidth}\centering
      \let\footnoterule=\relax
    \begin{tabular}{p{.2\linewidth}*1{p{.17\linewidth}}@{}c*1{p{.05\linewidth}}*1{p{.09\linewidth}}p{.35\linewidth}}
      \toprule
      Function\footnote{A function name shown in typewriter font is the name of the constructor of this function in the
        accompanying Sage program.}
      & Graph & 
      & Slopes & Continuity & Notes \\
      \midrule
      }
      {
      \bottomrule
    \end{tabular}
  \end{minipage}
}
\newenvironment{CompendiumProceduresTabular}{%
    \begin{minipage}{\linewidth}\centering
      \let\footnoterule=\relax
    \begin{tabular}{p{.3\linewidth}*2{p{.17\linewidth}}p{.30\linewidth}}
      \toprule
      & \multicolumn{2}{c}{Graphs} \\
      \cmidrule{2-3}
      \multicolumn{1}{c}{Procedure\footnote{A procedure name shown in typewriter font is the name of the corresponding function in the
        accompanying Sage program.}}
      & \multicolumn{1}{c}{From} & \multicolumn{1}{c}{To}
      & \multicolumn{1}{c}{Notes} \\
      \midrule
      }
      {
      \bottomrule
    \end{tabular}
  \end{minipage}
}

\newcommand\sref[1]{{\bfseries \ParaSign\,\ref{#1}}}  

\renewcommand\multirowsetup{\relax}

\begin{table}[b]  
  \centering
  \caption{An updated compendium of known extreme functions for the
    infinite group problem I. Parametrized classes of continuous
    functions for the 1-dimensional case with up to two slopes.
  }
  \label{tab:compendium-1}
  \begin{CompendiumTabular}
      \vspace{0pt}
      \sagefunc{gmic}
      & \CompendiumGraphics{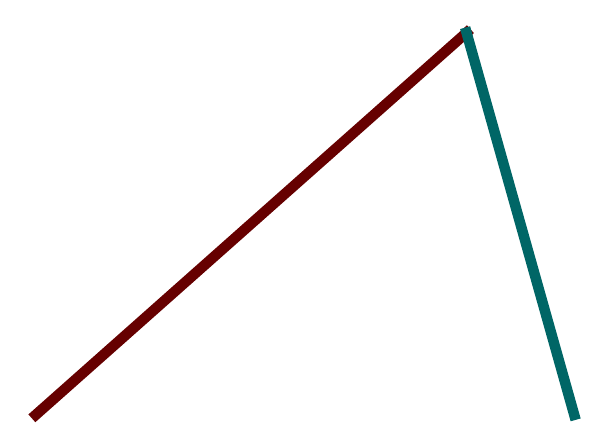} 
      & 
      & \CompendiumSlopes{2}
      & \CompendiumContinuity{C}
      & \vspace{0pt} 
      The famous Gomory mixed integer cut, going back to Gomory's 1960 paper
      \cite{Gom60}. Dominates the \sagefunc{gomory_fractional} cut, which
      is not minimal (\autoref{fig:hierarchy}).
      \\
      \vspace{0pt}
      \sagefunc{gj_2_slope}
      & \CompendiumGraphics{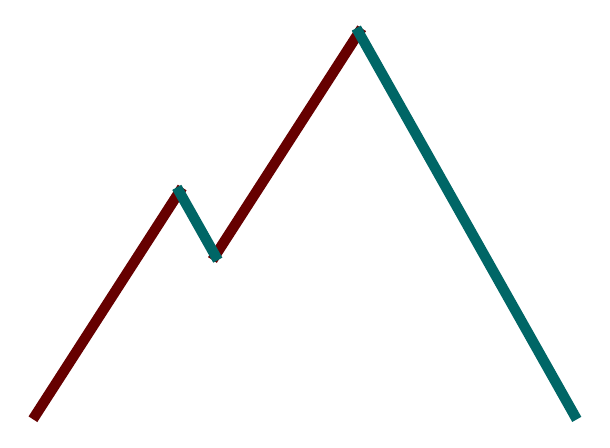}
      & 
      & \CompendiumSlopes{2}
      & \CompendiumContinuity{C}
      & \vspace{0pt}
      \multirow{2}{\linewidth}{%
          Two 
          families of continuous extreme functions with 2~slopes, 
          from 
          Gomory--Johnson 
          \cite{tspace}.
          By the
          Gomory--Johnson 2-Slope Theorem (\autoref{th:2-slope}), all
          continuous piecewise linear minimal valid functions
          with 2 slopes are extreme.
        }
      \\
      \vspace{0pt}
      \sagefunc{gj_2_slope_repeat} 
      & \CompendiumGraphics{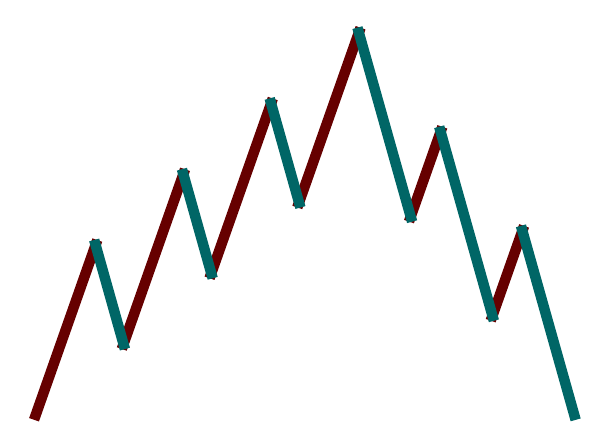}
      & 
      & \CompendiumSlopes{2}
      & \CompendiumContinuity{C}
      & 
      \\
      \vspace{0pt}
      \sagefunc{dg_2_step_mir}
      & \CompendiumGraphics{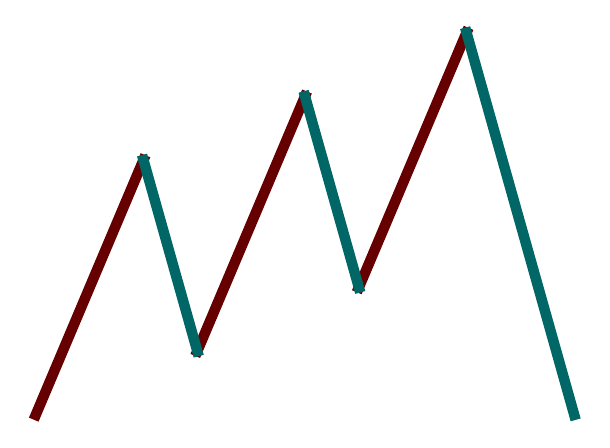}
      &
      & \CompendiumSlopes{2}
      & \CompendiumContinuity{C}
      & \vspace{0pt} Described by Dash--G\"unl\"uk \cite{twoStepMIR}. 
      Extremality follows from the 2-Slope Theorem (\autoref{th:2-slope}).
      \\
      \vspace{0pt}
      \sagefunc{kf_n_step_mir} 
      & \CompendiumGraphics{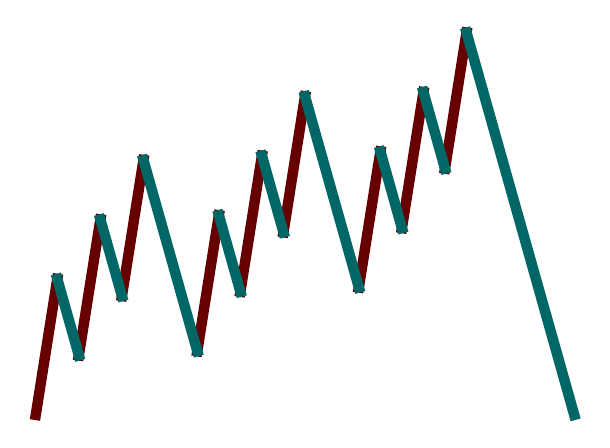}
      & 
      & \CompendiumSlopes{2}
      & \CompendiumContinuity{C}
      & \vspace{0pt} 
      Described by Kianfar--Fathi \cite{kianfar1}. 
      Extremality follows from the 2-Slope Theorem (\autoref{th:2-slope}).
      \\
      \vspace{0pt}
      \sagefunc{bccz_counterexample}
      & \CompendiumGraphics{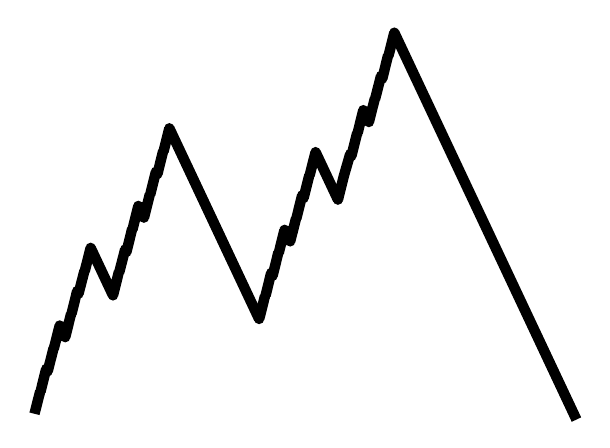}
      & 
      & \CompendiumSlopes{1--2}\footnote{The function is not piecewise linear. 
        In one case ($\mu^- < 1$) \cite{bccz08222222}, it is absolutely
        continuous and thus Lebesgue--almost everywhere differentiable; the derivatives take
        one of two values where they exist.
        In a second case ($\mu^- = 1$) \cite{zhou:extreme-notes}, it is merely continuous (but not
        absolutely continuous) and Lebesgue--almost everywhere differentiable;
        the derivatives take only one value where they exist.  See
        \autoref{sec:bccz_counterexample} for more details.
      }
      & \CompendiumContinuity{C}
      & \vspace{0pt} Limit of 
      \sagefunc{kf_n_step_mir} 
      for $n\to\infty$; not a piecewise linear function. 
      Described by Basu--Conforti-Cornu\'ejols--Zambelli \cite{bccz08222222};
      see~\sref{sec:bccz_counterexample}.
      \\
    \end{CompendiumTabular}
  \end{table}

\begin{table}[tp]
  \centering
  \caption{An updated compendium of known extreme functions for the
    infinite group problem II. Parametrized classes of continuous
    functions for the 1-dimensional case with at least three slopes.
  }
  \label{tab:compendium-1b}
  \begin{CompendiumTabular}
      \vspace{0pt}
      \sagefunc{gj_forward_3_slope} 
      & \CompendiumGraphics{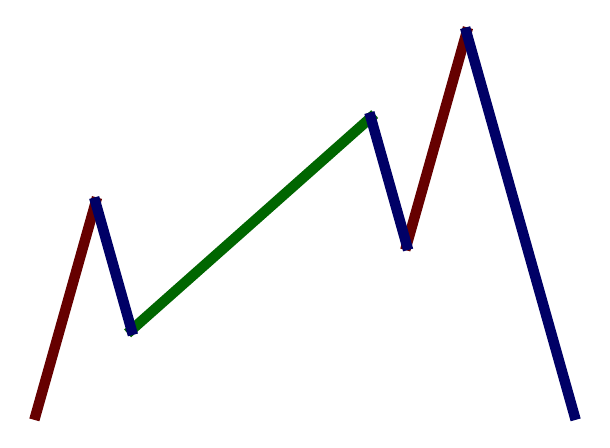}
      & 
      & \CompendiumSlopes{3}
      & \CompendiumContinuity{C}
      & \vspace{0pt} 
      Described by Gomory--Johnson \cite{tspace}.
      \\
      \vspace{0pt}
      \sagefunc{drlm_backward_3_slope} 
      & \CompendiumGraphics{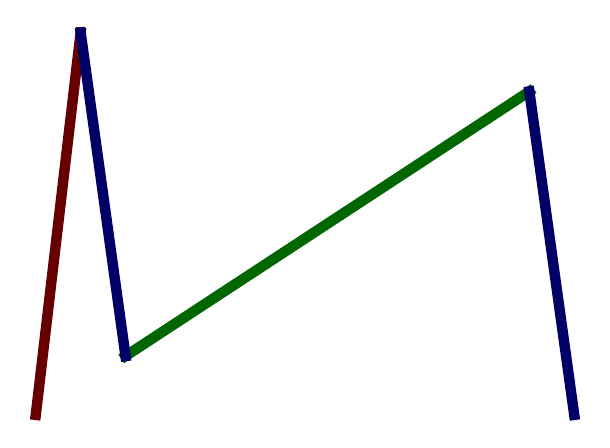}
      & 
      & \CompendiumSlopes{3}
      & \CompendiumContinuity{C}
      & \vspace{0pt} Described by Dey--Richard--Li--Miller \cite{dey1}
      based on Ar\'aoz--Evans--Gomory--Johnson~\cite{AraozEvansGomoryJohnson03}. 
      \\
      \vspace{0pt}
      \sagefunc{dr_projected_sequential_merge_3_slope}
      & \CompendiumGraphics{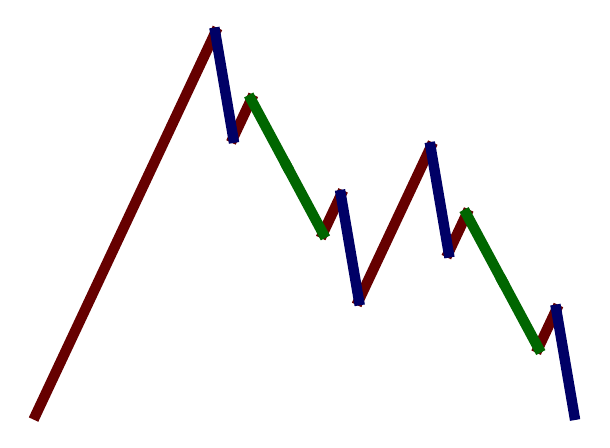}
      & 
      & \CompendiumSlopes{3}
      & \CompendiumContinuity{C}
      & \vspace{0pt} Described by Dey--Richard~\cite{dey2},
      using their \sagefunc{projected_sequential_merge}
      procedure; see \autoref{tab:compendium-procedures} and \sref{s:sequential-merge}.
      \\
      \vspace{0pt}
      \sagefunc{bhk_irrational}
      & \CompendiumGraphics{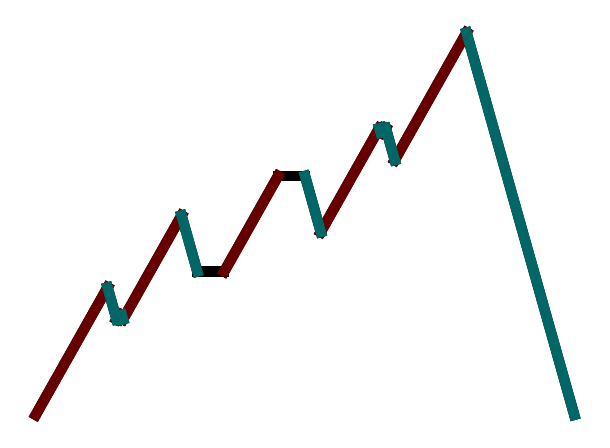}
      & 
      & \CompendiumSlopes{3}
      & \CompendiumContinuity{C}
      & \vspace{0pt} Only extreme when certain parameters are $\Q$-linearly
      independent. Described by Basu--Hildebrand--K\"oppe \cite{basu-hildebrand-koeppe:equivariant}; see
      \sref{s:limits-extremality-not-preserved}. 
      \\
    \vspace{0pt} 
    \sagefunc{chen_4_slope}\footnote{Chen \cite{chen} also
      constructs a family of 3-slope functions, which he claims to be
      extreme. However, his proof for this class is flawed, and none of the
      functions in the described family appear to be extreme, as pointed out in
      \cite{zhou:extreme-notes}. The functions are available as
      \sagefunc{chen_3_slope_not_extreme}.}  
    & \CompendiumGraphics{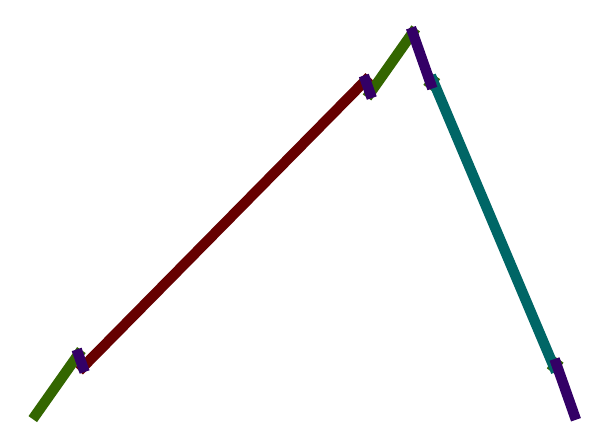}
    & 
    & \CompendiumSlopes{4}
    & \CompendiumContinuity{C}
    & \vspace{0pt} Described by Chen~\cite{chen}.
    \\
    \end{CompendiumTabular}
  \end{table}

\begin{table}[tp]
  \centering
  \caption{An updated compendium of known extreme functions for the
    infinite group problem III. Parametrized families of
    discontinuous functions for the 1-dimensional case.  
  }
  \label{tab:compendium-2}
  \begin{CompendiumTabular}
      \vspace{0pt}
      \sagefunc{ll_strong_fractional}\footnote{In the survey
        \cite[Table 19.4]{Richard-Dey-2010:50-year-survey}, this is called ``Improved GFC.''}
      & \CompendiumGraphics{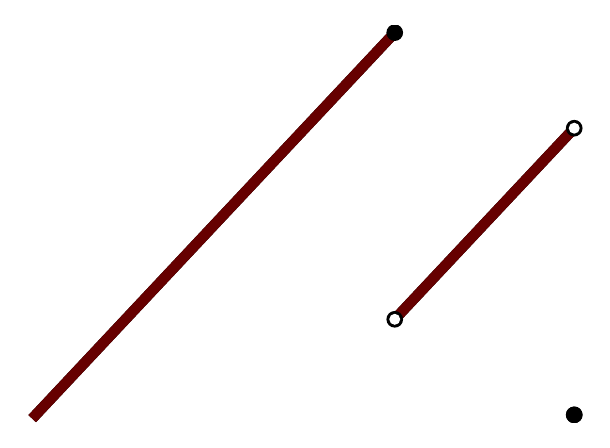}
      & 
      & \CompendiumSlopes{1}
      & \CompendiumContinuity{D}
      & \vspace{0pt} 
      Described by Letchford--Lodi \cite{Letchford-Lodi-2002}\footnote{Note that
        there is a mistake in \cite[Figure 3]{Letchford-Lodi-2002}. The correct
        figure appears here.}; dominates the \sagefunc{gomory_fractional} cut
      (\autoref{fig:hierarchy}).   
      Extreme only if $f \geq \frac12$; then special case of
      \sagefunc{dg_2_step_mir_limit}, \sagefunc{drlm_2_slope_limit} (below).
      \\
      \vspace{0pt}
      \sagefunc{dg_2_step_mir_limit} 
      & \CompendiumGraphics{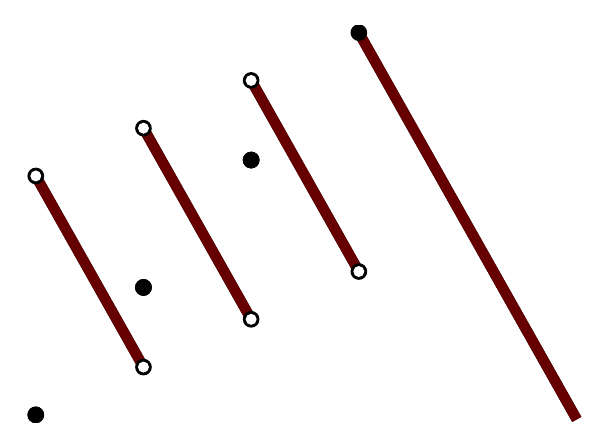}
      & 
      & \CompendiumSlopes{1}
      & \CompendiumContinuity{D}
      & \vspace{0pt} 
      Described by Dash--G\"unl\"uk~\cite{twoStepMIR} (``extended 2-step MIR''). 
      Special case of \sagefunc{drlm_2_slope_limit} (below).
      Defined as a limit of \sagefunc{dg_2_step_mir} functions;  
      see \sref{s:limits} for a discussion of limits.
      \\
      \vspace{0pt}
      \sagefunc{drlm_2_slope_limit} 
      & \CompendiumGraphics{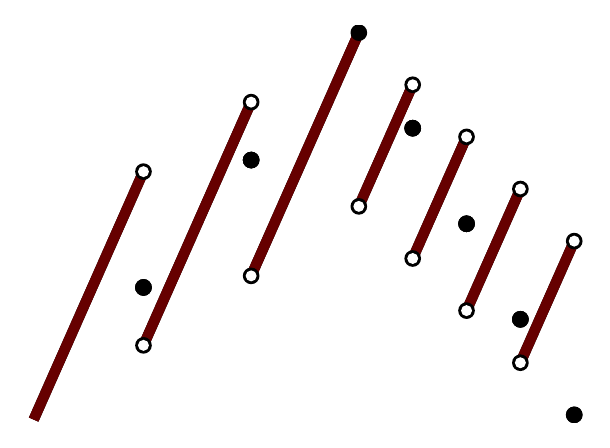}
      & 
      & \CompendiumSlopes{1}
      & \CompendiumContinuity{D}
      & \vspace{0pt}
      From Dey--Richard--Li--Miller \cite{dey1}, generalizing \sagefunc{dg_2_step_mir_limit}
      (above). 
      Defined as a limit;  
      see \sref{s:limits} for a discussion of limits.
      \\
      \vspace{0pt}
      \sagefunc{drlm_3_slope_limit}
      & \CompendiumGraphics{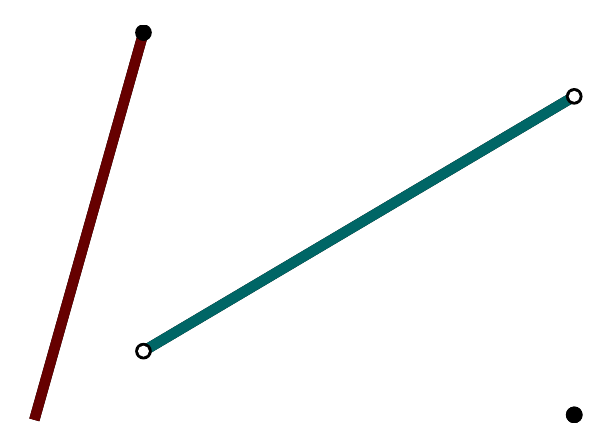}
      & 
      & \CompendiumSlopes{2}
      & \CompendiumContinuity{D}
      & \vspace{0pt} Described by Dey--Richard--Li--Miller \cite{dey1}.  Defined as the
      limit of \sagefunc{drlm_backward_3_slope} functions;  
      see \sref{s:limits} for a discussion of limits.
      \\
      \vspace{0pt}
      \sage{\sagefunc{rlm_dpl1_extreme_3a}}
      & \CompendiumGraphics{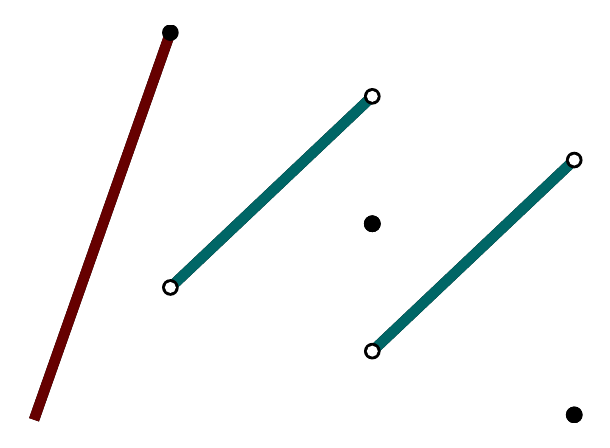}
      & 
      & \CompendiumSlopes{2}
      & \CompendiumContinuity{D}
      & \vspace{0pt}
      A $\mathrm{DPL}_1$-extreme function from Richard--Li--Miller \cite[case 3a]{Richard-Li-Miller-2009:Approximate-Liftings}.  
      Proved extreme in \cite{zhou:extreme-notes}.
      (All other $\mathrm{DPL}_1$-extreme functions from
      \cite{Richard-Li-Miller-2009:Approximate-Liftings} are known to be special cases of
      \sagefunc{drlm_2_slope_limit} and \sagefunc{drlm_3_slope_limit}.)
      \\
    \end{CompendiumTabular}
  \end{table}
  
  \begin{table}[tp]
  \centering
  \caption[An updated compendium of known extreme functions for the
    infinite group problem IV. ``Sporadic'' functions for the
    1-dimensional case.]
    {An updated compendium of known extreme functions for the
      infinite group problem IV. ``Sporadic'' functions for the
      1-dimensional case. 
      {These
      functions were found by computer experiments.  They have not been
      described in the literature as a member of a parametrized family; but
      there is no reason to assume this could not be done.}
  }
  \label{tab:compendium-3}
  \begin{CompendiumTabular}
    \vspace{0pt} 
    \sagefunc{hildebrand_2_sided_discont_1_slope_1}
    & \CompendiumGraphics{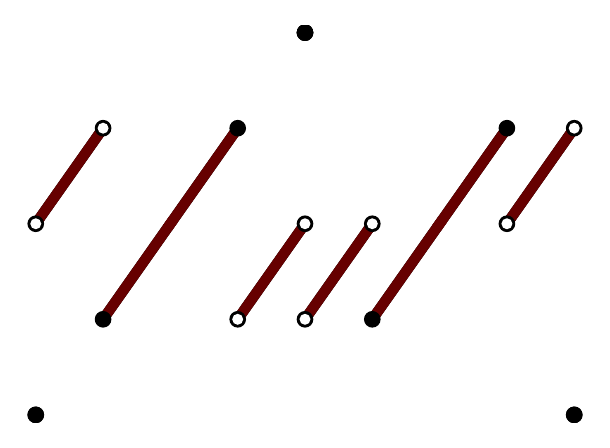}
    & \vspace{0pt}  
    & \CompendiumSlopes{1}
    & \CompendiumContinuity{D}
    & \vspace{0pt} An extreme function that is
    discontinuous on both sides of the origin, from Hildebrand (2013, unpublished).\CompendiumNEWRESULT
    \\
    \vspace{0pt} 
    \sagefunc{hildebrand_2_sided_discont_2_slope_1}
    &  \CompendiumGraphics{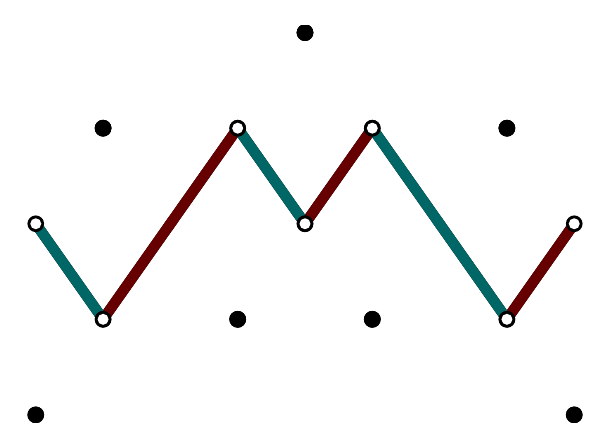}
    & \vspace{0pt}  
    & \CompendiumSlopes{2}
    & \CompendiumContinuity{D}
    & \vspace{0pt} An extreme function that is
    discontinuous on both sides of the origin, from Hildebrand (2013, unpublished).\CompendiumNEWRESULT
    \\
    \vspace{0pt} 
    \sagefunc{hildebrand_discont_3_slope_1}
    & \CompendiumGraphics{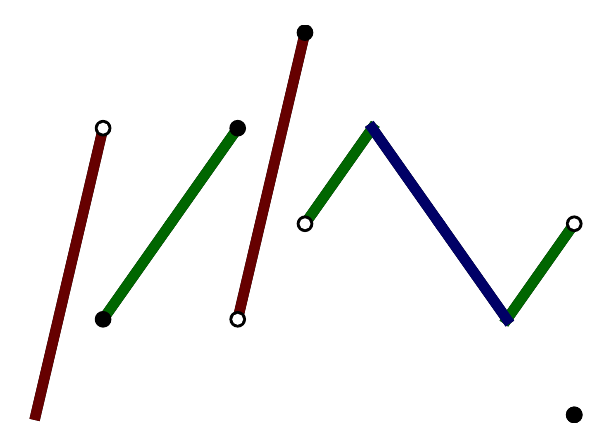}
    & \vspace{0pt}  
    & \CompendiumSlopes{3}
    & \CompendiumContinuity{D}
    & \vspace{0pt} A discontinuous extreme function with 3 slopes,
    from Hildebrand (2013, unpublished).\CompendiumNEWRESULT
    \\
    \vspace{0pt} 
    \sagefunc{hildebrand_5_slope_22_1}\footnote{Several examples
      are known.  Use autocompletion in Sage to obtain a list, by typing
      \sage{hildebrand\underscore{}5\underscore{}slope} and pressing the
      \textsc{tab} key.}
    & \CompendiumGraphics{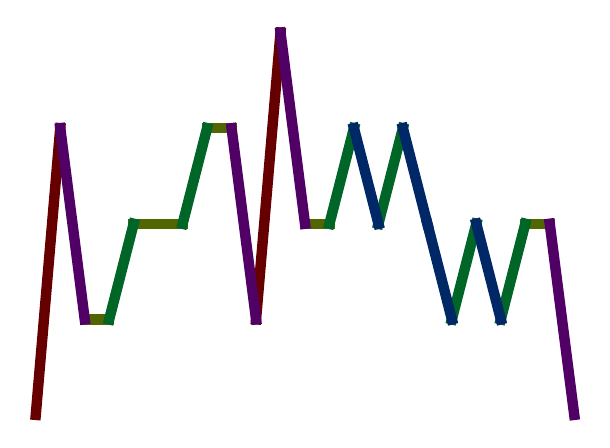}
    & 
    & \CompendiumSlopes{5}
    & \CompendiumContinuity{C}
    & \vspace{0pt} An extreme function with 5~slopes, from 
    Hildebrand (2013, unpublished).  Several
    examples are known.\CompendiumNEWRESULT 
    \\
    \vspace{0pt} 
    \sagefunc{kzh_7_slope_1}\footnote{Several examples
      are known.  Use autocompletion in Sage to obtain a list, by typing
      \sage{kzh\underscore{}} and pressing the
      \textsc{tab} key.}
    & \CompendiumGraphics{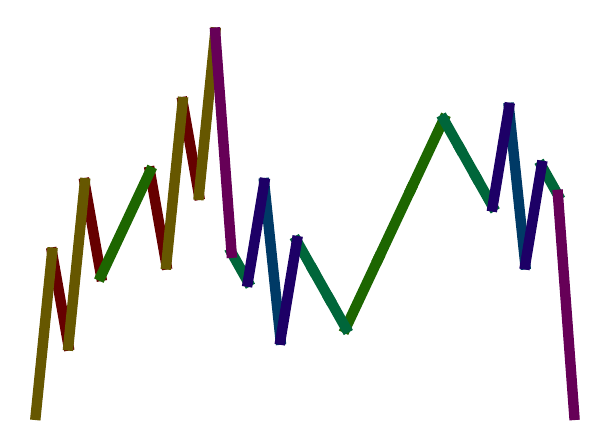}
    & 
    & \CompendiumSlopes{7}
    & \CompendiumContinuity{C}
    & \vspace{0pt} An extreme function with 7~slopes, from
    K\"oppe--Zhou \cite{koeppe-zhou:extreme-search}.  Several 
    examples are known.\CompendiumNEWRESULT 
    \\
    \vspace{0pt} 
    \sagefunc{kzh_28_slope_1}
    & \CompendiumGraphics{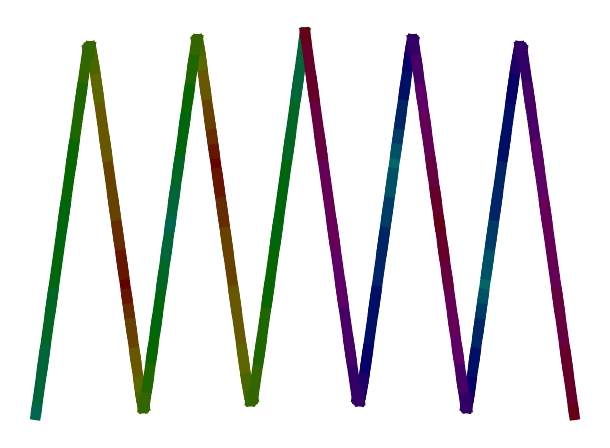}
    & 
    & \CompendiumSlopes{28}
    & \CompendiumContinuity{C}
    & \vspace{0pt} An extreme function with 28~slopes, from K\"oppe--Zhou
    \cite{koeppe-zhou:extreme-search}.
    The shown graph does not convey the complexity of this
    function, which has 395 breakpoints in~$[0,1]$ sampled from $\frac1{778}\Z$. 
    \CompendiumNEWRESULT 
    \\
    \end{CompendiumTabular}
  \label{tab:compendium}
\end{table}


\begin{table}[tp]
  \centering
  \caption{An updated compendium of known extreme functions for the
    infinite group problem V. Procedures.}
  \label{tab:compendium-procedures}

  \begin{CompendiumProceduresTabular}
    \vspace{0pt}
    \sagefunc{automorphism}
    & \CompendiumGraphics{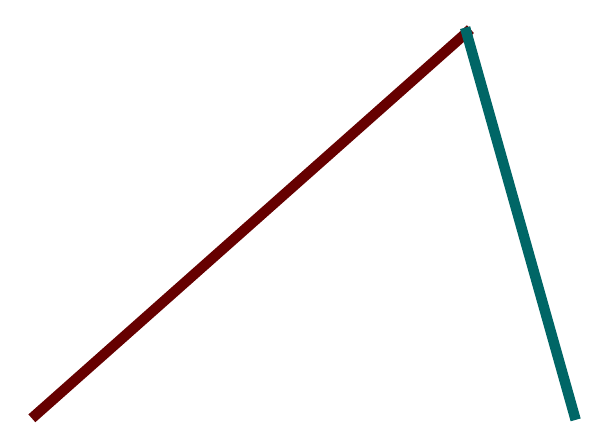}
    & \CompendiumGraphics{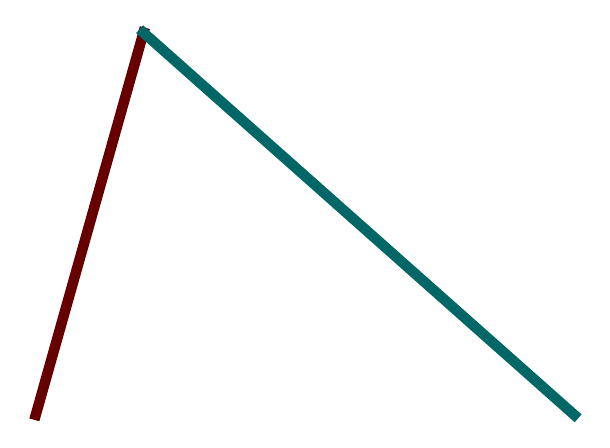}
    & \vspace{0pt}
    From Johnson \cite{johnson}; see \cite[section
    19.5.2.1]{Richard-Dey-2010:50-year-survey}. 
    \\
    \vspace{0pt}
    \sagefunc{multiplicative_homomorphism}
    & \CompendiumGraphics{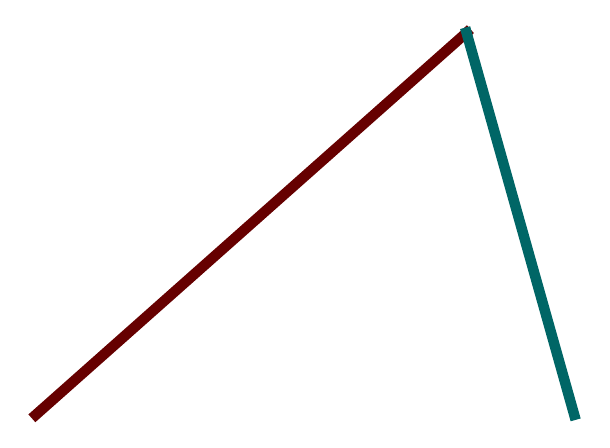}
    & \CompendiumGraphics{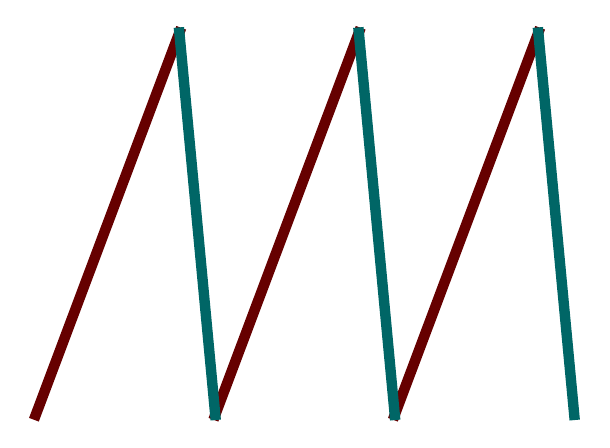}
    & \vspace{0pt}
    See \cite[sections 19.4.1, 19.5.2.1]{Richard-Dey-2010:50-year-survey}.
    \\
    \vspace{0pt}
    \sagefunc{projected_sequential_merge}
    & \CompendiumGraphics{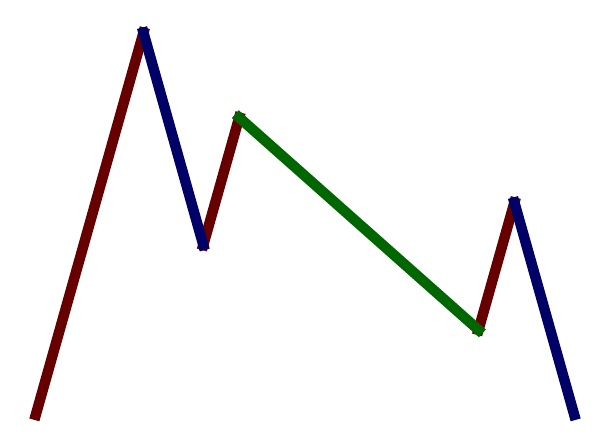}
    & \CompendiumGraphics{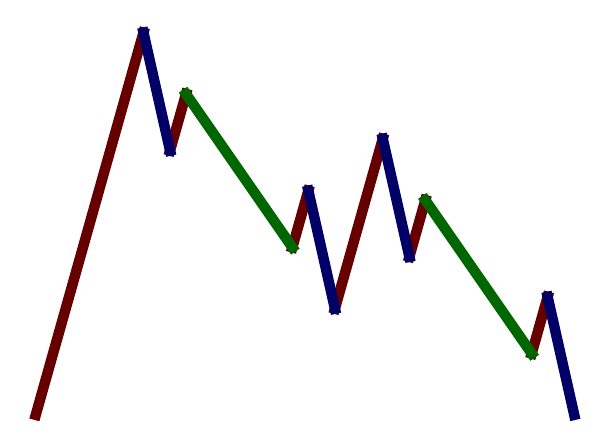} 
    & \vspace{0pt}
    Operation $\mergeProj$ from Dey--Richard \cite{dey2}; see 
    \sref{s:sequential-merge}.
    \\
    \vspace{0pt}
    \sagefunc{restrict_to_finite_group}
    & \CompendiumGraphics{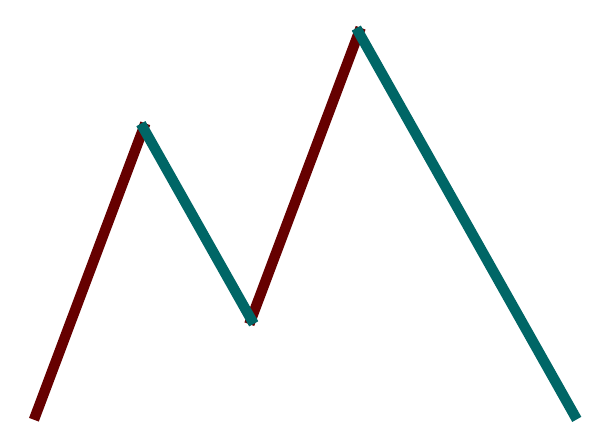}
    & \CompendiumGraphics{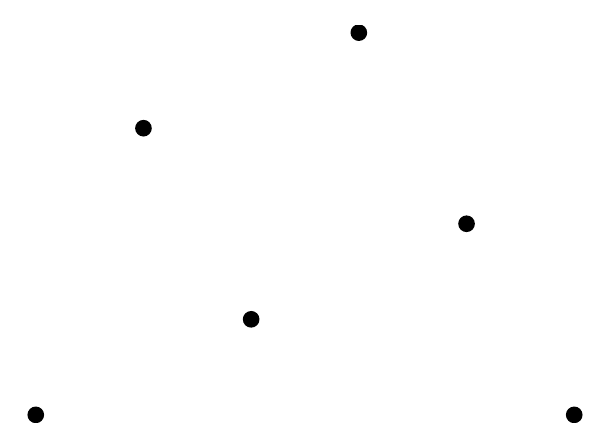}
    & \vspace{0pt}
    Restrictions to finite group problems $R_f(\tfrac{1}{q} \Z, \Z)$ preserve
    extremality if $f$ and all breakpoints lie in $\tfrac{1}{q} \Z$. 
    See \sref{s:interpolation}.
    \\
    \vspace{0pt}
    \sage{\sagefunc{restrict_to_finite_group}\allowbreak(oversampling=3)}
    & \CompendiumGraphics{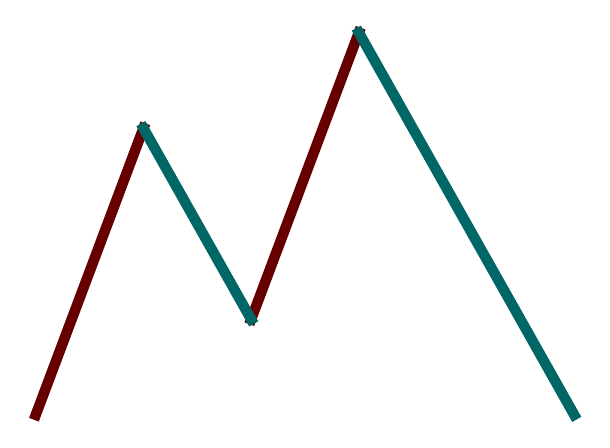}
    & \CompendiumGraphics{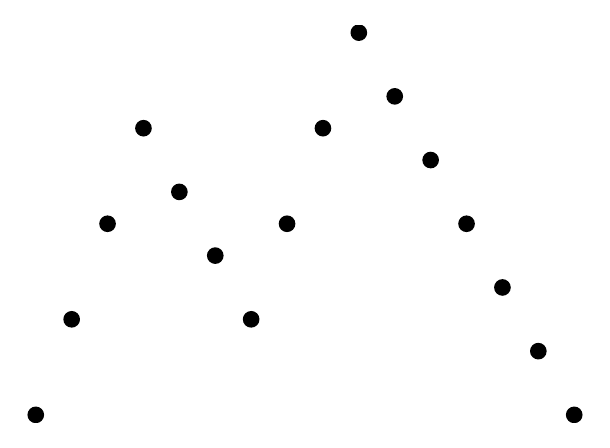}
    & \vspace{0pt}
    If \sage{oversampling} by a factor $m \geq 3$, the restriction is extreme
    for $R_f(\tfrac{1}{mq} \Z, \Z)$ if and
    only if the original function is extreme.
    See \sref{s:interpolation}.
    \\
    \vspace{0pt}
    \sagefunc{interpolate_to_infinite_group}
    & \CompendiumGraphics{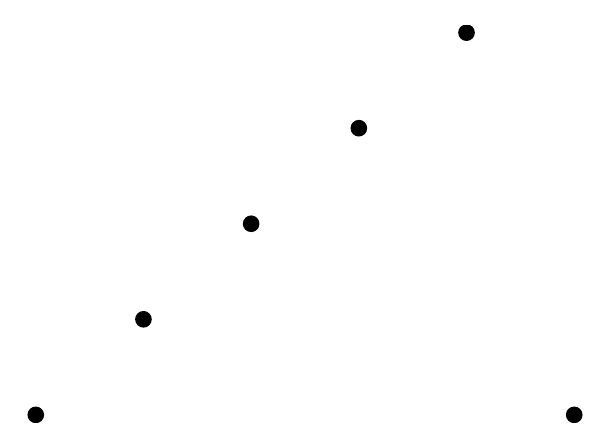}
    & \CompendiumGraphics{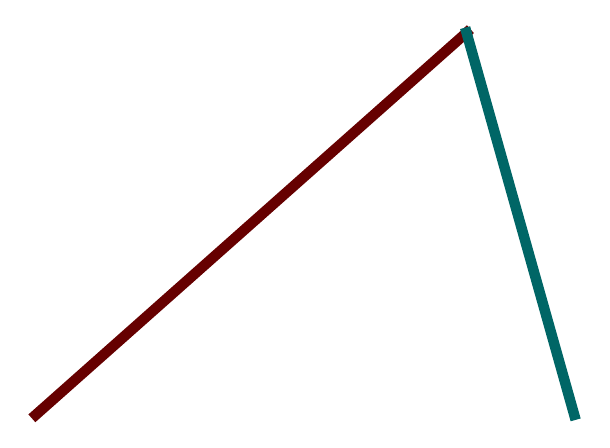}
    & \vspace{0pt}
    Interpolation from finite group problems $R_f(\tfrac{1}{q} \Z, \Z)$
    preserves minimality, but in general not extremality. 
    See \sref{s:interpolation}.
    \\
    \vspace{0pt}
    \sagefunc{two_slope_fill_in}
    & \CompendiumGraphics{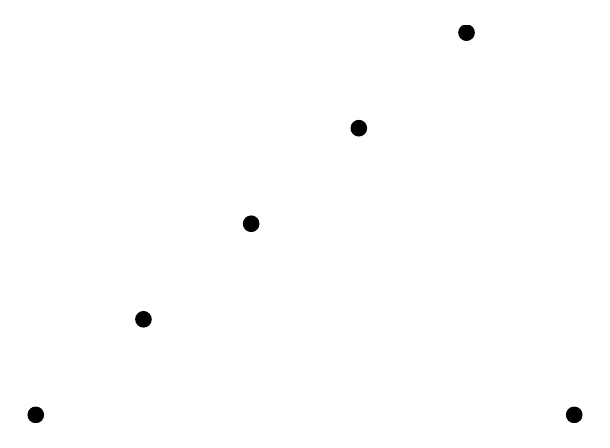}
    & \CompendiumGraphics{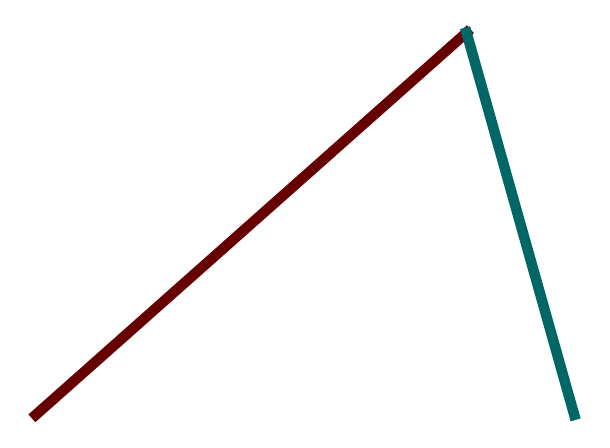}
    & \vspace{0pt}
    Described by Gomory--Johnson \cite{infinite2}, 
    Johnson \cite{johnson}.  For $k=1$, if 
    minimal, equal to
    \sagefunc{interpolate_to_infinite_group} (above).
    For $k>1$, 
    see \cite[section 19.5.2.3]{Richard-Dey-2010:50-year-survey} 
    and
    \cite{basu2013unique,basu-cornuejols-koeppe:unique-minimal-liftings-simplicial}
    for recent developments.
    \\
  \end{CompendiumProceduresTabular}
\end{table}

\clearpage

\begin{landscape}
\begin{table}[p]
  \caption{List of notation in the infinite group problem literature}
  \label{tab:notation}
  \centering
  \begin{minipage}{\linewidth}
    \let\footnoterule=\relax
    \def\arraystretch{1.35}
  \begin{tabular}[t]{>{\RaggedRight}p{.25\linewidth}*{10}c}
    \toprule
    & \multicolumn{2}{c}{Gomory--Johnson}
    & \multicolumn{2}{c}{Dey et al.}
    & \multicolumn{3}{c}{Basu et al.}
    & \multicolumn{3}{c}{Surveys}
    \\
    \cmidrule(lr){2-3}\cmidrule(lr){4-5}\cmidrule(lr){6-8}\cmidrule(lr){9-11}
    Concept 
    & \cite{infinite,infinite2} 
    & \cite{tspace} 
    & \cite{dw2008,dey3}
    & \cite{dey1}
    & \cite{bccz08222222}
    & \cite{bhkm}
    & 
    \cite{basu-hildebrand-koeppe:equivariant,bhk-IPCOext} 
    & \cite{Richard-Dey-2010:50-year-survey}
    & \cite{corner_survey}
    & this 
    \\
    \midrule
    Additive group of reals mod 1 
    & 
    $I$ 
    & 
    $G$
    & 
    $I$
    & 
    $I$
    & 
    \relax
    & 
    \relax
    & 
    \relax
    & 
    $I$
    & 
    \relax
    & 
    \relax
    \\
    Mapping from reals to group elements 
    & 
    $u = \mathcal F(x)$
    & 
    \relax
    & 
    $\mathbb P(u)$
    & 
    $\mathbb P(u)$
    & 
    \relax
    & 
    \relax
    & 
    \relax
    & 
    $u = \mathcal F(v)$
    \\
    Mapping from group elements to canonical reals
    & 
    $x = |u|$
    & 
    \relax
    & 
    \relax
    & 
    $\mathbb P^{-1}(u)$
    & 
    \relax
    & 
    \relax
    & 
    \relax
    & 
    $v = \mathcal F^{-1}(u)$
    \\
    \midrule
    Number of rows of the group problem
    & 
    $1$
    & 
    $1$
    & 
    $m$
    & 
    $m$
    & 
    $1$
    & 
    $k$
    & 
    $k$
    & 
    $m$
    & 
    $q$
    & 
    $k$
    \\
    Group (domain of solutions, valid functions)
    & 
    $U = I$ 
    & 
    $G$
    & 
    $I^m$
    & 
    $I^m$
    & 
    \relax
    & 
    $\R^k$
    & 
    $G = \R^k$
    & 
    $G = I^m$
    & 
    $\R^q$
    & 
    $G = \R^k$
    \\
    Subgroup (periodicity)
    & 
    \relax
    & 
    \relax
    & 
    \relax
    & 
    \relax
    & 
    \relax
    & 
    $\Z^k$
    & 
    $S = \Z^k$
    & 
    \relax
    & 
    $\Z^q$
    & 
    $S = \Z^k$
    \\
    Right-hand side 
    & 
    $u_0$
    & 
    $u_0$
    & 
    $r$
    & 
    $r$
    & 
    $f$
    & 
    $-f$
    & 
    $\ve{f}$
    & 
    $r$
    & 
    $-f$
    & 
    $\ve{f}$
    \\
    Group problem 
    & 
    $\mathrm P(U, u_0)$
    & 
    \relax
    & 
    ($m$DIIGP)
    & 
    \relax
    & 
    & 
    (IR)
    & 
    \relax
    & 
    \relax
    & 
    & 
    \eqref{GP}
    \\
    Solutions to the group problem
    & 
    $t(u)$
    & 
    $\{t(u)\}$
    & 
    $t(u)$
    & 
    $t(u)$
    & 
    $s_r$
    & 
    $s_r$
    & 
    $s(\ve r)$
    & 
    $t(u)$
    & 
    $x_r$
    & 
    $y(\ve r)$
    \\
    Solution set of the group problem
    & 
    $T(U, u_0)$
    & 
    \relax
    & 
    $m$DIIGP
    & 
    $PI(r,m)$
    & 
    \relax
    & 
    \relax
    & 
    \relax
    & 
    $\mathrm{MG}(G, \emptyset, r)$
    & 
    $G_f$
    & 
    \\
    Its convex hull
    & 
    \relax
    & 
    \relax
    & 
    \relax
    & 
    \relax
    & 
    \relax
    & 
    \relax
    & 
    $R_{\ve f}(G, S)$
    & 
    \relax
    & 
    \relax
    & 
    $R_{\ve f}(G, S)$
    \\
    Its enclosing space
    & 
    \relax
    & 
    \relax
    & 
    \relax
    & 
    \relax
    & 
    \relax
    & 
    \relax
    & 
    $\mathcal V$
    & 
    \relax
    & 
    \relax
    & 
    $\R^{(G)}$
    \\
    \midrule
    Valid functions
    & 
    $\pi(u)$
    & 
    $\pi(u)$
    & 
    $\phi(u)$
    & 
    $\phi(u)$
    & 
    $\pi(r)$
    & 
    $\pi(r)$
    & 
    $\pi(\ve r)$
    & 
    $\phi(u)$
    & 
    $\pi(r)$
    & 
    $\pi(\ve r)$
    \\
    Set of tight solutions for a valid function
    & 
    \relax
    & 
    $P(\pi)$
    & 
    $P(\phi)$
    & 
    \relax
    & 
    $P(\pi)$
    & 
    $S(\pi)$
    & 
    \relax
    & 
    \relax
    & 
    \relax
    & 
    $P(\pi)$
    \\
    Subadditivity slack
    & 
    $\nabla(u,v)$  
    & 
    \relax
    & 
    \relax
    & 
    \relax
    & 
    \relax
    & 
    \relax
    & 
    $\Delta\pi(\ve u, \ve v)$
    & 
    \relax
    & 
    \relax
    & 
    $\Delta\pi(\ve{u}, \ve{v})$
    \\
    Additivity domain (equality set)
    & 
    \relax
    & 
    $E(\pi)$
    & 
    $E(\phi)$
    & 
    \relax
    & 
    $E(\pi)$ 
    & 
    $E(\pi)$ 
    & 
    $E(\pi)$
    & 
    \relax
    & 
    \relax
    & 
    $E(\pi)$
    \\
    \bottomrule
  \end{tabular}
\end{minipage}
\end{table}
\end{landscape}


\section{List of notation in the literature}\label{s:appendix-notation}

Table~\ref{tab:notation} (on page~\pageref{tab:notation}) compares the
notation in the present survey with that in selected original articles on the
infinite group problem and the surveys
\cite{Richard-Dey-2010:50-year-survey,corner_survey}. 

\section*{Acknowledgments}
Thanks go to Yuan Zhou for compiling the electronic compendium
of extreme functions in \cite{electronic-compendium}, 
and Chun Yu Hong and Yuan Zhou for their work on the software \cite{infinite-group-relaxation-code}.

The authors gratefully acknowledge partial support from the National Science
Foundation through grants 
DMS-0914873 (R.~Hildebrand, M.~K\"oppe)
and DMS-1320051 (M.~K\"oppe).

\providecommand\ISBN{ISBN }
\bibliographystyle{../amsabbrvurl}
\bibliography{../bib/MLFCB_bib}



\end{document}